\documentclass[12pt]{amsart}
\usepackage[english]{babel}
\usepackage{amsmath}
\usepackage{amsthm}
\usepackage{mathtools}
\usepackage{mathrsfs}
\usepackage{amssymb}
\usepackage{xcolor}
\usepackage{xfrac}
\usepackage{esint}
\usepackage{graphicx}
\usepackage{float}
\usepackage{array}
\usepackage{enumitem}
\usepackage[new]{old-arrows}
\usepackage[all,cmtip]{xy}
\usepackage{tikz-cd}
\usepackage{centernot}
\usepackage{stmaryrd}
\usepackage{comment}
\excludecomment{confidential}

\makeatletter
\@namedef{subjclassname@2020}{\textup{2020} Mathematics Subject Classification}
\makeatother

\numberwithin{equation}{section}

\newtheorem{theorem}{Theorem}[section]
\newtheorem{corollary}[theorem]{Corollary}
\newtheorem{lemma}[theorem]{Lemma}
\newtheorem{proposition}[theorem]{Proposition}
\newtheorem{definition}[theorem]{Definition}

\DeclareMathOperator{\Alg}{Alg}

\DeclareMathOperator{\alg}{alg}

\newcommand{\N}{\mathbb{N}}
\newcommand{\Z}{\mathbb{Z}}

\newcommand{\R}{\mathbb{R}}

\def\a{\alpha}
\def\b{\beta}
\def\e{\varepsilon}
\def\D{\Delta}
\def\d{\delta}

\def\G{\Gamma}
\def\l{\lambda}

\def\L{\Lambda}
\def\o{\omega}

\def\r{\rho}

\def\s{\sigma}

\def\v{\varphi}

\DeclareMathOperator{\spann}{span}
\newcommand{\mc}{\mathcal}
\newcommand{\mf}{\mathfrak}

\DeclareMathOperator{\sign}{sign}

\begin{document}

\title[Periodic solutions to nonlocal pseudo-differential equations.]{Periodic solutions to nonlocal pseudo-differential equations. A bifurcation theoretical perspective.}

\author{Juan Carlos Sampedro} \thanks{The author has been supported by the Research Grant PID2021--123343NB-I00 of the Spanish Ministry of Science and Innovation and by the Institute of Interdisciplinar Mathematics of Complutense University.}
\address{Institute of Interdisciplinary Mathematics \\
	Departamento de Matem\'atica Aplicada a la Ingenier\'ia Industrial\\ 
	Universidad Polit\'ecnica de Madrid\\ 
	28012-Madrid \\
	Spain}
\email{juancarlos.sampedro@upm.es}

\begin{abstract}
In this paper we use abstract bifurcation theory for Fredholm operators of index zero to deal with periodic even solutions of the one-dimensional equation $\mc{L}u=\l u+|u|^{p}$, where $\mc{L}$ is a nonlocal pseudodifferential operator defined as a Fourier multiplier and $\l$ is the bifurcation parameter. Our general setting includes the fractional Laplacian $\mc{L}\equiv(-\D)^{s}$ and sharpens the results obtained for this operator to date. As a direct application, we establish the existence of traveling waves for general nonlocal dispersive equations for some velocity ranges.
\end{abstract}


\maketitle

\section{Introduction}

In this paper we study the existence of even solutions for the periodic nonlocal pseudo-differential equation
\begin{equation}
	\label{e}
	\mc{L}u= \l u +  |u|^{p}, \quad x\in \mathbb{T},
\end{equation}
where $\mc{L}$ is a pseudo-differential operator defined as a Fourier multiplier operator by
\begin{equation*}
\mc{L}u : = \sum_{n\in\Z} |n|^{2s} \mf{m}(n) \widehat{u}(n) e^{inx}, \quad x\in \mathbb{T},
\end{equation*}
where the multiplier $\mf{m}: \Z \to \R_{\geq 0}$ satisfies certain reasonable hypotheses and $\mathbb{T}\equiv \R / 2\pi\Z$ is the $1$-dimensional torus. Throughout this article, for technical reasons, we assume that $s\geq\frac{1}{2}$ and $p\geq 2$. For the results specific to the fractional Laplacian, we further assume that $s<1$. Finally, we consider $\l\in\R$ to be the bifurcation parameter. For $p=2$, this type of equation naturally arises from the study of traveling waves of the general nonlocal dispersive model
\begin{equation}
	\label{E1}
\partial_{t}u+2u\partial_{x}u-\partial_{x}(\mc{T}D^{2s}u)=0, \;\; u=u(t,x), \quad t>0, \;\; x\in \mathbb{T},
\end{equation}
where $0<s\leq 2$, $\mc{T}$ is a pseudo-differential operator defined as a Fourier multiplier operator by
\begin{equation}
\mc{T}u : = \sum_{n\in\Z} \mf{m}(n) \widehat{u}(n) e^{inx}, \quad x\in \mathbb{T},
\end{equation}
and $D^{\alpha}u = (-\D)^{\frac{\a}{2}}u$ is the Riesz potential of negative order, whose Fourier coefficients are given by $\widehat{D^{\a}u}(n) = |n|^{\a}\widehat{u}(n)$ for each $n\in \Z$. As usual, the Fourier coefficients are taken in the spacial variable.
Specific choices of the multiplier $\mf{m}(n)$ and the parameter $s$ give rise to particular models of well-known dispersive equations. For instance, if $\mf{m}(n)=1$, $n\in\Z$ and $s=1$, we obtain the Korteweg--de Vries equation. If $\mf{m}(n)=1$, $n\in\Z$ and $s=\frac{1}{2}$, we obtain the Benjamin--Ono equation and if $s=\frac{1}{2}$ and
	$$\mf{m}(n)=\left\{\begin{array}{ll}\big|\coth(\delta n)-\frac{1}{\delta n}\big| & \hbox{if} \; n\neq 0, \\
		0 & \hbox{if} \; n=0,
	\end{array}
	\right.
	$$ 
	where $\d>0$, we obtain the Intermediate Long Wave equation with depth $\d$. Given a speed $c\in\R$, the periodic traveling waves $u(x,t)=\varphi(x-ct)$ of \eqref{E1} satisfy the following nonlocal pseudo-differential equation:
$$
\mc{L}\varphi= -c\varphi +\varphi^{2}+\mu, \quad y\in \mathbb{T},
$$
where $\mu$ is the integration constant. Therefore, periodic solutions of \eqref{e} with $p=2$, give rise to traveling waves of speed $c=-\l$ and $\mu=0$.
\par In recent years, the study of nonlocal equations has received a great deal of attention. The nonlocal operator that has been most studied is the so-called fractional Laplacian given by
$$
(-\D)^{s}u(x) := C(s)\int_{\R}\frac{u(x)-u(y)}{|x-y|^{1+2s}}\; dy, \;\; x\in\R, \quad C(s):=\frac{s4^{s}\Gamma(\tfrac{1}{2}+s)}{\sqrt{\pi}\Gamma(1-s)},
$$
where $s\in (0,1)$ and where the integral must be understood in the principal value sense. For periodic functions, the fractional Laplacian can be expressed in the Fourier side as
$$
(-\D)^{s}u(x) = \sum_{n\in\Z}|n|^{2s}\widehat{u}(n)e^{inx}, \quad x\in\mathbb{T}.
$$
Equations of the form $(-\D)^{s}u=f(u)$ have been the subject of numerous studies under different boundary conditions. See for instance \cite{AV, DV, SV, SV2, RO, RO2} and references therein. In this paper, we focus on periodic boundary conditions.

Since the well-known local extension introduced by Caffarelli and Silvestre \cite{CS1}, these type of problems have been explored by rewriting the nonlocal problem in terms of a local degenerate elliptic problem with a Neumann boundary condition in one dimension higher. For example, the works of Ambrosio et al. \cite{A1,A2,A3,A4} and DelaTorre, del Pino, González, and Wei \cite{DPGW} use this approach to establish the existence of solutions for fractional problems. Specifically, within the periodic one-dimensional setting, Ambrosio \cite{A3} proved that if $s<\tfrac{1}{2}$, and $f(x,u)$ is a continuous function, $2\pi$-periodic in $x$, and satisfies $tf(x,t)\geq 0$ for all $x,t\in\R$, along with the Ambrosetti–Rabinowitz conditions and polynomial growth at rate $p\in(1,\tfrac{1+2s}{1-2s})$, then the problem
$$
(-\Delta)^{s}u=f(x,u), \quad x\in\mathbb{T},
$$
admits a periodic nontrivial solution $u\in \mc{C}^{\alpha}(\mathbb{T})$.
It is essential to recognize that equation \eqref{e} does not in general admit the extension formulation available for $\mc{L}\equiv (-\Delta)^{s}$. Furthermore, even though we set $\mc{L}\equiv (-\Delta)^{s}$ in \eqref{e}, the function $f(x,t)=\lambda t+ |t|^{p}$ fails to meet the sign condition $tf(x,t)\geq 0$ for all $x,t\in\R$. 
\par In recent years, the problem has been approached from various angles. For instance, in 2018, Barrios, García-Melián, and Quaas \cite{BGMQ} studied, among other issues, the periodic solutions of the equation
\begin{equation}
	\label{E0.1}
	(-\Delta)^{s}u = \lambda u + |u|^{p-1}u, \quad x \in \mathbb{T},
\end{equation}
using variational methods. They established that for $p>1$ if $s\geq \tfrac{1}{2}$, and for $1<p<\tfrac{1+2s}{1-2s}$ if $s<\tfrac{1}{2}$, there exists $\lambda_{\ast}\geq 0$ such that, if $\lambda < -\lambda_{\ast}$ (respectively, $\lambda>0$), equation \eqref{E0.1} admits at least one periodic non-constant positive (respectively, sign-changing) solution. The variational analysis conducted for equation \eqref{E0.1} does not directly yield our parameter ranges for non-constant solutions of \eqref{e}, as the critical points could be constant. The same authors also study $(-\Delta)^s u=u+|u|^p$ in \cite[Section 7, Corollary 1]{BGMQ}, obtaining periodic solutions of small amplitude and a priori bounds. Continuing with the variational perspective, it is also remarkable the work of Cabré, Csató and Mas \cite{CCM} where they prove various qualitative and structural results for the non-constant periodic constrained minimizers of semilinear elliptic equations for integro-differential operators in $\R$. Additionally, they surprisingly provide a form of strong maximum principle for periodic solutions to integro-differential equations. Finally, Bruell and Dhara \cite{BD} prove the existence of periodic solutions of a type of nonlocal equation with homogeneous symbol of order $-r$, where $r>1$ using analytic bifurcation theory.
\par In this paper, we prove the existence of non-constant periodic solutions to equation \eqref{e} by employing a bifurcation theoretical perspective. Our first main result is stated as follows:

\begin{theorem}
	\label{Th1.1}
	Suppose that $s\geq \tfrac{1}{2}$, $2\leq p<4s+1$ and the multiplier $\mf{m}$ satisfies {\rm(M1)--(M3)} below. Then, the pseudo-differential equation
	\begin{equation}
		\label{Eq111}
	\mc{L}u = \l u + |u|^p, \quad x\in\mathbb{T},
\end{equation}
	admits at least one non-constant even solution for every $\l\in (\mf{m}(1),2^{2s}\mf{m}(2))$. Moreover, if $p=2$, equation \eqref{Eq111} admits one non-constant even solution for every 
	$$\l\in (-2^{2s}\mf{m}(2),-\mf{m}(1))\cup(\mf{m}(1),2^{2s}\mf{m}(2)).$$
\end{theorem}
This result implies the existence of non-constant periodic traveling waves with speed
$$
c\in (-2^{2s}\mf{m}(2),-\mf{m}(1))\cup(\mf{m}(1),2^{2s}\mf{m}(2))
$$ 
for the nonlocal dispersive equation \eqref{E1} provided that $s\geq \tfrac{1}{2}$.
Moreover, if $\mf{m}(n) = 1$ for each $n\in\Z$, that is, if $\mc{L}\equiv (-\D)^{s}$, we can establish the existence of non-constant even solutions for a wider range of $\l$:

\begin{theorem}
	\label{TP01}
	Suppose that $\tfrac{1}{2}\leq s<1$ and $p\geq 2$. Then, the pseudo-differential equation
	\begin{equation}
		\label{EE01}
		(-\D)^{s}u = \l u + |u|^{p}, \quad x\in\mathbb{T},
	\end{equation}
	admits at least one non-constant even solution for each
	$$
	\l\in \bigcup_{k\in\N}\left(k^{2s}, (k+1)^{2s}\right).
	$$
\end{theorem}

Moreover, if $p=2$, we can refine Theorems \ref{Th1.1} and \ref{TP01} to obtain the following.

\begin{theorem}
	\label{Th1.2}
	Suppose that $\tfrac{1}{2}\leq s<1$. The following statements hold regarding the equation
	\begin{equation}
		\label{1.4}
		(-\D)^{s}u = \l u + u^{2}, \quad x\in\mathbb{T}.
	\end{equation}
	
	\begin{itemize} 
		\item[{\rm(a)}] Equation \eqref{1.4} admits at least one non-constant even solution for each $\l > 1$. 
		\item[{\rm(b)}] Equation \eqref{1.4}
		admits at least one non-constant strictly positive even solution for each $\l < -1$.
	\end{itemize}
\end{theorem}
In fact, this result is optimal in the case $s=\tfrac{1}{2}$, that is, for traveling waves of the Benjamin--Ono equation. For this equation, Benjamin \cite{Be} and Amick--Toland \cite{AT} were able to find all the periodic solutions in closed form, see Section \ref{S8} for further details. 

These findings provide explicit parameter ranges for non-constant solutions of fixed period, both for the general multiplier operator $\mathcal{L}$ and for the fractional Laplacian $\mc{L}\equiv (-\Delta)^{s}$, including the special case of $p=2$. As previously noted, these results not only complement and generalize those previously available, but also introduce a new approach to tackle non-local problems through the use of bifurcation theoretical techniques for Fredholm operators. This represents a significant turning point in this field.
\par As we have mentioned, to prove Theorems \ref{Th1.1}, \ref{TP01} and \ref{Th1.2}, we will follow a bifurcation theoretical perspective. More precisely, we will use the Crandall--Rabinowitz Theorem \ref{Cr-Rb} to establish the local existence of solutions, and we will use the global alternative Theorem \ref{TGB} to study the global behavior of the continuum of solutions. This global result was obtained by López-Gómez and Mora-Corral in \cite{LG,LGMC,LGMC1} and sharpened in \cite{LGSM} in the light of the Fitzpatrick, Pejsachowicz and Rabier topological degree \cite{FPRa, FPRb, PR}, a generalization of the Leray--Schauder degree to Fredholm operators of index zero. 
To the best of our knowledge, this is the first time in which the topological degree for Fredholm operators is applied to nonlocal periodic problems of type \eqref{e}.
\par This paper is organized as follows. Section \ref{S2} is dedicated to the exposition of the problem we will address. There, we will also prove some technical preliminaries concerning the operators involved in the analysis of the equation. In Section \ref{S3} we perform an spectral analysis of the linearization of our problem. Section \ref{S4} covers local bifurcation theory. In Section \ref{S5} we study the constant solutions of \eqref{e} and the topological nature of the set consisted on these solutions. Section \ref{S6} is dedicated to the analysis of a priori bounds for solutions of \eqref{e}. For this task, we apply the almost sharp Gagliardo--Nirenberg--Sobolev inequality recently found by Liang and Wang \cite{LW}. In section \ref{S7} we apply the Global Alternative Theorem \ref{TGB} to prove Theorems \ref{Th1.1}, \ref{TP01} and \ref{Th1.2}. The final section \ref{S8} focuses on the study of the Benjamin--Ono equation at the light of our results. We have also included several appendices to recall some notions of bifurcation theory that will be used throughout this paper.
\par Along this article, given a pair $(U,V)$ of real Banach spaces, the space of linear bounded operators $T:U\to V$ is denoted by $L(U,V)$. Naturally, we set $L(U):=L(U,U)$.  We denote by $GL(U,V)$ the space of topological isomorphisms and $GL(U):=GL(U,U)$.
Given $T\in L(U,V)$, we denote by $N[T]$ and $R[T]$, the kernel and the range of $T$, respectively.
Finally, $\Phi_0(U,V)$ stands for the set of Fredholm operators of index zero $T:U\to V$ and $\Phi_0(U):=\Phi_{0}(U,U)$.

\section{Exposition of the problem and Nonlinear operators}\label{S2}

In this section, we define the problem we aim to study in this paper in detail, along with what we mean by a solution to it. Given $s > 0$, along this paper we work with the Sobolev spaces of periodic functions
\begin{align*}
	H^{s}(\mathbb{T}) :  = \left\{u(x)=\sum_{n\in\Z}\widehat{u}(n) e^{inx} \in L^{2}(\mathbb{T}) \, : \, \overline{\widehat{u}(n)}=\widehat{u}(-n), \; \; \sum_{n\in\Z} |n|^{2s} |\widehat{u}(n)|^{2} < +\infty\right\},
\end{align*}
endowed with the norm
$$\|u\|_{H^{s}}:=\sqrt{\|u\|_{L^{2}}^{2}+\|u\|_{\dot{H}^{s}}^{2}}, \quad \|u\|_{\dot{H}^{s}}:=\left(\sum_{n\in\Z}|n|^{2s}|\widehat{u}(n)|^{2}\right)^{\frac{1}{2}}.$$
Recall that $H^{s}(\mathbb{T})$ is a Hilbert space with inner product
$$
(u,v)_{H^{s}}:=(u,v)_{L^{2}}+(u,v)_{\dot{H}^{s}}, \quad (u,v)_{\dot{H}^{s}}:= \sum_{n\in\Z}|n|^{2s} \widehat{u}(n) \overline{\widehat{v}(n)}.
$$
As we are interested in even solutions of the equation \eqref{e}, subsequently, for each $s > 0$, we will denote by $H_{+}^{s}(\mathbb{T})$ the subspace of $H^{s}(\mathbb{T})$ consisted on even functions. Note that $H^{s}_{+}(\mathbb{T})$ is a Banach space as it is a closed subspace of $H^{s}(\mathbb{T})$. Analogously, we define $L^{q}_{+}(\mathbb{T})\subset L^{q}(\mathbb{T})$ for each $q\geq 1$ and $\mc{C}^{r}_{+}(\mathbb{T})\subset \mc{C}^{r}(\mathbb{T})$ for each $r\geq 0$.
\par The goal of this paper is the study of the nonlocal pseudo-differential equation
\begin{equation}
	\label{Ee}
\mc{L}u= \l u +  |u|^{p}, \quad x\in \mathbb{T}, \;\; u\in H^{2s}_{+}(\mathbb{T}),
\end{equation}
where $s\geq\frac{1}{2}$, $p\geq 2$, $\mc{L}$ is a pseudo-differential operator defined as a Fourier multiplier operator by
$$
\mc{L}u : = \sum_{n\in\Z} |n|^{2s} \mf{m}(n) \widehat{u}(n) e^{inx}, \quad x\in \mathbb{T},
$$
and the multiplier $\mf{m}: \Z \to \R_{\geq 0}$ satisfies the following hypothesis:
\vspace{5pt}
\begin{itemize}
	\item[(\textbf{M1})] $\mf{m}(-n)=\mf{m}(n)$ for each $n\in\Z$.
	\vspace{2pt}
	\item[(\textbf{M2})] The function $\mf{m}:\N\to \R_{\geq0}$ is non-decreasing.
	\vspace{2pt}
	\item[(\textbf{M3})] There exist $m_{0},m_{1}>0$ such that 
	$m_{0} < \mf{m}(n) < m_{1}$
	for all $n\in\Z\backslash \{0\}$.
\end{itemize}
Clearly the operator $\mc{L}$ maps $H^{2s}(\mathbb{T})$ into $L^{2}(\mathbb{T})$ with operator norm $\|\mc{L}\|\leq \sqrt{2\pi}m_{1}$.
\par A \textit{solution} of \eqref{Ee} is a function $u\in H^{2s}_{+}(\mathbb{T})$ that satisfies \eqref{Ee} pointwise almost everywhere. The solutions of the equation \eqref{Ee} can be rewritten as the zeros of the following nonlinear operator
\begin{equation}
	\label{E3}
	\mf{F}:\R\times H_{+}^{2s}(\mathbb{T})\longrightarrow L_{+}^{2}(\mathbb{T}), \quad \mf{F}(\l,u)=\mc{L}u-\l u - |u|^{p}.
\end{equation}
The following result establishes that $\mf{F}$ is well defined.
\begin{lemma}
	Let $(\l, u)\in \R\times H_{+}^{2s}(\mathbb{T})$. Then, 
	\begin{equation}
		\label{e1}
 \mf{F}(\l,u)=\mc{L}u-\l u - |u|^{p}\in L^{2}_{+}(\mathbb{T}).
\end{equation}
\end{lemma}
\begin{proof}
	First of all, note that for $u\in H^{2s}_{+}(\mathbb{T})$,
	$$
	\|\mc{L}u\|_{L^{2}}^{2}=2\pi\sum_{n\in\Z}|n|^{4s}|\mf{m}(n)|^{2}|\widehat{u}(n)|^{2} \leq 2\pi m_{1}^{2} \sum_{n\in\Z}|n|^{4s}|\widehat{u}(n)|^{2} \leq 2\pi m_{1}^{2}\|u\|_{H^{2s}}^{2},
	$$
	where we have used hypothesis (M3). Then, by a direct bound, we obtain
	\begin{align*}
		\|\mc{L}u-\l u - |u|^{p}\|_{L^{2}}  & \leq \|\mc{L}u\|_{L^{2}}+|\l|\|u\|_{L^{2}}+\||u|^{p}\|_{L^{2}}   \\
		& \leq \sqrt{2\pi} m_{1} \|u\|_{H^{2s}} +  |\l| \|u\|_{L^{2}} + \left(\int_{\mathbb{T}}|u|^{2p}\right)^{\frac{1}{2}}.
	\end{align*}
	Therefore,  the inclusion $\mf{F}(\l,u)\in L^{2}(\mathbb{T})$ holds if $u\in L^{2p}(\mathbb{T})$.
	As $s\geq \tfrac{1}{2}$, by the Sobolev embedding,
	$
	H^{2s}(\mathbb{T})\hookrightarrow\mc{C}(\mathbb{T})
	$,
	we infer that $u\in \mc{C}(\mathbb{T})\hookrightarrow L^{\infty}(\mathbb{T})$. Therefore,
	\begin{equation*}
		\int_{\mathbb{T}}|u|^{2p}\leq 2\pi\|u\|_{L^{\infty}}^{2p}< +\infty.
	\end{equation*}
	This proves that $\mf{F}(\l,u)\in L^{2}(\mathbb{T})$. Finally, we will show that if $u\in H^{2s}(\mathbb{T})$ is even, then $\mf{F}(\l,u)\in L^{2}(\mathbb{T})$ is also even. Suppose that $u\in H^{2s}_{+}(\mathbb{T})$. Then, as $u$ is even, we have
	$$
		\mc{L}u =  2\sum_{n=1}^{+\infty} |n|^{2s}\mf{m}(n) \widehat{u}(n) \cos(nx).
	$$
	Now, consider the sequence of even functions $\{v_{m}\}_{m\in\N}\subset H^{2s}_{+}(\mathbb{T})$ defined by
	$$
	v_{m}(x):=	  2\sum_{n=1}^{m}|n|^{2s}\mf{m}(n) \widehat{u}(n) \cos(nx), \quad x\in\mathbb{T}.
	$$
	Then, $v_{m}\to \mc{L}u$ as $m\to +\infty$ in $L^{2}(\mathbb{T})$. Therefore, there exists a subsequence $\{v_{m_{k}}\}_{k\in\N}$ such that $v_{m_{k}}\to \mc{L}u$ as $k\to +\infty$ almost everywhere. This implies that $\mc{L}u$ is even. Finally,
	\begin{align*}
		[\mf{F}(\l,u)](-x) & = [\mc{L}u](-x)-\l u(-x) - |u(-x)|^{p} \\
		& = [\mc{L}u](x)-\l u(x) - |u(x)|^{p} = [\mf{F}(\l,u)](x),
	\end{align*}
	for almost every $x\in\mathbb{T}$. Then $\mf{F}(\l,u)\in L^{2}_{+}(\mathbb{T})$. This concludes the proof.
\end{proof}
To state the following result, we first define the function 
$$
\o: (1,+\infty) \longrightarrow [0,+\infty], \quad \o(p):= \left\{\begin{array}{ll}
	\lfloor p \rfloor & \hbox{if} \;\; p\notin \N, \\
	p-1 & \hbox{if} \;\; p \in 2\N-1, \\
	+\infty & \hbox{if} \;\; p\in 2\N.
	\end{array}\right.
$$
The following lemma provides $\mf{F}$ with the sufficient regularity for the subsequent arguments. 
\begin{lemma}
	\label{L1}
	The operator $\mf{F}$ is of class $\mc{C}^{\o(p)}$, symbolically, $\mf{F}\in\mc{C}^{\o(p)}(\R\times  H_{+}^{2s}(\mathbb{T}), L_{+}^{2}(\mathbb{T}))$. Moreover, the first three derivatives are given by
	\begin{itemize}
		\item If $p\geq 2$, for all $(\l,u)\in\R\times  H_{+}^{2s}(\mathbb{T})$,
		\begin{align}
			\label{E4}
	& \partial_{u}\mf{F}(\l,u): H_{+}^{2s}(\mathbb{T}) \longrightarrow L_{+}^{2}(\mathbb{T}), \\
	 & \partial_{u}\mf{F}(\l,u)[v]=\mc{L}v-\l v - p u |u|^{p-2} v. \nonumber
      \end{align}
      \item If $p \geq 2$, for all $(\l,u)\in\R\times  H_{+}^{2s}(\mathbb{T})$, 
		\begin{align*}
		& \partial_{uu}^{2}\mf{F}(\l,u):H_{+}^{2s}(\mathbb{T})\times H_{+}^{2s}(\mathbb{T}) \longrightarrow L_{+}^{2}(\mathbb{T}), \\
		& \partial_{uu}^{2}\mf{F}(\l,u)[v_1,v_2] = -p(p-1)|u|^{p-2}v_1 v_2.
	     \end{align*}
		\item If $p \geq 4$, for all $(\l,u)\in\R\times  H_{+}^{2s}(\mathbb{T})$, 
		\begin{align*}
		& \partial_{uuu}^{3}\mf{F}(\l,u):H_{+}^{2s}(\mathbb{T})\times H_{+}^{2s}(\mathbb{T}) \times H_{+}^{2s}(\mathbb{T}) \longrightarrow L_{+}^{2}(\mathbb{T}), \\
		& \partial_{uuu}^{3}\mf{F}(\l,u)[v_1,v_2,v_3] = -p(p-1)(p-2)u|u|^{p-4}v_1 v_2 v_3.
		\end{align*}
	\end{itemize}
        For $p=2$, the factor $|u|^{p-2}$ is understood to be $1$, and $\partial_{uuu}^{3}\mf{F}=0$.

\end{lemma} 

\begin{proof}
	Let us start by proving the continuity of the operator $\mf{F}$. Let $\{(\l_{n},u_{n})\}_{n\in\N}\subset\R\times H_{+}^{2s}(\mathbb{T})$ be a sequence satisfying 
	$$\lim_{n\to+\infty}\l_n = \l_0 \quad \text{ and } \quad
	\lim_{n\to+\infty}u_{n}=u_{0} \quad \text{ in } \; \; H_{+}^{2s}(\mathbb{T}).$$
	Choose $N\in\N$ such that $\|u_{n}-u_{0}\|_{H^{2s}}<1$ for each $n\geq N$. In particular, this implies that
	\begin{equation}
		\label{Eqqq}
	\|u_{n}\|_{H^{2s}}<1+\|u_{0}\|_{H^{2s}}, \quad n\geq N.
	\end{equation}
	We rewrite the difference of the corresponding operator as
	\begin{equation}
		\label{p1}
		\mf{F}(\l_{n},u_{n})-\mf{F}(\l_{0},u_{0})=\mc{L}(u_{n}-u_{0})-(\l_{n}-\l_{0})u_{n}-\l_{0}(u_{n}-u_{0}) - (|u_n|^{p}-|u_0|^{p}).
	\end{equation}
	Fixing $n\geq N$ and taking the $L^{2}$-norm in equation \eqref{p1}, we obtain
	\begin{align}
		\label{E5}
		\|\mf{F}(\l_{n},u_{n})-\mf{F}(\l_{0},u_{0})\|_{L^{2}} 
		\leq C \left( \|u_{n}-u_{0}\|_{H^{2s}}+|\l_{n}-\l_{0}|+\|u_{n}-u_{0}\|_{L^{2}}+\||u_n|^{p}-|u_0|^{p}\|_{L^{2}} \right),
	\end{align}
	where we have used inequality \eqref{Eqqq}. On the other hand, setting 
	$$
	f(x)=|x|^{p}, \quad x\in\R,
	$$
	we have that $f\in \mc{C}^{\o(p)}(\R)\subset \mc{C}^{1}(\R)$, and therefore, by the mean value theorem, for each $n\geq N$ (enlarging $N$ if necessary), we have 
	$$
	\|f(u_n)-f(u_0)\|_{L^{\infty}}\leq M \|u_n-u_0\|_{L^{\infty}}, \quad M:=\max\left\{|f'(x)|  :  x\in [-1-\|u_0\|_{L^{\infty}},1+\|u_0\|_{L^{\infty}}]\right\}.
	$$
	By the Sobolev embedding $H^{2s}(\mathbb{T})\hookrightarrow \mc{C}(\mathbb{T})$, we infer that $\|u_n-u_0\|_{L^{\infty}}\to 0$ as $n\to+\infty$. Then, taking the limit, we obtain
	\begin{align}
		\label{e3}
		\lim_{n\to+\infty}\|f(u_n)-f(u_0)\|_{L^{2}}^{2} & =\lim_{n\to+\infty}\int_{\mathbb{T}}|f(u_n)-f(u_0)|^{2}\leq 2\pi M^{2} \lim_{n\to+\infty}\|u_n-u_0\|^{2}_{L^{\infty}} = 0.
	\end{align}
	Inequalities \eqref{E5} and \eqref{e3} yield
	$$\lim_{n\to+\infty}\|\mf{F}(\l_{n},u_{n})-\mf{F}(\l_{0},u_{0})\|_{L^{2}}=0.$$
	Hence, this implies that $\mf{F}$ is continuous.
	\par Let us prove that $\mf{F}\in\mc{C}^{\o(p)}$. It is enough to prove $\mf{F}\in\mc{C}^{1}$ for $p\geq 2$ as the proof of the higher differentiability relies in the same techniques. We start by proving \eqref{E4}. Let $(\l,u)\in\R\times H^{2s}_{+}(\mathbb{T})$ and $v\in H^{2s}_{+}(\mathbb{T})$. We can rewrite
	\begin{align*}
		\mf{F}(\l,u+v)-\mf{F}(\l,u)-\mathcal{L}v+\l v+ f'(u) v & =-\left(f(u+v)-f(u)-f'(u)v\right) \\
		& = - \left(\int_{0}^{1} (f'(u+tv)-f'(u))\; dt \right) v.
	\end{align*}
 Consequently, a direct bound on the $L^{2}$-norm and the Jensen's inequality yield
	\begin{align*}
		\|f(u+v)-f(u)-f'(u)v \|_{L^{2}} & = \left\| \left(\int_{0}^{1} (f'(u+tv)-f'(u))\; dt \right) v \right\|_{L^{2}} \\
		& \leq \sqrt{2\pi} \|v\|_{L^{\infty}}\left(\int_{0}^{1}\|f'(u+tv)-f'(u)\|_{L^{\infty}}^{2} \; dt\right)^{\frac{1}{2}}
	\end{align*}
	Take $\varepsilon>0$ and suppose that $\|v\|_{L^{\infty}}<\varepsilon$. By the Sobolev embedding $H^{2s}(\mathbb{T})\hookrightarrow L^{\infty}(\mathbb{T})$ and the application of the mean value theorem to $f'$ (note that $f'\in\mc{C}^{\o(p)-1}(\R)\subset \mc{C}^{1}(\R)$ because $p\geq 2$), we deduce that
	\begin{align*}
		\|f(u+v)-f(u)-f'(u)v \|_{L^{2}} &  \leq C_{1} M(v) \|v\|_{H^{2s}}\left(\int_{0}^{1}|t|^{2}\|v\|_{L^{\infty}}^{2} \; dt\right)^{\frac{1}{2}} \leq C_{2} \|v\|_{H^{2s}}^{2},
	\end{align*}
	for some constants $C_{1}, C_{2}>0$ independent of $v$ and where we have denoted
	$$
	M(v):=\max\left\{|f''(x)| \; :  |x|\leq \|u\|_{L^{\infty}}+\|v\|_{L^{\infty}}\right\}\leq \max\left\{|f''(x)| \; :  |x|\leq \|u\|_{L^{\infty}}+\varepsilon\right\}.
	$$
	Then, by the definition of differentiability,
	\begin{align*}
		\lim_{v\to 0}\frac{\|\mf{F}(\l,u+v)-\mf{F}(\l,u)-\mathcal{L}v+\l v + f'(u) v\|_{L^{2}}}{\|v\|_{H^{2s}}} 
		\leq C_{2} \lim_{v\to 0}\|v\|_{H^{2s}}=0.
	\end{align*}
	This proves \eqref{E4}. The continuity of $\partial_{u}\mf{F} : \R\times H^{2s}_{+}(\mathbb{T})\to L(H^{2s}_{+}(\mathbb{T}), L^{2}_{+}(\mathbb{T}))$ is proven in the same way we proved the continuity of $\mf{F}$. Higher derivatives follow similarly from the uniform continuity of the derivatives of $f$ on bounded intervals and the embedding $H^{2s}(\mathbb{T})\hookrightarrow \mc{C}(\mathbb{T})$. This concludes the proof.
\end{proof}

The next is a regularity result for the weak solutions of the corresponding linear pseudo--differential equation involving the Fourier multiplier operator $\mc{L}$. The proof is an adaptation to the nonlocal case of the Friedrichs' theorem, see for instance pages 177--182 of Yosida \cite{Y}. Given $c\in L^{\infty}_{+}(\mathbb{T})$ and $f\in L^{2}_{+}(\mathbb{T})$, recall that a \textit{weak solution} of the equation $\mc{L}u + c(x)u= f$, $x\in\mathbb{T}$, is a function $u\in H^{s}_{+}(\mathbb{T})$ such that
$$
\int_{\mathbb{T}}\mc{L}u \; \varphi +\int_{\mathbb{T}}c(x) u \varphi= \int_{\mathbb{T}}f \varphi, \quad \text{for each} \; \; \varphi\in \mc{C}^{\infty}_{+}(\mathbb{T}),
$$
or equivalently,
$$
2\pi\sum_{n\in\Z}|n|^{2s}\mf{m}(n)\widehat{u}(n)\widehat{\varphi}(n) +\int_{\mathbb{T}}c(x) u \varphi= \int_{\mathbb{T}}f \varphi, \quad \text{for each} \; \; \varphi\in \mc{C}^{\infty}_{+}(\mathbb{T}).
$$

\begin{proposition}
	\label{Reg}
	Consider $c\in L^{\infty}_{+}(\mathbb{T})$ and $f\in L^{2}_{+}(\mathbb{T})$. Let $u\in H_{+}^{s}(\mathbb{T})$ be a weak solution of the problem
	\begin{equation}
		\label{E7.1}
		\mc{L}u+c(x)u=f, \quad x\in\mathbb{T}.
	\end{equation}
	Then, $u\in H_{+}^{2s}(\mathbb{T})$ and 
	\begin{equation}
		\|u\|_{H^{2s}}\leq C (\|f\|_{L^{2}}+\|u\|_{L^{2}}),
	\end{equation}
	for some positive constant $C>0$ independent of $u$.
\end{proposition}

\begin{proof}
	Choose $\mu > \|c\|_{L^{\infty}}$. Let us consider the bilinear form $\mf{a}:H_{+}^{s}(\mathbb{T})\times H_{+}^{s}(\mathbb{T}) \to \R$ defined by
	\begin{equation*}
		\mf{a}(u_{1},u_{2}):=2\pi\sum_{n\in\Z}|n|^{2s}\mf{m}(n)\widehat{u}_{1}(n)\widehat{u}_{2}(n) +  \int_{\mathbb{T}} (\mu+c(x))u_{1}u_{2} .
	\end{equation*}
	The continuity of $\mf{a}$ is easily proven by a direct bound as
	\begin{align*}
		|\mf{a}(u_{1},u_{2})|  \leq 2\pi m_{1}\|u_{1}\|_{\dot{H}^{s}}\|u_{2}\|_{\dot{H}^{s}}+(\mu+\|c\|_{L^{\infty}})\|u_{1}\|_{L^{2}}\|u_{2}\|_{L^{2}} \leq C \|u_{1}\|_{H^{s}}\|u_{2}\|_{H^{s}}.
	\end{align*}
	On the other hand, the bilinear form $\mf{a}$ is coercive. Indeed,
	\begin{align*}
		\mf{a}(u,u) & =2\pi\sum_{n\in\Z}|n|^{2s}\mf{m}(n)|\widehat{u}(n)|^{2} +  \int_{\mathbb{T}} (\mu+c(x))u^{2} \\
		& \geq 2\pi m_{0} \|u\|_{\dot{H}^{s}}^{2}+(\mu-\|c\|_{L^{\infty}})\|u\|_{L^{2}}^{2} \geq C \|u\|^{2}_{H^{s}}.
	\end{align*}
	The application of the Lax--Milgram theorem to $\mf{a}$ implies the existence of a bounded linear isomorphism $T: H^{s}_{+}(\mathbb{T})\to H^{s}_{+}(\mathbb{T})$ such that
	$$
	\mf{a}(T u_{1},u_{2})=(u_{1},u_{2})_{H^{s}}, \quad u_{1},u_{2}\in H^{s}_{+}(\mathbb{T}).
	$$
	Let $\varepsilon>0$. Repeating the argument for the constant coefficient bilinear form $\mf{b}:H_{+}^{s}(\mathbb{T})\times H_{+}^{s}(\mathbb{T}) \to \R$ defined by
	\begin{equation*}
		\mf{b}(u_{1},u_{2}):=2\pi\sum_{n\in\Z}|n|^{2s}\mf{m}(n)\widehat{u}_{1}(n)\widehat{u}_{2}(n)+\varepsilon\int_{\mathbb{T}}u_{1}u_{2},
	\end{equation*}
	we prove the existence of a bounded linear isomorphism $Q: H^{s}_{+}(\mathbb{T})\to H^{s}_{+}(\mathbb{T})$ such that
	$$
	\mf{b}( Q u_{1}, u_{2})=(u_{1},u_{2})_{H^{s}}, \quad u_{1},u_{2}\in H^{s}_{+}(\mathbb{T}).
	$$
	Therefore, we have proved the identity 
	\begin{equation}
		\label{E9}
		\mf{a}(u_{1},u_{2})=\mf{b}(Q T^{-1} u_{1}, u_{2}), \quad u_{1},u_{2}\in H^{s}_{+}(\mathbb{T}).
	\end{equation}
	We proceed to prove that, in fact, $Q T^{-1}: H^{2s}_{+}(\mathbb{T})\to H^{2s}_{+}(\mathbb{T})$. On the one hand, note that for $\e>0$ and each $u_{2}\in\mc{C}^{\infty}_{+}(\mathbb{T})$,
	\begin{equation}
		\label{E10}
		\mf{a}(u_{1},\mc{L}u_{2}+\varepsilon u_{2})=\mf{b}(Q T^{-1} u_{1}, \mc{L}u_{2}+\varepsilon u_{2}), \quad u_{1}\in H^{s}_{+}(\mathbb{T}).
	\end{equation}
	We introduce the bilinear forms $\mf{a}_{1},\mf{b}_{1}:H^{2s}_{+}(\mathbb{T})\times H^{2s}_{+}(\mathbb{T})\to\R$ defined by
	\begin{align*}
		& \mf{a}_{1}(u_{1},u_{2}):=2\pi\sum_{n\in\Z}|n|^{4s}\mf{m}^{2}(n)\widehat{u}_{1}(n)\widehat{u}_{2}(n)+2\pi\varepsilon\sum_{n\in\Z}|n|^{2s}\mf{m}(n)\widehat{u}_{1}(n)\widehat{u}_{2}(n)+\int_{\mathbb{T}}(\mu+c(x))u_{1}(\mc{L}u_{2}+\varepsilon u_{2}), \\
		& \mf{b}_{1}(u_{1},u_{2}):=2\pi\sum_{n\in\Z}|n|^{4s}\mf{m}^{2}(n)\widehat{u}_{1}(n)\widehat{u}_{2}(n)+2\pi\varepsilon\sum_{n\in\Z}|n|^{2s}\mf{m}(n)\widehat{u}_{1}(n)\widehat{u}_{2}(n)+\varepsilon\int_{\mathbb{T}}u_{1}(\mc{L}u_{2}+\varepsilon u_{2}).
	\end{align*}
	It must be observed that if $u_{2}\in\mc{C}^{\infty}_{+}(\mathbb{T})$, then
	\begin{align*}
		\mf{a}_{1}(u_{1},u_{2})=\mf{a}(u_{1},\mc{L}u_{2}+\varepsilon u_{2}), \quad \mf{b}_{1}(u_{1},u_{2})=\mf{b}(u_{1}, \mc{L}u_{2}+\varepsilon u_{2}).
	\end{align*}
	By a similar analysis to the one we have performed earlier, we can prove, modifying $\e$ if necessary, that $\mf{a}_{1}$ and $\mf{b}_{1}$ are continuous and coercive. Therefore, applying the Lax--Milgram theorem to the bilinear forms $\mf{a}_{1}$ and $\mf{b}_{1}$, we infer the existence of bounded linear isomorphisms $T_{1}, Q_{1} : H^{2s}_{+}(\mathbb{T})\to H^{2s}_{+}(\mathbb{T})$ such that
	$$
	\mf{a}(T_{1} u_{1},\mc{L}u_{2}+\varepsilon u_{2})=(u_{1},u_{2})_{H^{2s}}, \quad \mf{b}(Q_{1} u_{1}, \mc{L}u_{2}+\varepsilon u_{2})=(u_{1},u_{2})_{H^{2s}},
	$$
	for all $u_{1}\in H^{2s}_{+}(\mathbb{T})$ and $u_{2}\in\mc{C}^{\infty}_{+}(\mathbb{T})$. Therefore, we have
	$$
	\mf{a}(u_{1},\mc{L}u_{2}+\varepsilon u_{2})=\mf{b}(Q_{1}T_{1}^{-1}u_{1}, \mc{L}u_{2}+\varepsilon u_{2}), \quad u_{1}\in H^{2s}_{+}(\mathbb{T}), \; u_{2}\in\mc{C}^{\infty}_{+}(\mathbb{T}).
	$$
	Then, by the identity \eqref{E10} we deduce that for each $u_{1}\in H^{2s}_{+}(\mathbb{T})$,
	$$
	\mf{b}([QT^{-1}-Q_{1}T_{1}^{-1}]u_{1}, \mc{L}u_{2}+\varepsilon u_{2})=0, \quad u_{2}\in\mc{C}^{\infty}_{+}(\mathbb{T}).
	$$
	Denoting $v := [QT^{-1}-Q_{1}T_{1}^{-1}]u_{1}$ and expanding $v, u_{2}$ in Fourier series, we deduce that
	\begin{equation*}
		\sum_{n\in \Z} (|n|^{2s}\mf{m}(n)+\varepsilon)^2 \widehat{v}(n)\widehat{u}_{2}(n)=0.
	\end{equation*}
	Choosing for each $k\in\Z$, the test function $u_{2}(x)=\cos(kx)$, we deduce that $\widehat{v}(k)=0$ for all $k\in\Z$. Hence $v\equiv 0$ and this implies that $QT^{-1}u = Q_{1}T_{1}^{-1}u$ for each $u\in H^{2s}_{+}(\mathbb{T})$. 
	
	Finally, we prove that every weak solution $u\in H^{s}_{+}(\mathbb{T})$ of \eqref{E7.1} lives in $H^{2s}_{+}(\mathbb{T})$. As $u$ is a weak solution of \eqref{E7.1}, then $\mf{a}(u,v)=\int_{\mathbb{T}}fv+\mu\int_{\mathbb{T}}uv$, for all $v\in H^{s}_{+}(\mathbb{T})$ and hence, we deduce by identity \eqref{E9} that
	\begin{equation*}
		\mf{b}(Q T^{-1} u, v) = \int_{\mathbb{T}}(f+\mu u) v, \quad v\in H^{s}_{+}(\mathbb{T}).
	\end{equation*}
	Expanding the functions $w := Q T^{-1}u, v$ and $f$ in Fourier series, we deduce that
	\begin{align*}
		\sum_{n\in\Z} (|n|^{2s}\mf{m}(n)+\varepsilon) \widehat{w}(n) \widehat{v}(n) = \sum_{n\in\Z}(\widehat{f}(n)+\mu\widehat{u}(n)) \widehat{v}(n).
	\end{align*}
	Hence, necessarily
	$$\widehat{w}(n)=\frac{\widehat{f}(n)+\mu \widehat{u}(n)}{|n|^{2s}\mf{m}(n)+\varepsilon}, \quad n\in\Z.$$
	As $f, u\in L^{2}_{+}(\mathbb{T})$, this implies that $QT^{-1}u\in H^{2s}_{+}(\mathbb{T})$. Indeed,
	\begin{align*}
	\|QT^{-1}u\|_{\dot{H}^{2s}}^{2} & =\sum_{n\in \Z} |n|^{4s} |\widehat{w}(n)|^{2} = \sum_{n\in\Z} |n|^{4s} \frac{|\widehat{f}(n)+\mu \widehat{u}(n)|^{2}}{(|n|^{2s}\mf{m}(n)+\varepsilon)^{2}} \\
	&  \leq  C \sum_{n\in\Z}|\widehat{f}(n)+\mu \widehat{u}(n)|^{2} = \frac{C}{2\pi}  \|f+\mu u\|^{2}_{L^{2}}<+\infty.
	\end{align*}
	Finally, as $QT^{-1}: H^{2s}_{+}(\mathbb{T}) \to H^{2s}_{+}(\mathbb{T})$ is an isomorphism, we conclude that $u\in H^{2s}_{+}(\mathbb{T})$ and 
	$$\|u\|_{H^{2s}}= \|TQ^{-1}QT^{-1} u \|_{H^{2s}}\leq C_{1}\|QT^{-1}u\|_{H^{2s}}\leq C_{2} \|f+\mu u\|_{L^{2}}\leq C_{3} (\|f\|_{L^{2}}+\|u\|_{L^{2}}),$$
	for some positive constants $C_{1}, C_{2}, C_{3}>0$. This concludes the proof.
\end{proof}

The next result is of capital importance in order to apply bifurcation theory for Fredholm operators. It states that the linearization of $\mf{F}$ with respect to the variable $u$ at every point is a Fredholm operator of index zero.

\begin{proposition}
	\label{Eqq}
	For each $(\l,u)\in\R\times H_{+}^{2s}(\mathbb{T})$, the linear operator
	\begin{equation}
		\label{E7}
		\partial_{u}\mf{F}(\l,u): H_{+}^{2s}(\mathbb{T}) \longrightarrow L_{+}^{2}(\mathbb{T}), \quad \partial_{u}\mf{F}(\l,u)[v]=\mc{L}v-\l v - pu |u|^{p-2} v,
	\end{equation}
	is Fredholm of index zero. Symbolically, $\partial_{u}\mf{F}(\l,u)\in \Phi_{0}(H_{+}^{2s}(\mathbb{T}), L^{2}_{+}(\mathbb{T}))$.
\end{proposition}

\begin{proof}
	Fix $(\l,u)\in \R\times H_{+}^{2s}(\mathbb{T})$ and take
	$$
	\mu>p\|u\|_{L^{\infty}}^{p-1}+\l.
	$$
	Let us consider the bilinear form $\mf{a}:H_{+}^{s}(\mathbb{T})\times H_{+}^{s}(\mathbb{T}) \to \R$ defined by
	\begin{equation*}
		\mf{a}(v_{1},v_{2}):=2\pi\sum_{n\in\Z}|n|^{2s}\mf{m}(n)\widehat{v}_{1}(n)\widehat{v}_{2}(n) - (\l-\mu)\int_{\mathbb{T}}v_{1}v_{2} - p\int_{\mathbb{T}} u|u|^{p-2}v_{1}v_{2}.
	\end{equation*}
	The continuity of $\mf{a}$ is easily proven as
	\begin{align*}
		|\mf{a}(v_{1},v_{2})| & \leq 2\pi m_{1}\|v_{1}\|_{\dot{H}^{s}}\|v_{2}\|_{\dot{H}^{s}}+|\l-\mu|\|v_{1}\|_{L^{2}}\|v_{2}\|_{L^{2}}+p \|u\|_{L^{\infty}}^{p-1}\|v_{1}\|_{L^{2}}\|v_{2}\|_{L^{2}} \\
		& \leq C\|v_{1}\|_{H^{s}}\|v_{2}\|_{H^{s}}.
	\end{align*}
	On the other hand, the bilinear form $\mf{a}$ is coercive. Indeed,
	\begin{align*}
		\mf{a}(v,v) & =2\pi\sum_{n\in\Z}|n|^{2s}\mf{m}(n)|\widehat{v}(n)|^{2} -(\l-\mu)\int_{\mathbb{T}}v^{2} - p \int_{\mathbb{T}}  u |u|^{p-2} v^{2} \\
		& \geq 2\pi m_{0} \|v\|_{\dot{H}^{s}}^{2}+(\mu-\l-p\|u\|_{L^{\infty}}^{p-1})\|v\|_{L^{2}}^{2} \geq C \|v\|^{2}_{H^{s}}.
	\end{align*}
	The application of the Lax--Milgram theorem to the bilinear form $\mf{a}$ implies that for each $f\in L^{2}_{+}(\mathbb{T})$, there exists a unique $v\in H^{s}_{+}(\mathbb{T})$ such that
	\begin{equation*}
		2\pi\sum_{n\in\Z}|n|^{2s}\mf{m}(n)\widehat{v}(n)\widehat{w}(n) - (\l-\mu)\int_{\mathbb{T}}v w - p\int_{\mathbb{T}} u |u|^{p-2} v w = \int_{\mathbb{T}}fw,
	\end{equation*}
	for all $w\in H^{s}_{+}(\mathbb{T})$. Hence, $v$ is the unique weak solution of the equation
	\begin{equation}
		\label{E8}
		\mc{L}v+(\mu-\l- p u |u|^{p-2}) v = f.
	\end{equation}
	By the regularity result stated in Proposition \ref{Reg}, we infer that $v\in H^{2s}_{+}(\mathbb{T})$. This proves that for each $(\l,u)\in \R\times H^{2s}_{+}(\mathbb{T})$,  the operator 
	\begin{equation}
		\partial_{u}\mf{F}(\l,u)+\mu J: H_{+}^{2s}(\mathbb{T}) \longrightarrow L_{+}^{2}(\mathbb{T}), \quad v\mapsto \mc{L}v-(\l-\mu) v- p u |u|^{p-2} v,
	\end{equation}
	is an isomorphism, where $J: H^{2s}_{+}(\mathbb{T})\hookrightarrow L^{2}_{+}(\mathbb{T})$ is the canonical embedding. Therefore, we can rewrite
	$$
	\partial_{u}\mf{F}(\l,u)= (\partial_{u}\mf{F}(\l,u)+\mu J)-\mu J.
	$$
	As the embedding $J$ is compact, $\partial_{u}\mf{F}(\l,u)$ is the sum of an invertible and a compact operator. Therefore, by \cite[Chap. XV, Th. 4.1]{GGS}, the operator $\partial_{u}\mf{F}(\l,u)$ is Fredholm of index zero.
\end{proof}

The final result of this section establishes that $\mf{F}$ is proper on closed and bounded subsets. This property is important in order to recover some compactness needed for the application of the global bifurcation Theorem \ref{TGB}. Recall that a map $f:X\to Y$ between two topological spaces $X, Y$ is \textit{proper} if the preimage of every compact set in $Y$ is compact in $X$.

\begin{proposition}
	\label{LF4}
	The operator $\mf{F}$ is proper on closed and bounded subsets of $\R\times H^{2s}_{+}(\mathbb{T})$.
\end{proposition}

\begin{proof}
	It suffices to prove that the restriction of $\mf{F}$ to the closed subset
	${K}:=[\lambda_{-},\lambda_{+}]\times \bar {B}_{R}$ is proper, where $\lambda_{-}<\lambda_{+}$ and $B_{R}$ stands for the open ball of $H^{2s}_{+}(\mathbb{T})$ of radius $R>0$ centered at $0$. It suffices to show that every sequence in $K$ whose image under $\mf{F}$ converges has a subsequence converging in $\R\times H^{2s}_{+}(\mathbb{T})$. We prove this below. In particular, the image is closed and the fibers are compact.
	\par
	To show that $\mf{F}(K)$ is closed in $L^{2}_{+}(\mathbb{T})$, let $\{\phi_{n}\}_{n\in\mathbb{N}}$ be a sequence in $\mf{F}(K)\subset L^{2}_{+}(\mathbb{T})$ such that
	\begin{equation}
		\label{7.2}
		\lim_{n\to+\infty} \phi_{n}=\phi \quad \text{in } L^{2}_{+}(\mathbb{T}).
	\end{equation}
	Choose any sequence $\{(\l_{n},u_{n})\}_{n\in\mathbb{N}}$ in $K$ such that \begin{equation}
		\label{7.3}
		\phi_{n}=\mf{F}(\l_{n},u_{n}) \quad \hbox{for all} \;\; n\in\mathbb{N}.
	\end{equation}
	By the compactness of the embeddings $H^{2s}_{+}(\mathbb{T})\hookrightarrow H^{s}_{+}(\mathbb{T})$ and $H^{2s}_{+}(\mathbb{T})\hookrightarrow\mc{C}(\mathbb{T})$ (note that $s\geq \tfrac{1}{2})$, we can extract a subsequence $\{(\l_{n_{k}},u_{n_{k}})\}_{k\in\mathbb{N}}$ such that, for some $(\l_{0},u_{0})\in [\l_-,\l_+]\times \mc{C}(\mathbb{T})$,  it holds $\lim_{k\to +\infty} \l_{n_{k}} = \l_{0}$ and
	\begin{equation}
		\label{7.4}
		\lim_{k\to+\infty} u_{n_{k}} =u_{0} \quad \text{in } \mc{C}(\mathbb{T}) \text{ and in }  H^{s}_{+}(\mathbb{T}).
	\end{equation}
	By \eqref{7.3}, we have that $u_{n_k}$ is a weak solution of $\mc{L}u-\l_{n_k} u-|u|^{p}=\phi_{n_k}$. That is, 
	$$
	2\pi\sum_{n\in\Z}|n|^{2s}\mf{m}(n)\widehat{u}_{n_k}(n)\widehat{\varphi}(n) -\l_{n_k}\int_{\mathbb{T}}u_{n_k}\varphi -\int_{\mathbb{T}}|u_{n_k}|^{p}\varphi=\int_{\mathbb{T}} \phi_{n_k} \varphi, \quad \text{for all} \;\; \varphi\in\mc{C}^{\infty}_{+}(\mathbb{T}). 
	$$
	Then, by \eqref{7.4}, taking $k\to+\infty$ we obtain
	$$
	2\pi\sum_{n\in\Z}|n|^{2s}\mf{m}(n)\widehat{u}_{0}(n)\widehat{\varphi}(n) -\l_0\int_{\mathbb{T}}u_0\varphi -\int_{\mathbb{T}}|u_0|^{p}\varphi=\int_{\mathbb{T}} \phi \varphi, \quad \text{for all} \;\; \varphi\in\mc{C}^{\infty}_{+}(\mathbb{T}). 
	$$
	Therefore $u_{0}$ must be a weak solution of 
	\begin{equation}
		\label{7.5}
			\mc{L}u_{0}-\l_{0} u_{0} - |u_{0}|^{p}=\phi, \quad \quad x\in \mathbb{T}.
	\end{equation}
	By the regularity result stated in Proposition \ref{Reg}, we deduce that $u_{0}\in H^{2s}_{+}(\mathbb{T})$, $\phi=\mf{F}(\l_{0},u_{0})$. Therefore,
	\begin{equation}
		\mc{L}(u_{n_{k}}-u_{0})-(\l_{n_{k}}-\l_{0})u_{n_{k}}-\l_{0} (u_{n_{k}}-u_{0}) - (|u_{n_{k}}|^{p}-|u_{0}|^{p})=\phi_{n_{k}}-\phi, \quad x\in \mathbb{T}.
	\end{equation}
	By the regularity result stated in Proposition \ref{Reg} and the boundedness of $K$, we have the estimate
	\begin{align*}
		\|u_{n_{k}}-u_{0}\|_{H^{2s}}\leq C_1 \left(|\l_{n_{k}}-\l_{0}| + \|\phi_{n_{k}}-\phi\|_{L^{2}}+\|u_{n_k}-u_0\|_{L^2}\right), \quad k\in\N,
	\end{align*}
	for some positive constant $C_1>0$.
	Therefore, as $K$ is closed, we infer $(\l_{0},u_{0})\in K$. This proves that $\phi\in \mf{F}(K)$.
	\par
	Now, pick $\phi\in L^{2}_{+}(\mathbb{T})$. To show that $\mf{F}^{-1}(\phi)\cap K$ is compact in $[\l_-,\l_+]\times H^{2s}_{+}(\mathbb{T})$, let $\{(\l_{n},u_{n})\}_{n\in\mathbb{N}}$ be a sequence in $\mf{F}^{-1}(\phi)\cap K$. Then,
	\begin{equation}
		\label{7.6}
		\mf{F}(\l_{n},u_{n})=\phi \quad \hbox{for all}\;\; n\in\mathbb{N}.
	\end{equation}
	Based again on the compactness of the imbedding $H^{2s}_{+}(\mathbb{T}) \hookrightarrow \mc{C}(\mathbb{T})$, we can extract a subsequence $\{(\l_{n_{k}},u_{n_{k}})\}_{k\in\mathbb{N}}$ such that, for some $(\l_{0},u_{0})\in [\l_-,\l_+]\times \mc{C}(\mathbb{T})$,   $\lim_{k\to \infty} \l_{n_{k}} = \l_{0}$ and \eqref{7.4} holds.
	Similarly,  $u_{0}\in \mc{C}(\mathbb{T})$ is a weak solution of \eqref{7.5} and, by regularity,
	$u_{0}\in H^{2s}_{+}(\mathbb{T})$ and $\mf{F}(\l_{0},u_{0})=\phi$. In particular, for every $k\in\mathbb{N}$,
	\begin{equation*}
		\mc{L}(u_{n_{k}}-u_{0})-(\l_{n_{k}}-\l_{0})u_{n_{k}}-\l_{0} (u_{n_{k}}-u_{0}) -  (|u_{n_{k}}|^{p}-|u_{0}|^{p})=0, \quad x\in \mathbb{T}.
	\end{equation*}
	By the regularity result of Proposition \ref{Reg}, we have the estimate
	\begin{equation*}
		\|u_{n_{k}}-u_{0}\|_{H^{2s}}\leq C\left(|\l_{n_{k}}-\l_{0}|+\|u_{n_k}-u_0\|_{L^2}\right), \quad k\in\N,
	\end{equation*}
	for some positive constant $C>0$.  Therefore, letting $k\to \infty$ we finally get
	that
	$$
	\lim_{k\to \infty} (\l_{n_{k}},u_{n_{k}}) = (\l_{0},u_{0})\quad \hbox{in}\;\;
	[\l_-,\l_+] \times H^{2s}_{+}(\mathbb{T}).
	$$
	This concludes the proof.
\end{proof}

\section{Spectral theory for the linearization}\label{S3}

In this section, we proceed with the spectral theoretic study of the linearization of the operator $\mf{F}$ on the \textit{trivial branch} $\mc{T}$ defined by 
$$
\mc{T}:=\{(\l,u)\in\R\times H^{2s}_{+}(\mathbb{T}) \; : \; u=0 \}.
$$
It is called trivial branch since $\mf{F}(\l,u)=0$ for all $(\l,u)\in\mc{T}$. The linearization of $\mf{F}$ on $\mc{T}$ is given by the family of operators $\mf{L}(\l):= \partial_{u}\mf{F}(\l,0)$, $\l\in\R$, given explicitly by
$$
\mf{L}:\R\longrightarrow\Phi_{0}(H^{2s}_{+}(\mathbb{T}),L^{2}_{+}(\mathbb{T})), 
\quad \mf{L}(\l)[v] = \mc{L}v - \l v.
$$
To analyze the spectral properties of the family $\mf{L}$, we use the language of nonlinear spectral theory collected in Appendix \ref{A1}.
Let us recall that the \textit{generalized spectrum} of $\mf{L}$ is given by
$$\Sigma(\mf{L}):=\{\l\in \R \, : \, \mf{L}(\l)\notin GL\left(H^{2s}_{+}(\mathbb{T}),L^{2}_{+}(\mathbb{T})\right)\}.$$
The next lemma describes the set $\Sigma(\mf{L})$.
\begin{proposition}
	\label{Pr3.2}
	The generalized spectrum of the family $\mf{L}(\l)$ is given by
	\begin{equation}
		\label{L3.1}
		\Sigma(\mf{L})=\left\{k^{2s} \mf{m}(k) \, : \, k\in\N\right\}\cup\{0\}.
	\end{equation}
	 Moreover, they are ordered as
	 \begin{equation*}
	 	0<\mf{m}(1) < 2^{2s}\mf{m}(2) < \cdots < k^{2s}\mf{m}(k) < \cdots,
	 \end{equation*}
	 and it holds that
		\begin{align*}
			N[\mf{L}(0)]=\spann\{1\}, \quad N[\mf{L}(k^{2s}\mf{m}(k))]=\spann\left\{\cos(kx)\right\}, \quad k\in \N.
		\end{align*}
\end{proposition}
\begin{proof}
	Let $\l\in\R$ and take $v\in N[\mf{L}(\l)]$. Then, $\mf{L}(\l)[v]=0$ or equivalently $\mc{L}v-\l v=0$. In the Fourier side this becomes
	\begin{equation}
		\label{E2}
		\sum_{n\in \Z}(|n|^{2s} \mf{m}(n)-\l)\widehat{u}(n) e^{i n x}=0.
	\end{equation}
	Hence, if $\l\neq0$ and $\l\neq k^{2s} \mf{m}(k)$ for each $k\in\N$, necessarily $\widehat{u}(n)=0$ for each $n\in\Z$ and this implies that $u\equiv0$. Hence $N[\mf{L}(\l)]=\{0\}$.
	On the one hand, suppose that $\l=0$. Then, equation \eqref{E2} implies that 
	$$\widehat{u}(n)=0, \quad n\neq 0.$$
	Hence $u\equiv \widehat{u}(0)$ and $u$ is constant. Consequently, $N[\mf{L}(0)]=\spann\{1\}$ and $0\in\Sigma(\mf{L})$. Finally, if $\l= k^{2s} \mf{m}(k)$ for some $k\neq 0$, then equation \eqref{E2} implies that 
	$$\widehat{u}(n)=0, \quad n\neq \pm k.$$
	Hence $u(x) =\widehat{u}(-k)e^{-ikx}+\widehat{u}(k)e^{i k x}=2\widehat{u}(k)\cos(kx)$ and therefore 
	$$
	N[\mf{L}(k^{2s}\mf{m}(k))]=\spann\{\cos(kx)\}.
	$$ 
	Then $k^{2s} \mf{m}(k)\in \Sigma(\mf{L})$ for each $k\in\N$. Finally, the order of the eigenvalues follows from hypothesis (M2) on the multiplier $\mf{m}$. This concludes the proof.
\end{proof}

The next result proves that the operator $\mf{L}(\l)$ acting in $L^{2}_{+}(\mathbb{T})$ is self-adjoint.

\begin{lemma}
	For each $\l\in\R$, $\mf{L}(\l): H^{2s}_{+}(\mathbb{T})\subset L^{2}_{+}(\mathbb{T}) \to L^{2}_{+}(\mathbb{T})$ is a self-adjoint operator.
\end{lemma}
\begin{proof}
 Firstly, note that given $u, v\in H^{2s}_{+}(\mathbb{T})$,
	$$
	\int_{\mathbb{T}}\mc{L}u \; v = 2\pi \sum_{n\in\Z}|n|^{2s}\mf{m}(n)\widehat{u}(n) \widehat{v}(n)  = \int_{\mathbb{T}}\mc{L}v \; u.
	$$
 Then, using this identity, we deduce
	\begin{align*}
		(\mf{L}(\l)[u],v)_{L^{2}} & =\int_{\mathbb{T}}\mf{L}(\l)[u]v =\int_{\mathbb{T}} (\mc{L}u-\l u)v  = \int_{\mathbb{T}}\mc{L}v u - \l \int_{\mathbb{T}}uv  \\
		&= \int_{\mathbb{T}}(\mc{L}v-\l v)u =\int_{\mathbb{T}} \mf{L}(\l)[v] u =(u,\mf{L}(\l)[v])_{L^{2}}.
	\end{align*}
	Consequently $\mf{L}(\l)$ is a symmetric operator. It is easily shown that $\mathscr{D}(\mf{L}(\l))=\mathscr{D}(\mf{L}(\l)^{\ast})$. Hence, $\mf{L}(\l)$ is self-adjoint.
\end{proof}

The following result computes the algebraic multiplicity of each generalized eigenvalue of $\Sigma(\mf{L})$. For the sake of notation, subsequently, we will denote $\Sigma(\mf{L})=\{\s_{k}\}_{k=0}^{\infty}$ where  
$$\s_0=0, \; \; \s_{k}=k^{2s} \mf{m}(k),  \quad k\in\N.$$

\begin{proposition}
	\label{Pr3.3}
	For each $k\in\N\cup\{0\}$, the generalized eigenvalue $\s_k$ is $1$-transversal and its generalized algebraic multiplicity is
	$$\chi[\mf{L},\s_{k}]=1.$$
\end{proposition}

\begin{proof}
	An elementary computation gives
	\begin{equation}
		\label{Ep}
	\mf{L}_{1}(\l):=\frac{{\rm{d}}\mf{L}}{{\rm{d}}\l}(\l)= -J, \quad \l\in\R,
	\end{equation}
	where $J: H^{2s}_{+}(\mathbb{T})\hookrightarrow L^{2}_{+}(\mathbb{T})$ is the canonical embedding. 
	As $\mf{L}(\l)$ is a Fredholm operator, $R[\mf{L}(\l)]$ is closed in $L^{2}_{+}(\mathbb{T})$. Therefore, we deduce that in $L^{2}_{+}(\mathbb{T})$,
	$$
	R[\mf{L}(\l)]=\overline{R[\mf{L}(\l)]}=N[\mf{L}^{\ast}(\l)]^{\perp}=N[\mf{L}(\l)]^{\perp},
	$$
	where we have used that $\mf{L}(\l)$ is a self-adjoint operator on $L^{2}_{+}(\mathbb{T})$. Consequently, we get that for $k=0$,
	$$
	R[\mf{L}(\s_{0})]=\left\{u\in L^{2}_{+}(\mathbb{T}) \, : \, \int_{\mathbb{T}} u=0 \right\},
	$$
	and for each $k\in\N$,
	\begin{equation}
		\label{Rr}
	R[\mf{L}(\s_{k})]=\left\{u\in L^{2}_{+}(\mathbb{T}) \, : \, \int_{\mathbb{T}}\cos(kx) u=0 \right\}.
    \end{equation}
	Hence, by \eqref{Ep}, we obtain that
	\begin{align*}
	&\mf{L}_{1}(\s_{0})\left[1\right] = -1 \notin R[\mf{L}(\s_0)], \\
	&\mf{L}_{1}(\s_{k})[\cos(kx)] = -\cos(kx)\notin R[\mf{L}(\s_{k})], \quad k\in\N.
	\end{align*}
	Therefore, the transversality condition 
	$$
	\mf{L}_{ 1}(\s_{k})\left(N[\mf{L}(\s_{k})]\right)\oplus R[\mf{L}(\s_{k})]= L^{2}_{+}(\mathbb{T}),
	$$
	holds for each $k\in\N\cup\{0\}$. This implies that the eigenvalues $\s_{k}$ are $1$-transversal and by \eqref{CTCR}, we deduce that
	$$
	\chi[\mf{L},\s_{k}]=\dim N[\mf{L}(\s_{k})]= 1
	$$
	The proof is concluded.
\end{proof}

\section{Local bifurcation}\label{S4}

In this section we study the local bifurcation of non-trivial solutions of equation \eqref{e} from the points $(\s_k,0)\in \R\times H^{2s}_{+}(\mathbb{T})$. The main ingredient will be the well-known Crandall--Rabinowitz Theorem \ref{Cr-Rb}. The spectral analysis performed in Section \ref{S3}, in particular Proposition \ref{Pr3.3}, assures us that the hypotheses of this theorem are fulfilled. 
\par 
The \textit{set of non-trivial solutions} of $\mf{F}$ is defined by
\begin{equation*}
	\mf{S}:=\left[ \mf{F}^{-1}(0)\backslash \mc{T}\right]\cup \{(\l,0):\;\l\in \Sigma(\mf{L})\}\subset \R\times H^{2s}_{+}(\mathbb{T}).
\end{equation*}
By the continuity of $\mf{F}$ and the implicit function theorem at $(\l,0)$ for $\l\notin\Sigma(\mf{L})$, the subset $\mf{S}$ is closed in $\R\times H^{2s}_{+}(\mathbb{T})$.
The main local result is the following.

\begin{theorem}
	\label{LCRWw}
	Let $k\in\N$. Then, the point $(k^{2s}\mf{m}(k),0)\in \R \times H^{2s}_{\mathbf{+}}(\mathbb{T})$ is a bifurcation point of the non-linearity
	$$\mf{F}:\R\times H^{2s}_{+}(\mathbb{T})\longrightarrow L^{2}_{+}(\mathbb{T}), \quad \mf{F}(\l,u):=\mc{L}u-\l u - |u|^{p},$$
	from the trivial branch $\mc{T}$ to a connected component $\mathscr{C}_{k}$ of the set of non-trivial solutions $\mf{S}$. Let 
	$$Y_{k}:=\{u\in H^{2s}_{\mathbf{+}}(\mathbb{T}) \, : \, (u,\cos(kx))_{H^{2s}}=0 \}.$$
	Then, the following statements hold:
	\vspace{2pt}
	\begin{enumerate}
		\item[{\rm{(a)}}] {\rm\textbf{Existence:}} There exist $\e>0$ and two $\mc{C}^{\o(p)-2}$-maps
		\begin{equation}
			\L_{k}: (-\e,\e) \longrightarrow \R, \quad \L_{k}(0)=k^{2s}\mf{m}(k), \qquad \Gamma_{k}: (-\e,\e)\longrightarrow Y_{k}, \quad  \Gamma_{k}(0)=0,
			\label{Cu}
		\end{equation}
		such that for each $s\in(-\e,\e)$,
		\begin{equation}
			\mathfrak{F}(\L_{k}(s),u_{k}(s))=0, \quad u_{k}(s):= s(\cos(kx)+\Gamma_{k}(s)).
			\label{Cu1}
		\end{equation}
		In other words, for some $\r>0$,
		$$\mathscr{C}_{k}\cap B_{\r}(\s_{k},0)=\{(\L_{k}(s), s(\cos(kx)+\Gamma_{k}(s))) \, : \, s\in(-\varepsilon,\varepsilon)\}\cap B_{\rho}(\s_k,0).$$
		\vspace{1pt}
		\item[{\rm{(b)}}] {\rm\textbf{Uniqueness:}} There exists $\r >0$ such that if $\mathfrak{F}(\l,u)=0$ and
		$(\l,u)\in B_\r(\s_{k},0)\subset \R\times H^{2s}_{\mathbf{+}}(\mathbb{T})$, then either $u = 0$ or
		$(\l,u)=(\L_{k}(s),u_{k}(s))$ for some $s\in(-\e,\e)$. Therefore, from $(\s_{k},0)$, there emanate precisely two branches of non-trivial solutions of equation \eqref{Ee}. In other words,
		$$\mf{S}\cap B_{\rho}(\s_{k},0)=\mathscr{C}_{k}\cap B_{\rho}(\s_{k},0).$$
		Moreover, locally the solutions emanating from $(\s_{k},0)$ have exactly $2k$ zeros.
		\vspace{5pt}
		\item[{\rm{(c)}}] {\rm\textbf{Bifurcation direction:}} If $p=2$, the bifurcation direction is given by
		\begin{equation}
			\label{BF0}
			\dot{\L}_{k}(0)=0, \quad \ddot{\L}_{k}(0)=\frac{1}{ k^{2s}\mf{m}(k)}\frac{2^{2s+1}\mf{m}(2k)-3\mf{m}(k)}{2^{2s}\mf{m}(2k)-\mf{m}(k)}.
		\end{equation}
		Then $\ddot{\L}_{k}(0)>0$ and hence we are dealing with a supercritical bifurcation. This implies that for a sufficiently small $\r>0$, if $(\l,u)\in B_\r(\s_{k},0)\cap\mf{S}$ and $u\neq 0$, then necessarily $\l>\s_k$.
	\end{enumerate}
	
\end{theorem}

\begin{proof}
	The first part of the result, items (a) and (b), are a direct application of the Crandall--Rabinowitz Theorem \ref{Cr-Rb} applied to the nonlinearity $\mf{F}$. Let us show that $\mf{F}$ satisfies the hypothesis of Theorem \ref{Cr-Rb}. First of all, note that $\mf{F}\in\mc{C}^{\o(p)}$ by Lemma \ref{L1}. Hypothesis (F1) is clearly satisfied. Hypothesis (F2) is satisfied by Proposition \ref{Eqq} and finally, (F3) is a consequence of Proposition \ref{Pr3.3}. 
    For the zero count, write $w_t=u_k(t)/t\to\cos(kx)$ in $H^{2s}\hookrightarrow H^1$. The equation $\mc{L}w_t=\L_k(t)w_t+|t|^p|w_t|^p/t$ and the Fourier estimate improve this convergence to $H^{1+2s}\hookrightarrow\mc{C}^1$. Thus the $2k$ simple zeros persist and no others occur.

	\par We proceed to prove item (c).  Suppose that $p=2$. By Theorem \ref{BD}, we have that
	\begin{equation}
		\label{Id}
		\dot\Lambda_{k}(0)=-\frac{1}{2}\frac{( \partial^{2}_{uu}\mf{F}(\s_{k},0)[\varphi_{k},\varphi_{k}],\varphi_{k}^{\ast})_{L^2}}{( \partial^{2}_{\l u}\mf{F}(\s_{k},0)[\varphi_{k}],\varphi_{k}^{\ast})_{L^2}},
	\end{equation}
	where $\varphi_{k}(x)=\cos(kx)$, $x\in\mathbb{T}$, and $\varphi^{\ast}_{k}\in L^{2}_{+}(\mathbb{T})$ is a function such that
	$$
	R[\mf{L}(\s_k)]=\{f\in L^{2}_{+}(\mathbb{T}) \; : \; (f,\varphi^{\ast}_{k})_{L^{2}}=0\}.
	$$
	By identity \eqref{Rr}, we can choose $\varphi^{\ast}_{k} = \cos(kx)$. An standard computation using Lemma \ref{L1} gives
	\begin{align*}
		\partial^{2}_{uu}\mf{F}(\s_{k},0)[\varphi_{k},\varphi_{k}]=-2\cos^{2}(kx), \quad \partial^{2}_{\l u}\mathfrak{F}(\s_{k},0)[\varphi_{k}]=- \cos(kx).
	\end{align*}
	Therefore, a direct substitution on identity \eqref{Id} yields
	\begin{equation*}
		\dot{\Lambda}_{k}(0)=-\frac{1}{2}\frac{\left(- 2\cos^{2}(kx),\cos(kx)\right)_{L^{2}}}{\left( -\cos(kx),\cos(kx)\right)_{L^{2}}}= -\frac{1}{\pi}\int_{\mathbb{T}}\cos^{3}(kx)\; {\rm{d}}x = 0.
	\end{equation*}
	This proves that $\dot{\L}_{k}(0)=0$. On the other hand, by Lemma \ref{L1}, we get
	$$
	\partial^{3}_{uuu}\mf{F}(\s_{k},0)[\varphi_{k},\varphi_{k},\varphi_{k}]=0,
	$$
	therefore, by Theorem \ref{BD}, 
	\begin{align*}
		\ddot\Lambda_{k}(0)=-\frac{(\partial^{2}_{uu}\mf{F}(\s_{k},0)[\varphi_{k},\phi_{k}],\varphi_{k}^{\ast})_{L^2}}{( \partial^{2}_{\l u}\mf{F}(\s_{k},0)[\varphi_{k}],\varphi_{k}^{\ast})_{L^2}},
	\end{align*}
	where $\phi_{k}\in H^{2s}_{\mathbf{+}}(\mathbb{T})$ is any function satisfying
	$$
	\partial^{2}_{uu}\mf{F}(\s_{k},0)[\varphi_{k},\varphi_{k}]+\partial_{u}\mf{F}(\s_{k},0)[\phi_{k}]=0.
	$$
	This is equivalent to the non-homogeneous pseudo-differential equation
	\begin{equation}
		\label{Eq}
		\mc{L}\phi_{k}-\s_{k}\phi_{k}=2\cos^{2}(kx), \quad x\in\mathbb{T}.
	\end{equation}
	As we can rewrite
	$$\cos^{2}(kx) =  \frac{1}{4}e^{-2kxi} +\frac{1}{2}+\frac{1}{4}e^{2kxi},$$
	expanding the equation \eqref{Eq} in Fourier series, we obtain
	$$
	\sum_{n\in\Z}(|n|^{2s}\mf{m}(n)-|k|^{2s}\mf{m}(k))\widehat{\phi_{k}}(n)e^{inx} = \frac{1}{2}e^{-2kxi} +1+\frac{1}{2}e^{2kxi}.
	$$
	Comparing each summand, we deduce that
	\begin{align*}
		&	\widehat{\phi_{k}}(0)=-\frac{1}{k^{2s}\mf{m}(k)}, \\
		&	\widehat{\phi_{k}}(-2k)=\widehat{\phi_{k}}(2k)=\frac{1}{2}\frac{1}{(2k)^{2s}\mf{m}(2k)-k^{2s}\mf{m}(k)}, \\
		& \widehat{\phi_{k}}(n)=0, \quad n\neq -2k, -k, 0, k, 2k.
	\end{align*}
	Therefore, the function $\phi_k$ is given by
	\begin{equation*}
		\phi_{k}(x) = - \frac{1}{k^{2s}\mf{m}(k)} + 2\widehat{\phi_{k}}(k)\cos(kx) + \frac{1}{(2k)^{2s}\mf{m}(2k)-k^{2s}\mf{m}(k)}\cos(2kx).
	\end{equation*}
	This explicit representation of $\phi_{k}$ gives
	\begin{align*}
		\left(\partial^{2}_{uu}\mf{F}(\s_{k},0)[\varphi_{k},\phi_{k}], \varphi_{k}\right)_{L^{2}}  & =  \left(- 2\cos(kx)\phi_{k}(x), \cos(kx)\right)_{L^{2}} \\ 
		& = -2\int_{\mathbb{T}}\cos^{2}(kx)\phi_{k}(x)\;{\rm{d}}x \\
		& = 2\pi\left(\frac{1}{k^{2s}\mf{m}(k)}-\frac{1}{2}\frac{1}{(2k)^{2s}\mf{m}(2k)-k^{2s}\mf{m}(k)}\right) \\
		& = \frac{\pi}{ k^{2s}\mf{m}(k)}\frac{2^{2s+1}\mf{m}(2k)-3\mf{m}(k)}{2^{2s}\mf{m}(2k)-\mf{m}(k)}.
	\end{align*}
	Therefore, from the identity
	$$
	\left(\partial^{2}_{\l u}\mf{F}(\s_{k},0)[\varphi_{k}],\varphi_{k}^{\ast}\right)_{L^{2}}=\left(-\cos(kx), \cos(kx)\right)_{L^{2}}=-\pi,
	$$
	we deduce that
	$$
	\ddot{\L}_{k}(0) =-\frac{(\partial^{2}_{uu}\mf{F}(\s_{k},0)[\varphi_{k},\phi_{k}],\varphi_{k}^{\ast})_{L^2}}{( \partial^{2}_{\l u}\mf{F}(\s_{k},0)[\varphi_{k}],\varphi_{k}^{\ast})_{L^2}}=\frac{1}{ k^{2s}\mf{m}(k)}\frac{2^{2s+1}\mf{m}(2k)-3\mf{m}(k)}{2^{2s}\mf{m}(2k)-\mf{m}(k)}.
	$$
	This proves \eqref{BF0}. Finally, let us prove the inequality
	$$
	\ddot{\L}_{k}(0)\geq \frac{1}{k^{2s}\mf{m}(k)}>0, \quad k\in \N. 
	$$
	Indeed, by \eqref{BF0}, this is equivalent to
	$$
	\frac{2^{2s+1}\mf{m}(2k)-3\mf{m}(k)}{2^{2s}\mf{m}(2k)-\mf{m}(k)}\geq 1.
	$$
	But note that we can rewrite this inequality as $2^{2s-1}\mf{m}(2k)-\mf{m}(k)\geq 0$ which holds by hypothesis (M2). This concludes the proof.
	\end{proof}

\section{Constant solutions}\label{S5}

In this section we study the existence and structure of the set of constant solutions of equation \eqref{e}.  Firstly, we observe that there is a trivial branch of constant positive and negative solutions emanating from $(0,0)\in \R\times H^{2s}_{+}(\mathbb{T})$.

\begin{proposition}
	\label{Pr5.1}
	The equation $\mc{L}u= \l u + |u|^p$,  $x\in \mathbb{T}$, admits the following unique non-trivial constant solutions:
	\begin{equation}
		\label{Const}
	u_{\l}(x)= -\l^{\frac{1}{p-1}}, \quad x\in \mathbb{T}, \; \; \l>0, \quad u_{\l}(x)= (-\l)^{\frac{1}{p-1}}, \quad x\in \mathbb{T}, \; \; \l<0.
	\end{equation}
	These constant solutions exist for each $\l\in\R\setminus\{0\}$, being positive if $\l<0$ and negative if $\l>0$. 
\end{proposition}

\begin{proof}
	Let $\alpha\in\R$ and consider the constant function $u(x)=\alpha$, $x\in\mathbb{T}$. Then, as
	$\widehat{u}(0)=\alpha$ and $\widehat{u}(n)=0$ for $n\neq 0$, we deduce that
	$$
	\mc{L}u = \sum_{n\in\Z}|n|^{2s} \mf{m}(n) \widehat{u}(n) e^{inx} = 0.
	$$
	Therefore, if we suppose that $u$ is a solution of $\mc{L}u = \l u +  |u|^{p}$,  $x\in \mathbb{T}$, we reach to the equation
	$\l \alpha +|\alpha|^{p}=0$, from which the result follows.
\end{proof}
Let us denote by $\mathscr{C}_{0}$ the connected component of $\mf{S}$ such that $(0,0)\in \mathscr{C}_{0}$. By proposition \ref{Pr5.1}, we have that
$$
\{(\l,u)\in\R\times H^{2s}_{+}(\mathbb{T}) \; : \; u = u_{\l} \hbox{ as in \eqref{Const}}\}\subset\mathscr{C}_{0}.
$$
Therefore, $\mathscr{C}_{0}$ is unbounded and $\mc{P}_{\l}(\mathscr{C}_{0})=\R$ where $\mc{P}_{\l}:\R\times H^{2s}_{+}(\mathbb{T})\to \R$, $(\l,u)\mapsto \l$, is the $\l·$-projection. Our next purpose is to study in greater detail the connected component $\mathscr{C}_{0}$. For that, we will call the solutions
$$
\mc{T}_{2}:=\{(\l,u)\in\R\times H^{2s}_{+}(\mathbb{T}) \; : \; u = u_{\l} \hbox{ as in \eqref{Const}}\},
$$
the \textit{secondary trivial branch}. The linearization of $\mf{F}$ on $\mc{T}_{2}$ is given by 
$$
\mf{P}:\R\longrightarrow\Phi_{0}(H^{2s}_{+}(\mathbb{T}),L^{2}_{+}(\mathbb{T})), \quad \mf{P}(\l)[v]=\partial_{u}\mf{F}(\l,u_{\l})[v] = \mc{L}v+\l(p-1) v.
$$
The next result studies the spectral properties of $\mf{P}$. We omit the proof as it is analogous to that of Lemma \ref{L3.1} and Proposition \ref{Pr3.3}.
\begin{lemma}
	\label{L5.2}
	The generalized spectrum of the family $\mf{P}(\l)$ is given by
	\begin{equation}
		\label{E1.1}
		\Sigma(\mf{P})=\left\{-\frac{k^{2s} \mf{m}(k)}{p-1} \, : \, k\in\N\right\}\cup\{0\}.
	\end{equation}
	Moreover, they are ordered as
	$$
	0>-\frac{\mf{m}(1)}{p-1}>-\frac{2^{2s} \mf{m}(2)}{p-1}>\cdots>-\frac{k^{2s} \mf{m}(k)}{p-1}>\cdots,
	$$
	and it holds that
	\begin{align*}
		N[\mf{P}(0)]=\spann\{1\}, \quad N[\mf{P}(-\tfrac{k^{2s} \mf{m}(k)}{p-1})]=\spann\left\{\cos(kx)\right\}, \quad k\in \N.
	\end{align*}
	Finally, the generalized eigenvalue $-\tfrac{k^{2s} \mf{m}(k)}{p-1}$ is $1$-transversal and its generalized algebraic multiplicity is
	$$\chi[\mf{P},-\tfrac{k^{2s} \mf{m}(k)}{p-1}]=1.$$
\end{lemma}

Thanks to Lemma \ref{L5.2}, we are able to prove that for $\l>0$, the connected component $\mathscr{C}_{0}$ is consisted uniquely on the negative constant solutions $u_{\l}(x)=-\l^{\frac{1}{p-1}}$. This is the content of the following proposition.
 
 \begin{proposition}
 	\label{Le7.1}
 	The following set identity holds:
 	\begin{equation}
 		\label{Si}
 	\mathscr{C}_{0}\cap \{(\l,u)\in\R\times H^{2s}_{+}(\mathbb{T}) : \l>0\} = \{(\l,u)\in\R\times H^{2s}_{+}(\mathbb{T}) : \l>0, \;  u(x)=-\l^{\frac{1}{p-1}}\}.
 	\end{equation}
 \end{proposition}
 
 \begin{proof}
     Let us start by proving that $\mathscr{C}_0\cap\{\l>0\}$ is connected. Set
     $$B:=\{(\l,-\l^{\frac{1}{p-1}}):\l>0\},\qquad D:=(\mf{S}\cap\{\l\leq0\})\cup B.$$
     The set $D$ is closed in $\mf{S}$. It is also open, by Lemma \ref{L5.2} and the implicit function theorem at points of $B$, and by local uniqueness in Theorem \ref{Cr-Rb} at $(0,0)$, where the local branch consists of the constants in Proposition \ref{Pr5.1}. By Corollary \ref{Cr6.2}, there are no other solutions with $\l=0$. Thus $\mathscr{C}_0\subset D$. Since $B\subset\mathscr{C}_0$, the positive-parameter part of $\mathscr{C}_0$ equals $B$ and is connected.
     
 	 If identity \eqref{Si} does not hold, there must exist $(\l_0,u_0)\in \mathscr{C}_{0}$ with $\l_{0}>0$ and $u_0\neq -\l_0^{\frac{1}{p-1}}$.
 	  As $(\l_0,u_0)\in\mathscr{C}_{0}$, the connectivity of $\mathscr{C}_{0}\cap \{(\l,u)\in\R\times H^{2s}_{+}(\mathbb{T}) :  \l>0\}$ implies the existence of a connected subset 
 	 $$
 	 A\subset\mathscr{C}_{0}\cap \{(\l,u)\in\R\times H^{2s}_{+}(\mathbb{T}) :  \l>0\}
 	 $$ 
 	 such that $(\l_0,u_0), (\l_{\ast}, -\l_{\ast}^{\frac{1}{p-1}})\in A$ for some $\l_{\ast}>0$. See Figure \ref{F1} for a graphical explanation of this fact. In this figure, and in all subsequent ones, we are representing the value of the parameter $\l$ in abscissas versus the norm $\|u\|_{H^{2s}}$ or $-\|u\|_{H^{2s}}$. This differentiation is made in order to express the multiplicity of solutions in the subsequent analysis. Therefore there must exists a sequence 
 	$
 	\{(\l_{n},u_{n})\}_{n\in\N}\subset A$, with $u_{n}$ non-constant, such that
 	$$
 	\lim_{n\to+\infty}(\l_n,u_n) = (\l_{\ast}, -\l_{\ast}^{\frac{1}{p-1}}) \;\; \hbox{in} \;\; \R\times H^{2s}_{+}(\mathbb{T}).
 	$$
 	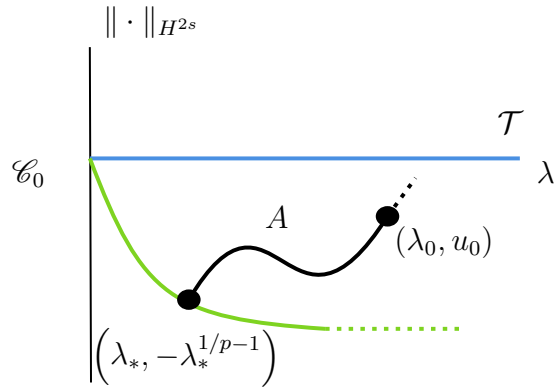
\begin{figure}[h!]

 		\tikzset{every picture/.style={line width=0.75pt}} 
 		
 		\begin{tikzpicture}[x=0.75pt,y=0.75pt,yscale=-1,xscale=1]
 			
 			\draw    (163.86,43.73) -- (164.44,213) ;
 			\draw [color={rgb, 255:red, 74; green, 144; blue, 226 }  ,draw opacity=1 ][line width=1.5]    (163.86,99.9) -- (380.73,99.9) ;
 			\draw [color={rgb, 255:red, 126; green, 211; blue, 33 }  ,draw opacity=1 ][line width=1.5]    (163.86,99.9) .. controls (190.45,170.11) and (204.78,180.64) .. (282.52,185.91) ;
 			\draw  [color={rgb, 255:red, 0; green, 0; blue, 0 }  ,draw opacity=1 ][fill={rgb, 255:red, 0; green, 0; blue, 0 }  ,fill opacity=1 ] (308.56,129.13) .. controls (308.56,126.56) and (310.99,124.47) .. (313.99,124.47) .. controls (316.98,124.47) and (319.41,126.56) .. (319.41,129.13) .. controls (319.41,131.7) and (316.98,133.78) .. (313.99,133.78) .. controls (310.99,133.78) and (308.56,131.7) .. (308.56,129.13) -- cycle ;
 			\draw [color={rgb, 255:red, 126; green, 211; blue, 33 }  ,draw opacity=1 ][line width=1.5]  [dash pattern={on 1.69pt off 2.76pt}]  (282.52,185.91) -- (350.04,185.91) ;
 			\draw [line width=1.5]    (213.8,171.18) .. controls (256.76,99.21) and (264.95,206.28) .. (313.99,129.13) ;
 			\draw  [color={rgb, 255:red, 0; green, 0; blue, 0 }  ,draw opacity=1 ][fill={rgb, 255:red, 0; green, 0; blue, 0 }  ,fill opacity=1 ] (208.37,171.18) .. controls (208.37,168.61) and (210.8,166.52) .. (213.8,166.52) .. controls (216.79,166.52) and (219.22,168.61) .. (219.22,171.18) .. controls (219.22,173.75) and (216.79,175.83) .. (213.8,175.83) .. controls (210.8,175.83) and (208.37,173.75) .. (208.37,171.18) -- cycle ;
 			\draw [line width=1.5]  [dash pattern={on 1.69pt off 2.76pt}]  (313.99,129.13) -- (328.12,109.76) ;
 			
 			\draw (169.03,21.57) node [anchor=north west][inner sep=0.75pt]    {$\| \cdot \| _{H^{2s}}$};
 			\draw (387.85,99.79) node [anchor=north west][inner sep=0.75pt]    {$\lambda $};
 			\draw (250.78,122.39) node [anchor=north west][inner sep=0.75pt]    {$A$};
 			\draw (163.25,183.85) node [anchor=north west][inner sep=0.75pt]  [rotate=-0.65]  {$\left( \lambda _{\ast } ,-\lambda _{\ast }^{\frac{1}{p-1}}\right)$};
 			\draw (315.99,132.53) node [anchor=north west][inner sep=0.75pt]    {$( \lambda _{0} ,u_{0})$};
 			\draw (123.04,99.26) node [anchor=north west][inner sep=0.75pt]    {$\mathscr{C}_{0}$};
 			\draw (368,74.4) node [anchor=north west][inner sep=0.75pt]    {$\mathcal{T}$};

 		\end{tikzpicture}
 		\caption{Graphic representation of the connected subset $A$}
 		\label{F1}
 	\end{figure}
 	
 	\noindent But then, by the implicit function theorem, we deduce that 
 	$$
 	\mf{P}(\l_{\ast})=\partial_{u}\mf{F}(\l_{\ast},-\l_{\ast}^{\frac{1}{p-1}})\notin GL(H^{2s}_{+}(\mathbb{T}),L^{2}_{+}(\mathbb{T})).
 	$$
 	Therefore, $\l_{\ast}\in\Sigma(\mf{P})$. As $\l_{\ast}>0$, this contradicts Lemma \ref{L5.2}, identity \eqref{E1.1}. The proof is concluded.
 \end{proof}

\section{A priori bounds}\label{S6}

This section is devoted to obtain $H^{2s}$ and $L^{\infty}$-a priori bounds for solutions of \eqref{e} in terms of the bifurcation parameter $\l$. This bounds will be fundamental in order to apply the global bifurcation techniques. The main ingredient we use for getting a priori bounds will be the almost sharp fractional Gagliardo--Nirenberg--Sobolev inequality on the torus $\mathbb{T}$  obtained recently by Liang and Wang \cite{LW}.
\par Our first result establishes $L^{2}$-a priori bounds for weak solutions of equation \eqref{e}.

\begin{proposition}
	\label{Pr6.1}
	Let $u\in H^{s}(\mathbb{T})$ be a weak solution of 
	\begin{equation}
		\label{EQ2}
	\mc{L}u = \l u +|u|^{p}, \quad x\in\mathbb{T}.
    \end{equation}
	Then, it holds that
	\begin{equation}
		\label{E}
	\|u\|_{L^{2}} \leq \sqrt{2\pi}|\l|^{\frac{1}{p-1}}.
	\end{equation}
\end{proposition}

\begin{proof}
	Firstly, recall that a weak solution $u\in H^{s}(\mathbb{T})$ of \eqref{EQ2} satisfies
	\begin{equation}
		\label{WS}
	2\pi\sum_{n\in\Z}|n|^{2s}\mf{m}(n)\widehat{u}(n)\overline{\widehat{\varphi}(n)} = \l \int_{\mathbb{T}}u\varphi + \int_{\mathbb{T}}|u|^{p}\varphi, \quad \text{for each} \;\; \varphi\in \mc{C}^{\infty}(\mathbb{T}).
	\end{equation}
	Note that as $s\geq \tfrac{1}{2}$, \eqref{WS} makes sense because for every $q > 1$, we have the Sobolev embedding $H^{s}(\mathbb{T})\hookrightarrow L^{q}(\mathbb{T})$. Then, choosing $\varphi = 1$ in the last identity, we obtain
	$$
	\int_{\mathbb{T}} |u|^{p} =  -\l \int_{\mathbb{T}} u.
	$$
	By H\"{o}lder's inequality, we deduce
	$$
	\frac{1}{(2\pi)^{\frac{p-2}{2}}}\left(\int_{\mathbb{T}}|u|^{2}\right)^{\frac{p}{2}}\leq \int_{\mathbb{T}}|u|^{p}\leq |\l| \int_{\mathbb{T}}|u|.
	$$
	Applying the Cauchy--Schwarz inequality we obtain
	$$
	\frac{1}{(2\pi)^{\frac{p-2}{2}}}\|u\|_{L^{2}}^{p}\leq |\l|\|u\|_{L^{1}}\leq \sqrt{2\pi}  |\l| \|u\|_{L^{2}}.
	$$
	This concludes the proof.
\end{proof}

As a rather direct consequence of Proposition \ref{Pr6.1}, we get that the unique solution of equation \eqref{e} with $\l=0$ is $u\equiv 0$. This is the content of the following corollary.

\begin{corollary}
	\label{Cr6.2}
	The unique weak solution $u\in H^{s}(\mathbb{T})$ of the equation
	\begin{equation*}
		\label{EQ3}
		\mc{L}u= |u|^{p}, \quad x\in\mathbb{T},
	\end{equation*}
	is $u\equiv 0$.
\end{corollary}

The next result establishes $\dot{H}^{s}$-a priori bounds for weak solutions of \eqref{e}. Here, as we said earlier, the main ingredient will be the  almost sharp fractional Gagliardo--Nirenberg--Sobolev inequality on the torus of Liang and Wang \cite{LW}.

\begin{theorem}
	\label{Th6.3}
	Let $u\in H^{s}(\mathbb{T})$ be a weak solution of 
	\begin{equation}
		\label{Eqq3}
		\mc{L}u = \l u + |u|^{p}, \quad x\in\mathbb{T}.
	\end{equation}
	Then, if $p<4s+1$, for each $\rho>0$, it holds that
	\begin{equation}
		\label{E11}
	\|u\|_{\dot{H}^{s}} \leq \Phi_{\rho}(\l),
\end{equation}
where $\Phi_{\rho}: \R \to [0,+\infty)$ is an even continuous function such that $\Phi_{\rho}(0)=0$, $\Phi_{\rho}(\l)>0$ for $\l\neq0$ and 
$$\Phi_{\rho}(\l)\sim \mathscr{L}(s,p,\r,m_0)|\l|^{\frac{2sp+2s-p+1}{(4s-p+1)(p-1)}}, \quad \text{as} \;\; |\l|\to+\infty,$$
where 
\begin{equation}
	\label{L}
\mathscr{L} \equiv \mathscr{L}(s,p,\r,m_0):=\left(\frac{(C_{GNS}+\r)}{m_{0}}(2\pi)^{\frac{p-1}{4s}(2s-1)}\right)^{\frac{2s}{4s-p+1}},
\end{equation}
and $C_{GNS}>0$ is a positive constant appearing in the Gagliardo--Nirenberg--Sobolev inequality.
\end{theorem}

\begin{proof}
	Firstly, recall that a weak solution $u\in H^{s}(\mathbb{T})$ of \eqref{Eqq3} satisfies
	\begin{equation}
		\label{EEE}
	2\pi\sum_{n\in\Z}|n|^{2s}\mf{m}(n)\widehat{u}(n)\overline{\widehat{\varphi}(n)} = \l \int_{\mathbb{T}}u\varphi + \int_{\mathbb{T}}|u|^{p}\varphi, \quad \text{for each} \;\; \varphi\in H^{s}(\mathbb{T}).
    \end{equation}
	Then, choosing $\varphi = u\in H^{s}(\mathbb{T})$ in identity \eqref{EEE}, we obtain
	\begin{equation}
		\label{idd}
	2\pi\sum_{n\in\Z}|n|^{2s}\mf{m}(n)|\widehat{u}(n)|^{2} = \l \int_{\mathbb{T}}|u|^{2} + \int_{\mathbb{T}}|u|^{p} u. 
	\end{equation}
	Now, by the hypothesis (M3) on the multiplier $\mf{m}$, we obtain the bound
	$$
	2\pi\sum_{n\in\Z}|n|^{2s}\mf{m}(n)|\widehat{u}(n)|^{2} \geq 2\pi m_{0}\sum_{n\in\Z}|n|^{2s} |\widehat{u}(n)|^{2} = 2\pi m_{0} \|u\|_{\dot{H}^{s}}^{2}.
	$$
	Introducing this inequality on identity \eqref{idd} and bounding the right hand part, we obtain
	\begin{equation}
		\label{EQQ}
	2\pi m_{0} \|u\|_{\dot{H}^{s}}^{2} \leq |\l| \|u\|_{L^{2}}^{2} + \int_{\mathbb{T}}|u|^{p}u\leq |\l|\|u\|_{L^{2}}^{2}+\|u\|_{L^{p+1}}^{p+1}.
	\end{equation}
	The almost sharp Gagliardo--Nirenberg--Sobolev inequality of Liang--Wang \cite[Pr. 2.3]{LW}, states that for each $\rho>0$, there exists $C(\rho)>0$ such that 
	\begin{equation}
	\label{QQ}
	\|u\|^{p+1}_{L^{p+1}} \leq (C_{GNS}+\r) \|u\|_{\dot{H}^{s}}^{\frac{p-1}{2s}} \|u\|^{2+\frac{p-1}{2s}(2s-1)}_{L^{2}} + C(\rho) \|u\|^{p+1}_{L^{2}},
	\end{equation}
	where $C_{GNS}>0$ is a constant independent of $u$. Inequalities \eqref{E}, \eqref{EQQ} and \eqref{QQ} yield
	\begin{align*}
		 2\pi m_{0} \|u\|_{\dot{H}^{s}}^{2}  & \leq (C_{GNS}+\r)\|u\|^{2+\frac{p-1}{2s}(2s-1)}_{L^{2}}\|u\|_{\dot{H}^{s}}^{\frac{p-1}{2s}}+|\l| \|u\|^{2}_{L^{2}}+ C(\r)\|u\|^{p+1}_{L^{2}} \\
		& \leq (C_{GNS}+\r)(2\pi)^{1+\frac{p-1}{4s}(2s-1)}|\l|^{\frac{2}{p-1}+\frac{2s-1}{2s}}\|u\|_{\dot{H}^{s}}^{\frac{p-1}{2s}} \\
		& \quad + 2\pi|\l|^{\frac{2}{p-1}+1}+C(\r)(2\pi)^{\frac{p+1}{2}}|\l|^{\frac{p+1}{p-1}}.
	\end{align*}
	Equivalently, the last inequality can be rewritten as
	$$
	2\pi m_{0} \|u\|_{\dot{H}^{s}}^{2} - \mathscr{C}_{1} |\l|^{\frac{2}{p-1}+\frac{2s-1}{2s}} \|u\|_{\dot{H}^{s}}^{\frac{p-1}{2s}} - \mathscr{C}_{2} |\l|^{\frac{p+1}{p-1}}\leq 0,
	$$
	where
	\begin{align*}
	& \mathscr{C}_{1} \equiv \mathscr{C}_{1}(s,p,\r):=(C_{GNS}+\r)(2\pi)^{1+\frac{p-1}{4s}(2s-1)}, \\
	& \mathscr{C}_{2} \equiv \mathscr{C}_{2}(p,\r):= 2\pi + C(\r)(2\pi)^{\frac{p+1}{2}}.
	\end{align*}
	Let $\mf{p}_{\l,\r}:\R_{\geq0}\to\R$ be the function defined by
	$$
	\mf{p}_{\l,\r}(x):=2\pi m_{0} x^{2} - \mathscr{C}_{1} |\l|^{\frac{2}{p-1}+\frac{2s-1}{2s}} x^{\frac{p-1}{2s}} - \mathscr{C}_{2} |\l|^{\frac{p+1}{p-1}}.
	$$
	Note that as $p<4s+1$, the dominant power of $\mf{p}_{\l,\r}$ is $x^2$. This is an important fact, as if $p\geq 4s+1$, the dominant power of $\mf{p}_{\l,\r}$ is not necessarily $x^2$ and the following arguments cannot be applied.  For $\l\neq0$, as
	$$
	\mf{p}_{\l,\r}(0)=-\mathscr{C}_{2} |\l|^{\frac{p+1}{p-1}}<0, \quad \lim_{x\uparrow+\infty}\mf{p}_{\l,\r}(x)=+\infty,
	$$
	and by an analysis of the first derivative of $\mf{p}_{\l,\r}$, we infer the existence of a unique positive zero $\Phi_{\r}(\l)>0$ of $\mf{p}_{\l,\r}$. Moreover, 
	$$
	\mf{p}_{\l,\r}(x)<0, \quad x\in [0,\Phi_{\r}(\l)).
	$$
	Therefore, $\|u\|_{\dot{H}^{s}} \leq \Phi_{\r}(\l)$.  Note that $\Phi_{\r}:\R \to \R_{\geq 0}$ is continuous and $\Phi_{\r}(0)=0$. On the other hand, by rewriting $\Phi_{\r}(\l)$ in the implicit form
	$$
	2\pi m_{0}\Phi_{\r}^{2}(\l)-\mathscr{C}_{1}|\l|^{\frac{2}{p-1}+\frac{2s-1}{2s}} \Phi_{\r}^{\frac{p-1}{2s}} (\l) - \mathscr{C}_{2} |\l|^{\frac{p+1}{p-1}} =0,
	$$
	or equivalently, 
	$$
	2\pi m_{0}\left(\frac{\Phi_{\r}(\l)}{|\l|^{\frac{1}{2}\frac{p+1}{p-1}}}\right)^{2}-\mathscr{C}_{1} \left(\frac{\Phi_{\r} (\l)}{|\l|^{\frac{1}{p-1}}}\right)^{\frac{p-1}{2s}} - \mathscr{C}_{2}=0,
	$$
	we deduce that $\Phi_{\r}(\l)$ is even. Moreover, rewriting the implicit equation of $\Phi_{\r}(\l)$ again, in the form
	$$
	\frac{\Phi_{\r}(\l)^{\frac{p-1}{2s}}}{|\l|^{\frac{1}{2s}}}\left(2\pi m_{0}\frac{\Phi_{\r}(\l)^{2-\frac{p-1}{2s}}}{|\l|^{\frac{p+1}{p-1}-\frac{1}{2s}}}-\mathscr{C}_{1}\right) =  \mathscr{C}_{2},
	$$
	we deduce, by a power analysis, that
	$$
	\lim_{|\l|\to+\infty}\frac{\Phi_{\r}(\l)^{2-\frac{p-1}{2s}}}{|\l|^{\frac{p+1}{p-1}-\frac{1}{2s}}}=\frac{\mathscr{C}_{1}}{2\pi m_{0}}.
	$$
	This concludes the proof.
\end{proof}

As a direct consequence of Proposition \ref{Pr6.1} and Theorem \ref{Th6.3}, we obtain the $H^{s}$-a priori bounds for weak solutions of \eqref{e}.

\begin{theorem}[$H^{s}$-\textbf{a priori bounds for weak solutions}]
	\label{TABW}
		Let $u\in H^{s}(\mathbb{T})$ be a weak solution of 
	\begin{equation*}
		\mc{L}u = \l u + |u|^{p}, \quad x\in\mathbb{T}.
	\end{equation*}
	Then, if $p<4s+1$, for each $\rho>0$, it holds that
	\begin{equation*}
		\|u\|_{H^{s}} \leq  \sqrt{2\pi |\l|^{\frac{2}{p-1}}+\Phi^{2}_{\r}(\l)},
	\end{equation*}
	where $\Phi_{\rho}: \R \to [0,+\infty)$ is the function defined in Theorem \ref{Th6.3}.
\end{theorem}

Now, we obtain a priori bounds for strong $H^{2s}$-solutions of equation \eqref{e}. This is the content of the next result.

\begin{theorem}
	\label{TT}
	Let $u\in H^{2s}(\mathbb{T})$ be a solution of 
	\begin{equation}
		\label{EQ31}
		\mc{L}u = \l u +|u|^{p}, \quad x\in\mathbb{T}.
	\end{equation}
	Then, if $p<4s+1$, for each $\r>0$, it holds that
	\begin{equation}
		\label{E111}
		\|u\|_{\dot{H}^{2s}} \leq \Psi_{\r}(\l),
	\end{equation}
	where $\Psi_{\r}:\R\to[0,+\infty)$ is an even continuous function such that $\Psi_{\r}(0)=0$, $\Psi_{\r}(\l)>0$ for $\l\neq 0$ and 
	\begin{equation}
		\label{ER}
	\Psi_{\r}(\l)\sim \frac{A^{p}_{2p}}{\sqrt{2\pi}m_{0}}\mathscr{L}^{p}|\l|^{\frac{2sp+2s-p+1}{(4s-p+1)(p-1)}p}, \quad \text{as} \;\; |\l|\to+\infty,
\end{equation}
where $A_{2p}$ is the optimal constant of the embedding $H^{s}(\mathbb{T})\hookrightarrow L^{2p}(\mathbb{T})$ and $\mathscr{L}$ is the constant given in \eqref{L}.
\end{theorem}

\begin{proof}
	Firstly, note that by hypothesis (M3),
	\begin{align*}
		\|\mc{L}u\|_{L^{2}}^{2} = 2\pi\sum_{n\in\Z}|n|^{4s}|\mf{m}(n)|^{2}|\widehat{u}(n)|^{2} \geq 2\pi m_{0}^{2}\sum_{n\in\Z}|n|^{4s}|\widehat{u}(n)|^{2} = 2\pi m_{0}^{2}\|u\|^{2}_{\dot{H}^{2s}}.
	\end{align*}
	Then $\|\mc{L}u\|_{L^{2}}\geq \sqrt{2\pi} m_{0}\|u\|_{\dot{H}^{2s}}$. Take the $L^{2}$-norm in equation \eqref{EQ31} to obtain
	\begin{align*}
		2\pi m_{0}^{2}\|u\|_{\dot{H}^{2s}}^{2} & \leq \|\l u + |u|^{p} \|_{L^{2}}^{2} = |\l|^{2} \|u\|^{2}_{L^{2}} + 2\l \int_{\mathbb{T}}|u|^{p}u + \int_{\mathbb{T}}|u|^{2p} \\
		& \leq |\l|^{2} \|u\|^{2}_{L^{2}}+2|\l| \|u\|^{p+1}_{L^{p+1}}+\|u\|^{2p}_{L^{2p}}.
	\end{align*}
	On the other hand, for each $1\leq r < +\infty$, the Sobolev embedding $H^{s}(\mathbb{T})\hookrightarrow L^{r}(\mathbb{T})$ yields
	\begin{equation*}
		\|u\|_{L^{r}}\leq A_{r} \|u\|_{H^{s}},
	\end{equation*}
	for some optimal constant $A_{r}>0$ independent of $u$. Then, applying the bounds obtained in Proposition \ref{Pr6.1} and Theorem \ref{Th6.3}, we obtain
	\begin{align*}
		2\pi m_{0}^{2}\|u\|^{2}_{\dot{H}^{2s}} & \leq |\l|^{2} \|u\|^{2}_{L^{2}} + 2|\l|  A_{p+1}^{p+1} \|u\|^{p+1}_{H^{s}} + A^{2p}_{2p}\|u\|^{2p}_{H^{s}} \\
		& = |\l|^{2} \|u\|^{2}_{L^{2}} + 2|\l| A^{p+1}_{p+1}(\|u\|^{2}_{L^{2}} + \|u\|^{2}_{\dot{H}^{s}})^{\frac{p+1}{2}} + A^{2p}_{2p}(\|u\|^{2}_{L^{2}}+\|u\|^{2}_{\dot{H}^{s}})^{p} \\
		& \leq 2\pi|\l|^{\frac{2}{p-1}+2}+2|\l|A^{p+1}_{p+1}\left( 2\pi|\l|^{\frac{2}{p-1}}+\Phi_{\r}^{2}(\l)\right)^{\frac{p+1}{2}}+A^{2p}_{2p}\left(2\pi|\l|^{\frac{2}{p-1}}+\Phi_{\r}^{2}(\l)\right)^{p}.
	\end{align*}
	Therefore, setting 
	\begin{equation}
		\label{Sqr}
	\Psi_{\r}(\l):= \frac{1}{\sqrt{2\pi}m_{0}}\sqrt{2\pi|\l|^{\frac{2}{p-1}+2}+2|\l|A^{p+1}_{p+1}\left( 2\pi|\l|^{\frac{2}{p-1}}+\Phi_{\r}^{2}(\l)\right)^{\frac{p+1}{2}}+A^{2p}_{2p}\left(2\pi|\l|^{\frac{2}{p-1}}+\Phi_{\r}^{2}(\l)\right)^{p}},
	\end{equation}
	we have $\|u\|_{\dot{H}^{2s}}\leq \Psi_{\r}(\l)$. Finally, to obtain the grow estimate \eqref{ER}, we have to determine which term of \eqref{Sqr} has the greater power. For that, taking into account that $$\Phi_{\r}(\l)\sim\mathscr{L}|\l|^{\frac{2sp+2s-p+1}{(4s-p+1)(p-1)}},$$
	where $\mathscr{L}$ is the constant given in \eqref{L}, it can be easily shown that for every $s\geq \tfrac{1}{2}$ and $2 \leq p < 4s+1$,
	\begin{equation*}
	\frac{2p}{p-1}  < 1+\frac{2sp+2s-p+1}{(4s-p+1)(p-1)}(p+1) < \frac{2sp+2s-p+1}{(4s-p+1)(p-1)}2p.
	\end{equation*}
	From this, the grow estimate \eqref{ER} follows. The proof is complete.
\end{proof}

As a direct consequence of Proposition \ref{Pr6.1} and Theorem \ref{TT}, we obtain the $H^{2s}$-a priori bounds for strong solutions of \eqref{e}.

\begin{theorem}[$H^{2s}$-\textbf{a priori bounds for strong solutions}]
	\label{TAB}
	Let $u\in H^{2s}(\mathbb{T})$ be a solution of 
	\begin{equation*}
		\mc{L}u = \l u +|u|^{p}, \quad x\in\mathbb{T}.
	\end{equation*}
	Then, if $p<4s+1$, for each $\r>0$, it holds that
	\begin{equation*}
		\|u\|_{H^{2s}} \leq \sqrt{2\pi |\l|^{\frac{2}{p-1}}+\Psi_{\r}^{2}(\l)},
	\end{equation*}
	where $\Psi_{\r}:\R\to[0,+\infty)$ is the function defined in Theorem \ref{TT}.
\end{theorem}

Finally, thanks to Theorem \ref{TAB}, we deduce $L^{\infty}$-a priori bounds for strong solutions of equation \eqref{e}.

\begin{theorem}[$L^{\infty}$-\textbf{a priori bounds for strong solutions}]
	\label{LAB}
	Let $u\in H^{2s}(\mathbb{T})$ be a solution of the equation
	\begin{equation}
		\label{Eo}
		\mc{L}u = \l u + |u|^{p}, \quad x\in\mathbb{T}.
	\end{equation}
	Then, $u\in \mc{C}^{2s-\frac{1}{2}}(\mathbb{T})$. Moreover,  if $p<4s+1$, for each $\r>0$, it holds that
	\begin{equation}
		\label{EQ}
		\|u\|_{L^{\infty}} \leq 2\sqrt{\zeta(4s)} \sqrt{2\pi |\l|^{\frac{2}{p-1}}+\Psi_{\r}^{2}(\l)},
	\end{equation}
	where $\Psi_{\r}:\R\to[0,+\infty)$ is the function defined in Theorem \ref{TT}.
\end{theorem}

\begin{proof}
	Since $u,|u|^p\in H^1(\mathbb{T})$, the equation and the Fourier estimate give $u\in H^{2s+1}(\mathbb{T})$. The first statement follows from the Sobolev embedding $H^{2s+1}(\mathbb{T})\hookrightarrow \mc{C}^{2s-\frac{1}{2}}(\mathbb{T})$. The second statement follows by the following standard argument. Let $u\in H^{2s}(\mathbb{T})$. Then H\"{o}lder's inequality yields
	\begin{align*}
		\|u\|_{L^{\infty}} & \leq |\widehat{u}(0)| + \sum_{n\neq 0} |\widehat{u}(n)| \leq \frac{1}{2\pi}\|u\|_{L^{1}}+ \sum_{n \neq 0} \frac{1}{|n|^{2s}}|n|^{2s}|\widehat{u}(n)| \\
		& \leq 
		\frac{1}{\sqrt{2\pi}}\|u\|_{L^{2}}+\left(\sum_{n\neq 0}\frac{1}{|n|^{4s}}\right)^{\frac{1}{2}}\left(\sum_{n\neq 0}|n|^{4s} |\widehat{u}(n)|^{2}\right)^{\frac{1}{2}} \\
		& \leq \max\left\{\frac{1}{\sqrt{2\pi}}, \sqrt{2\zeta(4s)}\right\} \left(\|u\|_{L^{2}}+\|u\|_{\dot{H}^{2s}}\right) \leq 2\sqrt{\zeta(4s)}\|u\|_{H^{2s}},
	\end{align*}
	where $\zeta(s)$ is the Riemann's zeta function. Theorem \ref{TAB} concludes the proof.
\end{proof}

\noindent \textbf{Remark:} It is important to mention that we believe that the restriction $p<4s+1$ we impose in order to apply the Gagliardo--Nirenberg--Sobolev inequality is far from optimal. For instance, in the following subsection where we deal with a priori bounds for the fractional Laplacian, this condition no longer applies. We conjecture that the conclusions of Theorems \ref{TAB} and \ref{LAB} still hold without the condition $p<4s+1$.

\subsection{A priori bounds for the fractional Laplacian}

In this subsection, we prove a priori bounds for the fractional Laplacian $\mc{L}\equiv (-\D)^{s}$, that is, when the multiplier satisfies $\mf{m}\equiv 1$. In this particular case, to prove a priori bounds, we use the blowing-up argument of Gidas--Spruck \cite{GS}. In this way, the restriction  $p<4s+1$ we must to impose in the later subsection in order to apply the Gagliardo--Nirenberg--Sobolev inequality, no longer applies. However, it is important to note that for the application of this technique, the operator $\mc{L}$ defined at first for periodic functions, must have a natural extension to functions defined in the whole real line $\R$. In the case of the fractional Laplacian, the representation
$$
(-\D)^{s}u(x) := C(s)\int_{\R}\frac{u(x)-u(y)}{|x-y|^{1+2s}}\; dy, \;\; x\in\R,
$$
makes the job. Let $\tfrac{1}{2}\leq s<1$, $p\geq 2$ and consider the problem
\begin{equation}
	\label{FLE}
	(-\D)^{s}u = \l u + |u|^{p}, \quad x\in\mathbb{T}, \; \; \l\geq0.
\end{equation}
The next result establishes a priori bounds from below for solutions of equation \eqref{FLE}.
\begin{lemma}
	\label{L7}
	Every solution $(\l,u)\in \R_{\geq 0}\times H^{2s}(\mathbb{T})$ of \eqref{FLE} satisfies
	\begin{equation}
		\label{BFB}
	u(x) \geq  -\l^{\frac{1}{p-1}}, \quad x\in\mathbb{T}.
	\end{equation}
\end{lemma}
\begin{proof}
	Let $(\l,u)\in \R_{\geq 0}\times H^{2s}(\mathbb{T})$ be a solution of \eqref{FLE} and $x_0\in\mathbb{T}$ be the absolute minimum of $u$. Then, it holds that $u(x_0)-u(y)\leq 0$ for each $y\in\mathbb{T}$ and consequently,
	$$
	(-\D)^{s}u(x_0)= C(s) \int_{\R}\frac{u(x_0)-u(y)}{|x_{0}-y|^{1+2s}}\leq 0.
	$$
	Therefore, evaluating the equation \eqref{FLE} in $x=x_0$, we deduce
	$
	\l u(x_0)+|u(x_0)|^{p}\leq 0
	$.
	From this we deduce \eqref{BFB}. The proof is concluded.
\end{proof}

Subsequently, for any bounded subset $A\subset \R$, we denote by $\mathscr{S}_{A}$ the set consisted on the solutions $(\l,u)\in A\times H^{2s}(\mathbb{T})$ of \eqref{FLE}. The main result concerning a priori bounds for solutions of \eqref{FLE} is the following. It's proof is an adaptation of Corollary 1 of Section 7 of Barrios, García-Melián and Quaas \cite{BGMQ} which prove the result for the case $\l=1$.

\begin{theorem}
	\label{APB}
	For every compact subset $K\subset \R_{\geq 0}$, there exists $M>0$ such that 
	$$
	\sup_{(\l,u)\in\mathscr{S}_{K}}\|u\|_{L^{\infty}}\leq M.
	$$
\end{theorem}

\begin{proof}
	Choose a compact subset $K\subset \R_{\geq 0}$ and suppose the statement is false for $K$. Then, there exists a sequence $\{(\l_n,u_n)\}_{n\in\N}\subset \mathscr{S}_{K}$ such that
	$$
	M_{n}:=\|u_{n}\|_{L^{\infty}}\to +\infty \;\; \text{as} \;\; n\to+\infty.
	$$
	By Lemma \ref{L7}, after discarding finitely many terms, $M_n=\max u_n$. Let $\{x_n\}_{n\in\N}\subset \mathbb{T}$ such that $u_n(x_n)=M_n$ and take $\l_{\ast}\in K$ satisfying
	$$
	\l_{\ast}^{\frac{1}{p-1}}=\max\{\l^{\frac{1}{p-1}} \; : \; \l\in K\}.
	$$
	Then, by Lemma \ref{L7}, we have that
	$$
	u_{n}(x)\geq -\l_{n}^{\frac{1}{p-1}}\geq -\l_{\ast}^{\frac{1}{p-1}}, \quad x\in\mathbb{T}, \; n\in\N.
	$$
	By compactness, passing to a suitable subsequence, we can suppose that
	$$
	\lim_{n\to+\infty}\l_n = \l_0\in K, \quad \lim_{n\to+\infty} x_n = x_0\in\mathbb{T}.
	$$
	Extending $u_n$ periodically to $\R$, we perform the following scaling change of variables.
	$$
	v_{n}(M_n^{\frac{p-1}{2s}}(x-x_n))=M_n^{-1}u_n(x), \quad x\in\R,
	$$
	or equivalently
	$$
	v_n(x) =  M_n^{-1} u_n(x_n+M_n^{-\frac{p-1}{2s}}x), \quad x\in \R. 
	$$
	Note that 
	$$
	-M_{n}^{-1}\l_{\ast}^{\frac{1}{p-1}}\leq v_n(y)\leq v_n(0) = M_n^{-1} u_n(x_n)=M_n^{-1}M_n=1, \quad y\in\R.
	$$
    The period of $v_n$ is $2\pi M_n^{\frac{p-1}{2s}}$, which tends to infinity.

	Now, we compute the fractional Laplacian of the function $v_{n}$. Observe that by making the change of variables $x_n+M_{n}^{-\frac{p-1}{2s}}y \mapsto z$, we get
	\begin{align*}
		(-\D)^{s}v_n(x) & = C(s)M_{n}^{-1}\int_{\R}\frac{u_{n}(x_n+M_n^{-\frac{p-1}{2s}}x)-u_n(x_n+M_n^{-\frac{p-1}{2s}}y)}{|x-y|^{1+2s}}\; dy \\
		& = C(s) M_{n}^{-1+\frac{p-1}{2s}} \int_{\R}\frac{u_{n}(x_n+M_n^{-\frac{p-1}{2s}}x)-u_{n}(z)}{\big|x-M_{n}^{\frac{p-1}{2s}}z+M_{n}^{\frac{p-1}{2s}}x_{n}\big|^{1+2s}}\; dz \\
		& = C(s) M_{n}^{-1+\frac{p-1}{2s}-\frac{p-1}{2s}(1+2s)}\int_{\R}\frac{u_{n}(x_n+M_n^{-\frac{p-1}{2s}}x)-u_{n}(z)}{\big|x_{n}+M_{n}^{-\frac{p-1}{2s}}x-z\big|^{1+2s}}\; dz \\
		& = M_{n}^{-p} [(-\D)^{s}u_{n}](x_n+M_n^{-\frac{p-1}{2s}}x).
	\end{align*}
	Therefore, we deduce from the original equation \eqref{FLE} that $v_n$ must satisfy the problem
	\begin{equation*}
		\left\{
		\begin{array}{ll}
			(-\D)^{s}v_{n} = \l_{n}M_{n}^{-p+1}v_{n}+|v_{n}|^{p},  & x\in\R, \\
			-M_{n}^{-1}\l_{\ast}^{\frac{1}{p-1}}\leq v_{n}(x)\leq v_{n}(0)=1, & x\in\R.
		\end{array}
		\right.
	\end{equation*}
	The uniform bounds and the interior regularity estimates of Silvestre \cite{Si} and Caffarelli--Silvestre \cite{CS} allow us to extract a subsequence converging locally in $\mc{C}^{2s+\alpha}$ for some $\alpha>0$, with $2s+\alpha\notin\N$. The tails of the integral are controlled by the uniform bound $|v_n|\leq1$. Passing to the limit, we obtain a function $v$ such that 
	\begin{equation*}
		\left\{
		\begin{array}{ll}
			(-\D)^{s}v = v^{p},  & x\in\mathbb{R}, \\
			0\leq v(x)\leq v(0)=1, & x\in\mathbb{R}.
		\end{array}
		\right.
	\end{equation*}
	By Corollary \ref{Cr6.2}, after discarding finitely many terms, we may assume that $\l_n>0$ for each $n\in\N$. Setting $a_n:=\l_n M_n^{1-p}$ and arguing as in the preceding change of variables, we deduce that the periodic function
	$$
	w_n(x):=a_n^{-\frac{1}{p-1}}v_n\left(a_n^{-\frac{1}{2s}}x\right), \quad x\in\R,
	$$
	satisfies $(-\D)^s w_n=w_n+|w_n|^p$ in $\R$. Hence, by \cite[Sec. 7, Cor. 1]{BGMQ}, there exists a constant $C>0$, independent of the period of $w_n$, such that $\|w_n\|_{L^{\infty}(\R)}\leq C$. Since $K$ is bounded and $M_n\to+\infty$, we obtain
	$$
	\|v_n\|_{L^{\infty}(\R)}=a_n^{\frac{1}{p-1}}\|w_n\|_{L^{\infty}(\R)}\leq C a_n^{\frac{1}{p-1}}\to 0 \quad \text{as} \quad n\to+\infty,
	$$
	which contradicts $v(0)=1$ and concludes the proof.
\end{proof}

The final result of this section establishes $H^{2s}$-a priori bounds for the fractional Laplacian.

\begin{theorem}
	\label{Th6.10}
	For every compact subset $K\subset \R_{\geq0}$, there exists $M>0$ such that 
	$$
	\sup_{(\l,u)\in\mathscr{S}_{K}}\|u\|_{H^{2s}}\leq M.
	$$
\end{theorem}

\begin{proof}
	Firstly, given $u\in H^{2s}(\mathbb{T})$, we compute the $L^{2}$-norm of the fractional Laplacian
	$$
	\|(-\D)^{s}u\|_{L^{2}}^{2} = 2\pi \sum_{n\in\Z} |n|^{4s}|\widehat{u}(n)|^{2} = 2\pi \|u\|_{\dot{H}^{2s}}^{2}.
	$$
	Therefore, taking the $L^{2}$-norm in the equation \eqref{FLE}, we obtain
	\begin{align*}
		2\pi \|u\|_{\dot{H}^{2s}}^{2} & = \|\l u + |u|^{p}\|_{L^{2}}^{2} = |\l|^{2}\|u\|^{2}_{L^{2}}+2\l \int_{\mathbb{T}}|u|^{p}u+\int_{\mathbb{T}}|u|^{2p} \\
		& \leq \l^{2} \|u\|_{L^{2}}^{2}+2\l\|u\|^{p+1}_{L^{p+1}}+\|u\|^{2p}_{L^{2p}} \leq C(\l_{\ast}) (\|u\|_{L^{\infty}}^{2}+\|u\|^{p+1}_{L^{\infty}}+\|u\|^{2p}_{L^{\infty}}),
	\end{align*}
	where  $C(\l_{\ast})>0$ is a positive constant depending on $\l_{\ast}:=\sup K = \max K$.
	Now, an application of Theorem \ref{APB} concludes the proof.
\end{proof}

\section{Global bifurcation theory}\label{S7}

This section is devoted to study the global structure of the connected components 
$$\mathscr{C}_{k}\subset \R\times H^{2s}_{\mathbf{+}}(\mathbb{T}), \quad k\geq 1,$$
of nontrivial solutions emanating from the bifurcation points $(k^{2s}\mf{m}(k),0)\in\R\times H^{2s}_{\mathbf{+}}(\mathbb{T})$. More precisely, we will use the results obtained in the preceding sections to apply the global alternative Theorem \ref{TGB} and prove Theorems \ref{Th1.1} and \ref{Th1.2}. 
\par We start by proving that the components $\mathscr{C}_{k}$ and $\mathscr{C}_{0}$ are disjoint.
\begin{lemma}
	\label{Lema7.1}
	For each $k\in\N$, $\mathscr{C}_{0}\cap\mathscr{C}_{k}=\emptyset$.
\end{lemma}
\begin{proof}
	Suppose that for some $k\in \N$, $\mathscr{C}_{k}\cap \mathscr{C}_{0}\neq \emptyset$. Then, as they are connected components, necessarily, $\mathscr{C}_{k}=\mathscr{C}_{0}$. Hence, $(k^{2s}\mf{m}(k),0)\in\mathscr{C}_{0}$. But this contradicts Proposition \ref{Le7.1}. The proof is concluded.
\end{proof}
The next result shows that the connected components $\mathscr{C}_{k}$, $k\geq 1$, live in $\R_{>0}\times H^{2s}_{+}(\mathbb{T})$.
\begin{lemma}
	\label{L7.2}
	For each $k\in\N$, the following set inclusion holds:
	\begin{equation}
		\mathscr{C}_{k}\subset \R_{>0}\times H^{2s}_{+}(\mathbb{T}).
	\end{equation}
\end{lemma}
\begin{proof}
	Indeed, if this is not true, as $(0,0)\notin \mathscr{C}_{k}$, there must exists $u_0\in H^{2s}_{+}(\mathbb{T})$, $u_0\not\equiv 0$, such that $(0,u_0)\in\mathscr{C}_{k}$. However, this contradicts Corollary \ref{Cr6.2}.
\end{proof}
The next result is the key to prove Theorem \ref{Th1.1}.
\begin{theorem}
	\label{Th7.1}
	Suppose that $p<4s+1$. Then, 
	\begin{equation}
		\label{Pro}
	(\mf{m}(1),2^{2s}\mf{m}(2))\subset \mc{P}_{\l}(\mathscr{C}_{1}),
   \end{equation}
	where $\mc{P}_{\l}:\R\times H^{2s}_{+}(\mathbb{T})\to \R$, $(\l,u)\mapsto \l$, is the $\l$-projection operator.
\end{theorem}

	\begin{proof}
		We apply Theorem \ref{TGB} to the nonlinearity 
		\begin{equation}
			\label{DO1}
			\mf{F}:\R\times H_{+}^{2s}(\mathbb{T})\longrightarrow L_{+}^{2}(\mathbb{T}), \quad \mf{F}(\l,u)=\mc{L}u-\l u -  |u|^{p}.
		\end{equation}
		We proceed to verify the hypothesis of Theorem \ref{TGB}. On the one hand, $\mf{F}$ is orientable in the sense of Fitzpatrick, Pejsachowicz and Rabier since the domain $\R\times H^{2s}_{+}(\mathbb{T})$ is simply connected. Hypothesis {\rm{(F2)}} holds by the very definition of \eqref{DO1}. Hypothesis {\rm(F3)} follows from Proposition \ref{Eqq} and {\rm (F4)} by Proposition \ref{LF4}. Hypothesis {\rm(F5)} follows from the fact that $\Sigma(\mf{L})=\{\s_{k}\}_{k=0}^{\infty}$, where
		$$
		\s_{0}=0, \quad \s_{k}=k^{2s}\mf{m}(k), \;\; k\in\N,
		$$
		proved in Proposition \ref{Pr3.2}. On the other hand, the linearization
		$$
		\mf{L}:\R\longrightarrow\Phi_{0}(H^{2s}_{+}(\mathbb{T}),L^{2}_{+}(\mathbb{T})), 
		\quad \mf{L}(\l)[v] = \mc{L}v - \l v,
		$$
		is clearly analytic and $\chi[\mf{L}, \s_{k}]=1\in 2\N-1$ for each $k\in\N\cup\{0\}$ by Proposition \ref{Pr3.3}. Then, the application of Theorem \ref{TGB} to the connected component $\mathscr{C}_{1}$ implies that or $\mathscr{C}_{1}$ is unbounded or there exists $\s_{m}\in \Sigma(\mf{L})$, $m\neq 1$, such that $(\s_{m},0)\in\mathscr{C}_{1}$. If the second statement holds with $m>1$, then \eqref{Pro} holds trivially. Indeed, in this case, $(\s_1,0), (\s_m,0)\in \mathscr{C}_{1}$ for some $m>1$. Since $\mathscr{C}_{1}$ is connected and $\mc{P}_{\l}$ is continuous, we infer that $\mc{P}_{\l}(\mathscr{C}_{1})$ is a connected subset of $\R$. Therefore, $\mc{P}_{\l}(\mathscr{C}_{1})=I_{\a,\b}$, where $I_{\a,\b}$ is an interval of $\R$ with boundary $\{\a,\b\}$ for some $\alpha\leq \beta$. Moreover, $\alpha \leq \s_{1} < \s_{2} \leq \s_{m} \leq \beta$. Therefore, $(\s_{1},\s_{2})\subset I_{\a,\b}= \mc{P}_{\l}(\mathscr{C}_{1})$ and \eqref{Pro} is proven.
		\par Suppose that $m=0$. This implies that $\mathscr{C}_{0}=\mathscr{C}_{1}$ and this cannot happen by Lemma \ref{Lema7.1}. On the other hand, suppose that $\mathscr{C}_{1}$ is unbounded and \eqref{Pro} does not hold. Then, there exist $0\leq\a \leq \mf{m}(1)$ and $\b\in [\mf{m}(1),2^{2s}\mf{m}(2))$, such that 
		$$
		\mc{P}_{\l}(\mathscr{C}_{1}) = I_{\a,\b},
		$$
		where $I_{\a,\b}$ is an interval of $\R$ with boundary $\{\a,\b\}$. The existence and the non-negativity of $\a$ is justified by Lemma \ref{L7.2}. In any case, by the unboundedness of $\mathscr{C}_{1}$, there exists a sequence $\{(\l_n,u_n)\}_{n\in\N}\subset \mathscr{C}_{1}$ such that $\l_n\to\l_{0}\in \overline{I_{\a,\b}}$ as $n\to+\infty$ and
		$$
		\lim_{n\to+\infty}\|u_n\|_{H^{2s}}=+\infty.
		$$
		But, as $p<4s+1$, this contradicts Theorem \ref{TAB}. This concludes the proof.
	\end{proof}

Thanks to these results we can prove the first main theorem of this paper.

\begin{theorem}
	Suppose that $s\geq \tfrac{1}{2}$, $2\leq p<4s+1$ and the multiplier $\mf{m}$ satisfies hypothesis {\rm{(M1)--(M3)}}. Then, the pseudo-differential equation
	$$
	\mc{L}u = \l u + |u|^p, \quad x\in\mathbb{T}, \;\; u\in H^{2s}_{+}(\mathbb{T}),
	$$
	admits at least one non-constant even solution for every $\l\in (\mf{m}(1),2^{2s}\mf{m}(2))$.
\end{theorem}

\begin{proof}
	The existence of a solution $u_{\l}\in H^{2s}_{+}(\mathbb{T})$ for each $\l\in (\mf{m}(1),2^{2s}\mf{m}(2))$ was proven in Theorem \ref{Th7.1}. Moreover, we proved that
	$$\{(\l,u_{\l})  :  \l\in (\mf{m}(1),2^{2s}\mf{m}(2))\}\subset \mathscr{C}_{1}.$$
	It remains to show that these solutions are non-constant. If there exists $\l_{\ast}\in (\mf{m}(1),2^{2s}\mf{m}(2))$ such that $u_{\l_{\ast}}$ is constant, by the uniqueness part of Proposition \ref{Pr5.1}, we must have one of the following alternatives:
	\begin{itemize}
	\item[{\rm (a)}] $u_{\l_{\ast}}(x) = 0$,  $x\in\mathbb{T}$.
	\item[{\rm (b)}] $u_{\l_{\ast}}(x) = -\l_{\ast}^{\frac{1}{p-1}}$,  $x\in\mathbb{T}$.
    \end{itemize}
	Item (a) cannot happen as $\mathscr{C}_{1}\cap\mc{T}\subset\{(\s_{k},0)\}_{k\in \N\cup\{0\}}$ and $\l_{\ast}\in(\s_{1},\s_{2})$. If item (b) holds, then, we would have that $\mathscr{C}_{1}\cap\mathscr{C}_{0}\neq\emptyset$. This contradicts Lemma \ref{Lema7.1}. The proof is concluded.
\end{proof}

Now, we state the same results for the fractional Laplacian $\mc{L}\equiv (-\D)^{s}$, that is, when the multiplier $\mf{m}$ is identically $1$. In this case, by Theorem \ref{Th6.10}, the condition $p<4s+1$ is redundant.

\begin{theorem}
	\label{Th7.5}
	Suppose that $\tfrac{1}{2}\leq s<1$ and $p\geq 2$. Then, the pseudo-differential equation
	$$
	(-\D)^{s}u = \l u + |u|^p, \quad x\in\mathbb{T}, \;\; u\in H^{2s}_{+}(\mathbb{T}),
	$$
	admits at least one non-constant even solution for every $\l\in (1,2^{2s})$.
\end{theorem}

Finally, we proceed to study the particular case $p=2$. The next symmetry result will simplify the structure of the solution set for $\l<0$.

\begin{proposition}
	\label{P5.3}
	Suppose that $p=2$ and let $\mf{C}_{\pm}\subset \R_{\pm}\times H^{2s}_{\mathbf{+}}(\mathbb{T})$ be the subsets of solutions defined by 
	$$\mf{C}_{+}:= \mf{F}^{-1}(0)\cap \{(\l,u) : \l>0 \}, \quad \mf{C}_{-}:= \mf{F}^{-1}(0)\cap \{(\l,u) : \l<0 \}.$$
	Then, the affine operator
	$$
	T: \mf{C}_{-} \longrightarrow \mf{C}_{+}, \quad T(\l,u) = (-\l, u+\l),
	$$
	is a homeomorphism. In particular, if $u$ is a solution of $\mc{L}u= \l u +  u^{2}$ with $\l<0$, then $v:=u+\l$ is a solution of $\mc{L}v= -\l v +  v^{2}$.
\end{proposition}
\begin{proof}
	Let $(\l,u)\in \mf{C}_{-}$, then setting $v = u+\l$, we obtain that
	\begin{align*}
		\mc{L}v = \l (v-\l) + (v-\l)^{2}.
	\end{align*}
	That is, $v$ satisfies the equation $\mc{L}v = -\l v + v^{2}$. This proves that $T(\mf{C}_{-})\subset \mf{C}_{+}$. The fact that $T$ is an homeomorphism follows immediately by standard computations by noting that the inverse of $T$ is given by $T^{-1}(\l,u)=(-\l,u+\l)$.
\end{proof}

Thanks to this proposition, the following result follows immediately. Hence, we omit the proof.

\begin{theorem}
	Suppose that $s\geq \tfrac{1}{2}$ and the multiplier $\mf{m}$ satisfies hypothesis {\rm{(M1)--(M3)}}. Then, the pseudo-differential equation
	$$
	\mc{L}u = \l u + |u|^2, \quad x\in\mathbb{T}, \;\; u\in H^{2s}_{+}(\mathbb{T}),
	$$
	admits at least one non-constant even solution for every $\l\in (-2^{2s}\mf{m}(2),-\mf{m}(1))\cup (\mf{m}(1),2^{2s}\mf{m}(2)) $.
\end{theorem}

Figure \ref{F2} provides an illustration of the action of the affine operator $T$. Recall that we are representing the value of the parameter $\l$ in abscissas versus the norm $\|u\|_{H^{2s}}$ or $-\|u\|_{H^{2s}}$. This differentiation is made in order to express the multiplicity of solutions.

\begin{figure}[h!]

	\tikzset{every picture/.style={line width=0.75pt}} 
	
	\begin{tikzpicture}[x=0.75pt,y=0.75pt,yscale=-1,xscale=1]
		
		\draw    (260,39) -- (260.84,221.76) ;
		\draw [color={rgb, 255:red, 74; green, 144; blue, 226 }  ,draw opacity=1 ][line width=1.5]    (108,145) -- (444.91,144.58) ;
		\draw [color={rgb, 255:red, 126; green, 211; blue, 33 }  ,draw opacity=1 ][line width=1.5]    (150.49,58.06) -- (338.84,206.76) ;
		\draw  [color={rgb, 255:red, 126; green, 211; blue, 33 }  ,draw opacity=1 ][line width=1.5]  (373.22,108.79) .. controls (296.83,133.27) and (297.18,156.61) .. (374.28,178.78) ;
		\draw  [color={rgb, 255:red, 126; green, 211; blue, 33 }  ,draw opacity=1 ][line width=1.5]  (136.32,88.96) .. controls (209.71,121.36) and (224.94,103.68) .. (182,35.91) ;
		\draw    (229,81) .. controls (259.54,77.06) and (278.43,88.64) .. (297.14,115.75) ;
		\draw [shift={(298,117)}, rotate = 235.84] [color={rgb, 255:red, 0; green, 0; blue, 0 }  ][line width=0.75]    (10.93,-3.29) .. controls (6.95,-1.4) and (3.31,-0.3) .. (0,0) .. controls (3.31,0.3) and (6.95,1.4) .. (10.93,3.29)   ;
		\draw  [color={rgb, 255:red, 126; green, 211; blue, 33 }  ,draw opacity=1 ][fill={rgb, 255:red, 126; green, 211; blue, 33 }  ,fill opacity=1 ] (313.69,144.65) .. controls (313.69,143.27) and (314.81,142.15) .. (316.19,142.15) .. controls (317.57,142.15) and (318.69,143.27) .. (318.69,144.65) .. controls (318.69,146.03) and (317.57,147.15) .. (316.19,147.15) .. controls (314.81,147.15) and (313.69,146.03) .. (313.69,144.65) -- cycle ;
		\draw  [dash pattern={on 0.84pt off 2.51pt}]  (202.78,100) -- (202.73,146.14) ;
		\draw  [color={rgb, 255:red, 126; green, 211; blue, 33 }  ,draw opacity=1 ][fill={rgb, 255:red, 126; green, 211; blue, 33 }  ,fill opacity=1 ] (200.28,100) .. controls (200.28,98.62) and (201.4,97.5) .. (202.78,97.5) .. controls (204.16,97.5) and (205.28,98.62) .. (205.28,100) .. controls (205.28,101.38) and (204.16,102.5) .. (202.78,102.5) .. controls (201.4,102.5) and (200.28,101.38) .. (200.28,100) -- cycle ;
		\draw  [color={rgb, 255:red, 126; green, 211; blue, 33 }  ,draw opacity=1 ][fill={rgb, 255:red, 126; green, 211; blue, 33 }  ,fill opacity=1 ] (258.28,145) .. controls (258.28,143.62) and (259.4,142.5) .. (260.78,142.5) .. controls (262.16,142.5) and (263.28,143.62) .. (263.28,145) .. controls (263.28,146.38) and (262.16,147.5) .. (260.78,147.5) .. controls (259.4,147.5) and (258.28,146.38) .. (258.28,145) -- cycle ;
		
		\draw (266.03,37.4) node [anchor=north west][inner sep=0.75pt]    {$\| \cdot \| _{H^{2s}}$};
		\draw (446.91,147.98) node [anchor=north west][inner sep=0.75pt]    {$\lambda $};
		\draw (305.04,202.08) node [anchor=north west][inner sep=0.75pt]    {$\mathscr{C}_{0}$};
		\draw (410,123.23) node [anchor=north west][inner sep=0.75pt]    {$\mathcal{T}$};
		\draw (329.04,90.08) node [anchor=north west][inner sep=0.75pt]    {$\mathscr{C}_{1}$};
		\draw (281,71.4) node [anchor=north west][inner sep=0.75pt]    {$T$};
		\draw (284.46,147.19) node [anchor=north west][inner sep=0.75pt]    {$\mathfrak{m}( 1)$};
		\draw (179.46,146.19) node [anchor=north west][inner sep=0.75pt]    {$-\mathfrak{m}( 1)$};

	\end{tikzpicture}
	\caption{Illustration of the connected component $\mathscr{C}_{0}$ and $\mathscr{C}_{1}$ and the action of the homeomorphism $T$}
	\label{F2}
\end{figure}
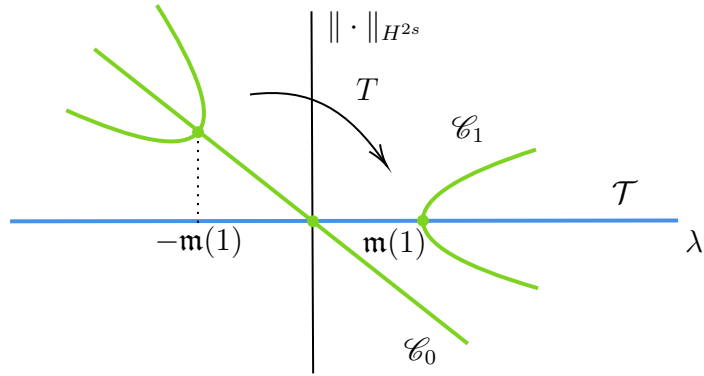

\subsection{Global bifurcation for the fractional laplacian.}

In this subsection, we sharpen the preceding results by exploiting the symmetry of the fractional Laplacian $\mc{L}\equiv (-\D)^{s}$.
We start with the following qualitative result:

\begin{lemma}
	\label{L7.1}
	If $u\in H^{2s}_{\mathbf{+}}(\mathbb{T})$ is a solution of
	$$
	(-\D)^{s}u = \l u + |u|^{p}, \quad x\in\mathbb{T}.
	$$
	Then, for each $k\in\N$, the function $v_{k}\in H^{2s}_{\mathbf{+}}(\mathbb{T})$ defined by 
	$$
	v_{k}(x):=k^{\frac{2s}{p-1}} u(k x), \quad x\in \mathbb{T},
	$$
	is a solution of
	$$
	(-\D)^{s}v_{k} = \l k^{2s} v_{k}+|v_k|^{p}, \quad x\in\mathbb{T}.
	$$
\end{lemma}

\begin{proof}
	Let $v\in H^{2s}_{\mathbf{+}}(\mathbb{T})$ defined by $v(x):=k^{\frac{2s}{p-1}} u(k x)$, $x\in \mathbb{T}$. We set $\widehat{u}(\xi)=0$ for $\xi\notin\Z$. Then, for each $n\in\Z$, we have
	\begin{align*}
		\widehat{v}_{k}(n) = k^{\frac{2s}{p-1}}\widehat{u}\left(\frac{n}{k}\right),
	\end{align*}
	and, consequently, $\widehat{v}_{k}(n)=0$ for each $n\in\Z$ such that $k\nmid n$. Therefore,
	\begin{align*}
		(-\D)^{s}v_{k}(x) & = \sum_{n\in\Z}|n|^{2s}\widehat{v}_{k}(n) e^{inx} = k^{\frac{2s}{p-1}} \sum_{n\in\Z}|n|^{2s}\widehat{u}\left(\frac{n}{k}\right) e^{inx} \\
		& = k^{\frac{2s}{p-1}}\sum_{n\in\Z}|kn|^{2s}\widehat{u}(n)e^{inkx}=k^{\frac{2sp}{p-1}}\sum_{n\in\Z}|n|^{2s}\widehat{u}(n)e^{inkx}=k^{\frac{2sp}{p-1}}(-\D)^{s}u(kx).
	\end{align*}
	Then, $v_{k}$ satisfies the equation
	\begin{align*}
		\frac{1}{k^{\frac{2sp}{p-1}}}(-\D)^{s}v_{k}=\frac{\l}{k^{\frac{2s}{p-1}}}v_{k}+\frac{1}{k^{\frac{2sp}{p-1}}}|v_{k}|^{p}, \quad x\in\mathbb{T},
	\end{align*}
	that is, $v_{k}$ satisfies $(-\D)^{s}v_{k} = \l k^{2s} v_{k}+|v_k|^{p}$. This concludes the proof.
\end{proof}

We can use Lemma \ref{L7.1} to construct certain continuous injections that will give us some qualitative information about the global structure of the connected components $\mathscr{C}_{k}$, $k\geq 1$. The following result constructs these continuous injections.

\begin{proposition}
	\label{Pr7.20}
	For each $k\in\N$, the operator
	$$
	T_{k}:\mathfrak{S}\longrightarrow\mathfrak{S}, \quad T_{k}(\l,u)=(k^{2s}\l , k^{\frac{2s}{p-1}} u(kx)),
	$$
	is a continuous injection.
\end{proposition}

\begin{proof}
	First of all, let us prove that $T_{k}:\R\times H^{2s}_{\mathbf{+}}(\mathbb{T})\to \R\times H^{2s}_{\mathbf{+}}(\mathbb{T})$ is continuous. Let $\{(\l_{n},u_{n})\}_{n\in\N}\subset \R\times H^{2s}_{\mathbf{+}}(\mathbb{T})$ and $(\l_{0},u_{0})\in \R\times H^{2s}_{\mathbf{+}}(\mathbb{T})$ such that 
	$$
	\lim_{n\to+\infty}(\l_{n},u_{n}) = (\l_{0},u_{0}), \quad \text{in} \;\; \R\times H^{2s}_{\mathbf{+}}(\mathbb{T}).
	$$
	Let us define $v_{n}(x):=k^{\frac{2s}{p-1}}u_{n}(kx)$, $x\in\mathbb{T}$, and $v_{0}(x):=k^{\frac{2s}{p-1}}u_{0}(kx)$, $x\in\mathbb{T}$. We  must prove that 
	$$
	\lim_{n\to+\infty} v_{n} = v_{0} \quad \text{in} \;\; H^{2s}_{\mathbf{+}}(\mathbb{T}).
	$$
	Clearly $v_{n},v_{0}\in H^{2s}_{\mathbf{+}}(\mathbb{T})$.
	On the one hand, the change of variable $y = kx$ and the periodicity of $u_{n}$ yield
	\begin{align*}
		\|v_{n}-v_{0}\|_{L^{2}}^{2} & = k^{\frac{4s}{p-1}}\int_{-\pi}^{\pi} |u_{n}(kx)-u_{0}(kx)|^{2} \;{\rm{d}}x = k^{\frac{4s}{p-1}-1} \int_{-k\pi}^{k\pi} |u_{n}(y)-u_{0}(y)|^{2} \; {\rm{d}}y \\
		& =  k^{\frac{4s}{p-1}-1} \int_{0}^{2k\pi} |u_{n}(y)-u_{0}(y)|^{2} \; {\rm{d}}y =  k^{\frac{4s}{p-1}-1} \sum_{j=0}^{k-1}\int_{2j\pi}^{2(j+1)\pi} |u_{n}(y)-u_{0}(y)|^{2} \; {\rm{d}}y \\
		& = k^{\frac{4s}{p-1}}\int_{-\pi}^{\pi}|u_{n}(y)-u_{0}(y)|^{2} \; {\rm{d}}y = k^{\frac{4s}{p-1}} \|u_{n}-u_{0}\|_{L^{2}}^{2}.
	\end{align*}
	On the other hand, arguing as in the proof of Lemma \ref{L7.1}, we obtain
	\begin{align*}
		\|v_{n}-v_{0}\|^{2}_{\dot{H}^{2s}}&=\sum_{j\in\Z}|j|^{4s}|\widehat{v}_{n}(j)-\widehat{v}_{0}(j)|^{2}=k^{\frac{4s}{p-1}}\sum_{j\in\Z}|j|^{4s}\Big|\widehat{u}_{n}\left(\frac{j}{k}\right)-\widehat{u}_{0}\left(\frac{j}{k}\right)\Big|^{2}\\
		&=k^{\frac{4s}{p-1}}\sum_{j\in\Z}|kj|^{4s}|\widehat{u}_{n}(j)-\widehat{u}_{0}(j)|^{2} = k^{\frac{4sp}{p-1}}\sum_{j\in\Z}|j|^{4s}|\widehat{u}_{n}(j)-\widehat{u}_{0}(j)|^{2} \\
		&= k^{\frac{4sp}{p-1}}\|u_{n}-u_{0}\|_{\dot{H}^{2s}}^{2}
	\end{align*}
	Therefore,
	\begin{align*}
		\lim_{n\to+\infty}\|v_{n}-v_{0}\|^{2}_{H^{2s}} = k^{\frac{4s}{p-1}} \lim_{n\to+\infty} \|u_{n}-u_{0}\|^{2}_{L^{2}} + k^{\frac{4sp}{p-1}}\lim_{n\to+\infty} \|u_{n}-u_{0}\|^{2}_{\dot{H}^{2s}} = 0.
	\end{align*}
	This proves the continuity of $T_{k}$.
	\par
	Let us prove that $T_{k}$ is injective. Let $(\l_{1},u_{1}), (\l_{2},u_{2})\in \R\times H^{2s}_{\mathbf{+}}(\mathbb{T})$ such that 
	$$
	(k^{2s}\l_{1}, k^{\frac{2s}{p-1}}u_{1}(kx)) = (k^{2s}\l_{2}, k^{\frac{2s}{p-1}}u_{2}(kx)), \quad x\in\mathbb{T}.
	$$
	Then, $\l_{1}=\l_{2}$ and $u_{1}(kx)=u_{2}(kx)$ for all $x\in\mathbb{T}$. This implies that
	$$
	u_1(x) = u_2(x), \quad \text{for each} \;\; x\in[-k\pi,k\pi].
	$$
	Hence $u_{1}(x)=u_{2}(x)$ for all $x\in\mathbb{T}$ and $T_{k}$ is injective. Finally, from Lemma \ref{L7.1}, we have $T_{k}(\mathfrak{S})\subset\mathfrak{S}$. This concludes the proof.
\end{proof}

In the following result we prove that for each $k\in\N$, the action of the continuous injections $T_{k}$ can be restricted to the connected component $\mathscr{C}_{1}$.

\begin{proposition}
	\label{Pr7.2}
	For each $k\in\N$, the operator
	$$
	T_{k}:\mathscr{C}_{1}\longrightarrow\mathscr{C}_{k}, \quad T_{k}(\l,u)=(k^{2s}\l , k^{\frac{2s}{p-1}} u(kx)),
	$$
	is a continuous injection. 
\end{proposition}

\begin{proof}
	By Proposition \ref{Pr7.20}, $T_{k}:\mathfrak{S}\to\mathfrak{S}$ is a continuous injection. It remains to prove that $T_{k}(\mathscr{C}_{1})\subset \mathscr{C}_{k}$. The connected components 
	$$\mathscr{C}_{n}\subset \R\times H^{2s}_{\mathbf{+}}(\mathbb{T}), \quad n\in\N,$$
	of nontrivial solutions emanate from the bifurcation points $(n^{2s},0)\in\R\times H^{2s}_{\mathbf{+}}(\mathbb{T})$. Moreover, by Theorem \ref{LCRWw}, for each $n\in\N$, there exist $\e>0$ and two $\mc{C}^{\o(p)-2}$-maps
	\begin{equation*}
		\L_{n}: (-\e,\e) \longrightarrow \R, \quad \L_{n}(0)=n^{2s}, \qquad \Gamma_{n}: (-\e,\e)\longrightarrow Y_{n}, \quad  \Gamma_{n}(0)=0,
	\end{equation*}
	such that for each $s\in(-\e,\e)$,
	\begin{equation*}
		\mathfrak{F}(\L_{n}(s),u_{n}(s))=0, \quad u_{n}(s):= s(\cos(nx)+\Gamma_{n}(s)).
	\end{equation*}
	In other words, for some $\r>0$,
	\begin{equation}
		\label{Eq3}
	\mathscr{C}_{n}\cap B_{\r}(n^{2s},0)=\{(\L_{n}(s), s(\cos(nx)+\Gamma_{n}(s))) \, : \, s\in(-\varepsilon,\varepsilon)\}\cap B_\rho(n^{2s},0).
\end{equation}
	Let us prove that the following intersection is nonempty.
	\begin{equation}
		\label{Eq5}
	T_k(\mathscr{C}_1)\cap\mathscr{C}_k\neq\emptyset.
	\end{equation}
	Indeed, $(1,0)\in\mathscr{C}_1$ and $T_k(1,0)=(k^{2s},0)\in\mathscr{C}_k$, which proves \eqref{Eq5}.
	\par Now, as $\mathscr{C}_{1}$ is connected and $T_{k}$ is continuous, $T_{k}(\mathscr{C}_{1})\subset\mf{S}$ is necessarily connected. Moreover, by \eqref{Eq5}, we have $T_{k}(\mathscr{C}_{1})\cap\mathscr{C}_{k}\neq \emptyset$. As $\mathscr{C}_{k}$ is a connected component of $\mathfrak{S}$, necessarily $T_{k}(\mathscr{C}_{1})\subset \mathscr{C}_{k}$. This concludes the proof.
\end{proof}

The following result is the key to understand the global structure of the connected component $\mathscr{C}_{1}$.

\begin{proposition}
	\label{Pr4.6}
	If $\mathscr{C}_{1}\cap\mathscr{C}_{k}\neq\emptyset$ for some $k\geq2$, then $\mathscr{C}_{1}=\mathscr{C}_{k}$ is unbounded.
\end{proposition}

\begin{proof}
	Suppose that there exists $k\geq2$, such that $\mathscr{C}_{1}\cap\mathscr{C}_{k}\neq \emptyset$. Then, as they are connected components of $\mathfrak{S}$, necessarily $\mathscr{C}_{1}=\mathscr{C}_{k}$. By Proposition \ref{Pr7.2}, the map 
	$$
	T_{k}:\mathscr{C}_{1}\longrightarrow\mathscr{C}_{k}, \quad T_{k}(\l,u) = (k^{2s}\l, k^{\frac{2s}{p-1}} u(kx)),
	$$
	is a continuous injection. We apply the action of $T_{k}$ several times on $\mathscr{C}_{1}$ as the following diagram shows:
	\begin{equation}
		\mathscr{C}_{1}\xrightarrow{T_{k}}\mathscr{C}_{1}\xrightarrow{T_{k}}\mathscr{C}_{1}\xrightarrow{T_{k}}\mathscr{C}_{1}\xrightarrow{T_{k}} \cdots
	\end{equation}
	This implies that for each $(\l,u)\in\mathscr{C}_{1}$, 
	$$
	(k^{2sq}\l, k^{\frac{2sq}{p-1}} u (k^{q}x))\in \mathscr{C}_{1}, \quad \hbox{for all} \;\; q\in\N.
	$$
	In particular, as $(1,0)\in\mathscr{C}_{1}$, we have $(k^{2sq},0)\in\mathscr{C}_{1}$ for all $q\in\N$. As 
	$$
	\lim_{q\to+\infty}k^{2sq}=+\infty,
	$$
	we conclude that $\mathscr{C}_{1}$ is unbounded.
\end{proof}

The next result extends the existence Theorem \ref{Th7.5} up to cover more range of values of the parameter $\l$. The main ingredient in its proof is Proposition \ref{Pr4.6}.

\begin{theorem}
	\label{TP0}
	Suppose that $\tfrac{1}{2}\leq s<1$ and $p\geq 2$. The equation 
		\begin{equation}
			\label{EE0}
			(-\D)^{s}u = \l u + |u|^{p}, \quad u\in H^{2s}_{\mathbf{+}}(\mathbb{T}),
		\end{equation}
		admits at least one non-constant even solution for each
		$$
		\l\in \bigcup_{k\in\N}\left(k^{2s}, (k+1)^{2s}\right).
		$$
\end{theorem}

\begin{proof}
	We apply Theorem \ref{TGB} to the nonlinearity 
	\begin{equation}
		\label{DO}
		\mf{F}:\R\times H_{+}^{2s}(\mathbb{T})\longrightarrow L_{+}^{2}(\mathbb{T}), \quad \mf{F}(\l,u)=(-\D)^{s}u-\l u -  |u|^{p}.
	\end{equation}
	By the same arguments used in the proof of Theorem \ref{Th7.1}, the hypothesis on $\mf{F}$ of Theorem \ref{TGB} are verified. On the other hand, the linearization
	$$
	\mf{L}:\R\longrightarrow\Phi_{0}(H^{2s}_{+}(\mathbb{T}),L^{2}_{+}(\mathbb{T})), 
	\quad \mf{L}(\l)[v] = (-\D)^{s}v - \l v,
	$$
	is clearly analytic and $\chi[\mf{L}, \s_{k}]=1\in 2\N-1$ by Proposition \ref{Pr3.3}. The application of Theorem \ref{TGB} to the connected component $\mathscr{C}_{1}$, implies that or $\mathscr{C}_{1}$ is unbounded or there exists $\s_{k}\in \Sigma(\mf{L})$, $k\neq 1$, such that $(\s_{k},0)\in\mathscr{C}_{1}$. By Lemma \ref{Lema7.1}, necessarily $k>1$. If the second statement holds, then $\mathscr{C}_{1}=\mathscr{C}_{k}$ and by Proposition \ref{Pr4.6}, $\mathscr{C}_{1}$ is necessarily unbounded. Therefore, in any case, the connected component $\mathscr{C}_{1}$ is unbounded. To prove the result, we start by showing that
	\begin{equation}
		\label{Eq4.5}
		[\s_{1},+\infty)\subset\mc{P}_{\l}(\mathscr{C}_{1}),
	\end{equation}
	where $\mc{P}_{\l}:\R\times H^{2s}(\mathbb{T})\to \R$ is the $\l$-projection operator given by $\mc{P}_{\l}(\l,u)=\l$. 
	Suppose that \eqref{Eq4.5} is false. Then, there exist $0\leq\a\leq \s_{1}$ and $\b \geq \s_1$, such that 
	\begin{equation}
		\label{PYS}
		\mc{P}_{\l}(\mathscr{C}_{1})= I_{\a,\b},
	\end{equation}
	where $I_{\a,\b}$ is an interval of the half real line $\R_{\geq 0}$ with boundary $\{\a,\b\}$. The existence and the non-negativity of $\a$ is justified by Lemma \ref{L7.2}. Choosing $K:=[\a,\b]\subset \R_{\geq0}$ in Theorem \ref{Th6.10}, we have that $\mathscr{C}_{1}$ is bounded on $[\a,\b]\times H^{2s}_{+}(\mathbb{T})$. Hence, by \eqref{PYS}, $\mathscr{C}_{1}$ is bounded in $\R\times H^{2s}_{+}(\mathbb{T})$. This contradicts the unboundedness of $\mathscr{C}_{1}$ and concludes the proof of \eqref{Eq4.5}. 
	\par Now, let $k\in \N$. Since the inclusion
	$
	(\s_{k},\s_{k+1}) \subset [\s_{1},+\infty)\subset\mc{P}_{\l}(\mathscr{C}_{1})
	$ holds,
	we infer that for each $\l\in (\s_{k},\s_{k+1})$, there exists $u_{\l}\in H^{2s}_{+}(\mathbb{T})$ such that $(\l,u_{\l})\in\mathscr{C}_{1}$. On the other hand, as
	$$\mathscr{C}_{1}\cap\mc{T} \subset \{(\s_{k},0) : k\in\N\},$$
	necessarily $u_{\l}\not\equiv 0$. Finally, if $u$ were constant, by the uniqueness result of Proposition \ref{Pr5.1}, we would have that $u\equiv -\l^{\frac{1}{p-1}}$. Then, in this case, we would have that $\mathscr{C}_{0}=\mathscr{C}_{1}$, which contradicts Lemma \ref{Lema7.1}. The proof is concluded.
\end{proof}

Finally, we will sharpen Theorem \ref{TP0} for the case $p=2$. In this case, we have the bifurcation direction given in Theorem \ref{LCRWw}, item (c). We will exploit this property in the following.

\begin{theorem}
	\label{TP}
	Suppose that $\tfrac{1}{2}\leq s<1$. The following statements hold regarding the equation
	\begin{equation}
		\label{1.40}
		(-\D)^{s}u = \l u + u^{2}, \quad x\in\mathbb{T}, \;\; u\in H^{2s}_{+}(\mathbb{T}).
	\end{equation}
	
	\begin{itemize} 
		\item[{\rm(a)}] Equation \eqref{1.40} admits at least one non-constant even solution $u_{\l}\in H^{2s}_{+}(\mathbb{T})$ for each $\l > 1$. 
		\item[{\rm(b)}] Equation \eqref{1.40}
		admits at least one non-constant strictly positive even solution $u_{\l}\in H^{2s}_{+}(\mathbb{T})$ for each $\l < -1$.
	\end{itemize}
\end{theorem}

\begin{proof}
	Let us start by proving item (a). From the proof of Theorem \ref{TP0}, we have that the connected component $\mathscr{C}_{1}$ is unbounded and $[\s_{1},+\infty)\subset\mc{P}_{\l}(\mathscr{C}_{1})$. If $\mathscr{C}_{1}\neq \mathscr{C}_{k}$ for every $k >1$, then the statement follows directly. Indeed, in this case, $(\s_{k},0)\notin \mathscr{C}_{1}$ for every $k>1$ and hence 
	\begin{equation}
		\label{1}
	\mathscr{C}_{1}\cap\mc{T}=\{(\s_1,0)\}.
	\end{equation}
	 Therefore, as $[\s_1,+\infty)\subset\mc{P}_{\l}(\mathscr{C}_{1})$, for each $\l>\s_{1}=1$, there exists $(\l,u_{\l})\in\mathscr{C}_{1}$ and by \eqref{1}, necessarily $u_{\l}\not\equiv 0$. By Lemma \ref{Lema7.1}, $u_{\l}$ cannot be constant. Therefore, for each $\l\in(1,+\infty)$, there exists at least one non-constant even solution $(\l,u_{\l})\in\mathscr{C}_{1}$.
	\par Otherwise, choose $k>1$ the minimal integer such that $\mathscr{C}_{1}=\mathscr{C}_{k}$. Then, if $k>2$, for each integer $n\in\{2,\dots,k-1\}$, $\mathscr{C}_{1}\neq \mathscr{C}_{n}$, and hence, since
	\begin{equation}
		\label{2}
	\mathscr{C}_{1}\cap \{(\l,0)\in \mc{T} \; : \; \l<\s_{k}\} = \{(\s_{1},0)\},
	\end{equation}
	there exists $(\s_{n},u_n)\in\mathscr{C}_{1}$ with $u_n\not\equiv 0$. Consequently, for each $\l\in (\s_1,\s_k)$, there exists at least a non-constant even solution $(\l,u_\l)\in\mathscr{C}_{1}$.
	\par By Theorem \ref{LCRWw}, item (c), the bifurcation on $(\s_k,0)$ is supercritical, that is, there exists $\r>0$ such that if $(\l,u)\in B_{\r}(\s_k,0)\cap \mf{S}$ and $u\not\equiv0$, then necessarily $\l>\s_{k}$. This implies that there exists $(\s_{k},u)\in \mathscr{C}_{1}$ with $u\not\equiv 0$. Indeed, if this is not the case, we would have that $\mathscr{C}_{1}\cap (\{\s_{k}\}\times H^{2s}_{+}(\mathbb{T}))=\{(\s_{k},0)\}$ and hence by the inclusion $(\s_{1},\s_{k}]\subset\mc{P}_{\l}(\mathscr{C}_{1})$ and the identity \eqref{2}, Theorem \ref{Th6.10} and Proposition \ref{LF4} yield, after taking a subsequence, a sequence $\{(\l_n,u_n)\}_{n\in\N}\subset\mathscr{C}_{1}$ such that $\l_{n}<\s_{k}$, $u_{n}\not\equiv 0$ for each $n\in\N$ and
	$$
	\lim_{n\to+\infty}\l_n = \s_k, \quad \lim_{n\to+\infty} u_n = 0 \quad \text{in} \;\; H^{2s}_{+}(\mathbb{T}).
	$$
	This contradicts the supercriticality of the bifurcation point $(\s_k,0)$. Therefore, the equation \eqref{1.40} admits at least one non-constant even solution for each $\l\in(\s_{1},\s_{k}]$. The same compactness and supercriticality argument applies at every $\s_j$, $j>1$, for which $\mathscr{C}_j=\mathscr{C}_1$. At the remaining spectral values, $(\s_j,0)\notin\mathscr{C}_1$. This proves existence for each $\l\in(\s_1,+\infty)$. Finally, item (b) follows from Proposition \ref{P5.3}, Lemma \ref{L7} and the strong maximum principle. 
\end{proof}

\section{Benjamin--Ono Equation}\label{S8}

The periodic Benjamin--Ono equation is given by
$$
\partial_{t}u+2u\partial_{x}u-\partial_{x}(\mc{H}\partial_{x}u)=0,  \;\; u=u(t,x) \quad t>0, \;\; x\in\mathbb{T},
$$
where $\mc{H}$ is the Hilbert transform given by
$$
\mc{H}u:=-i\sum_{n\in\Z}\sign(n)\widehat{u}(n)e^{inx}, \quad u\in L^{2}(\mathbb{T}).
$$
Note that $\mc{H}\partial_{x}\equiv (-\D)^{\frac{1}{2}}$. Therefore, in this case, $\mf{m}(n)=1$ for all $n\in\Z$ and the traveling wave equation \eqref{Ee} reduces to
\begin{equation}
	\label{rr}
(-\D)^{\frac{1}{2}}u = \l u + u^{2}, \quad u\in H^{1}_{+}(\mathbb{T}),
\end{equation}
where $c:=-\l$ is the velocity of the traveling wave. This equation can be solved analytically as the model is completely integrable. The aim of this section is to write all the solutions of \eqref{rr} in a closed manner following the works of Benjamin \cite{Be} and Amick--Toland \cite{AT}, and to plot the global bifurcation diagram of equation \eqref{rr}. We will see that the structure of this diagram coincides with the results obtained along this paper. 

Benjamin \cite{Be} provided some analytically closed formula for some solutions. Here, we mimic its proof to link this solutions to the connected components $\mathscr{C}_{k}$, $k\geq 1$. Let $u\in H^{1}_{\mathbf{+}}(\mathbb{T})$. Then, we have
\begin{equation*}
	u(x)=\sum_{n\in\Z}\widehat{u}(n)e^{inx}, \quad u^{2}(x)=\sum_{n\in\Z}\left(\sum_{k\in\Z}\widehat{u}(k)\widehat{u}(n-k)\right)e^{inx}.
\end{equation*}
Introducing this formulas in equation \eqref{rr}, we obtain
\begin{align*}
	\sum_{n\in\Z}(|n|-\l)\widehat{u}(n)e^{inx} = \sum_{n\in\Z}\left(\sum_{k\in\Z}\widehat{u}(k)\widehat{u}(n-k)\right)e^{inx}.
\end{align*}
Therefore, we deduce the identities
\begin{equation*}
	(|n|-\l)\widehat{u}(n) = \sum_{k\in\Z}\widehat{u}(k)\widehat{u}(n-k), \quad n\in\Z.
\end{equation*}
Now, if we try with the ansatz
$$
\widehat{u}(n)=\alpha e^{-\beta |n|}, \quad n\in\Z,
$$
for some $\alpha\neq 0$ and $\beta>0$, we deduce that
\begin{align*}
	(|n|-\l)e^{-\beta |n|} = \alpha \sum_{k\in\Z}e^{-\beta(|k|+|n-k|)}=\alpha e^{-\beta|n|}(|n|+\coth(\beta)), \quad n\in\Z.
\end{align*}
Consequently, 
$$
|n|-\l = \alpha(|n|+\coth(\beta)), \quad n\in\Z,
$$
and considering $|\l|>1$, we must have
$
\alpha = 1$ and $\beta(\l) = - \tanh^{-1}(\tfrac{1}{\l})
$.
Therefore, our solution can be expressed as
$$
u_{\l}^{+}(x) = \sum_{n\in\Z}e^{-\beta(\l)|n|+inx}=\frac{\sinh(\beta(\l))}{\cosh(\beta(\l))-\cos(x)}, \quad \l<-1.
$$
Note that the requirement $\l<-1$ is necessary for the convergence of the series. Up to translation by $\pi$, we can rewrite $u^+_{\l}$ as
$$
u^{+}_{\l}(x)=\frac{1}{-\l+\sqrt{\l^2-1}\cos(x)}, \quad x\in\mathbb{T}.
$$
As equation \eqref{rr} is invariant by translations and $\cos(x+\pi)=-\cos(x)$ is also an even function, we can also check that
$$
u^{-}_{\l}(x)=\frac{1}{-\l-\sqrt{\l^2-1}\cos(x)}, \quad x\in\mathbb{T},
$$
is also a solution of \eqref{rr}. Note that $u^{\pm}_{\l}(x)>0$ for each $x\in\mathbb{T}$. 
Now, we apply the affine operator $T: \mf{C}_{-} \to \mf{C}_{+}$ of Proposition \ref{P5.3} to $(\l,u_{\l}^{\pm})$, $\l<-1$, to obtain the new solutions
$$
(\l,u^{\pm}_{1,\l}):=(\l, u^{\pm}_{-\l}-\l) = \left(\l, \frac{1}{\l\pm\sqrt{\l^2-1}\cos(x)}-\l\right)\in \R\times H^{1}_{+}(\mathbb{T}), \quad \l>1.
$$
We can also apply the continuous injections $T_{k}:\mf{S}\to \mf{S}$ of Proposition \ref{Pr7.20} to $(\l,u^{\pm}_{1,\l})$, $\l>1$, to obtain the new family of solutions
$$
(\l,u^{\pm}_{k,\l}):= \left(\l, \frac{k^2}{\l\pm\sqrt{\l^2-k^2}\cos(kx)}-\l\right)\in \R\times H^{1}_{+}(\mathbb{T}), \quad \l>k.
$$
By the uniqueness result of Amick--Toland \cite{AT}, we deduce that for each $k\in\N$, the connected component $\mathscr{C}_{k}$ is given by
$$
\mathscr{C}_{k}=\left\{(\l,u^{\pm}_{k,\l})\in \R\times H^{1}_{+}(\mathbb{T}) \; : \; \l\geq k \right\}\subset \R_{\geq k}\times H^{1}_{+}(\mathbb{T}),
$$
and for $k=0$, by Proposition \ref{Le7.1}, the connected component $\mathscr{C}_{0}$ for $\l>0$, is given by
$$
\mathscr{C}_{0}\cap \left(\R_{\geq 0}\times H^{1}_{+}(\mathbb{T})\right)=\left\{(\l,-\l)\in \R\times H^{1}_{+}(\mathbb{T}) \; : \; \l\geq 0 \right\}.
$$
Therefore, we deduce that the connected components $\{\mathscr{C}_{k}\}_{k\geq 0}$ are mutually disjoint and for each $k\geq 1$,
$$
\mc{P}_{\l}(\mathscr{C}_{k})=[k,+\infty).
$$
Consequently, the schematic bifurcation diagram of equation \eqref{rr} is the one shown in Figure \ref{F3}. Recall that we are representing the value of the parameter $\l$ in abscissas versus the norm $\|u\|_{H^{2s}}$ or $-\|u\|_{H^{2s}}$. This differentiation is made in order to express the multiplicity of solutions.
Therefore, equation \eqref{rr} admits at least one non-constant even solution $u^{+}_{1,\l}$ for each $\l > 1$
and admits at least one non-constant strictly positive even solution $u^{+}_{\l}$ for each $\l < -1$. This is entirely consistent with the thesis of Theorem \ref{TP}.

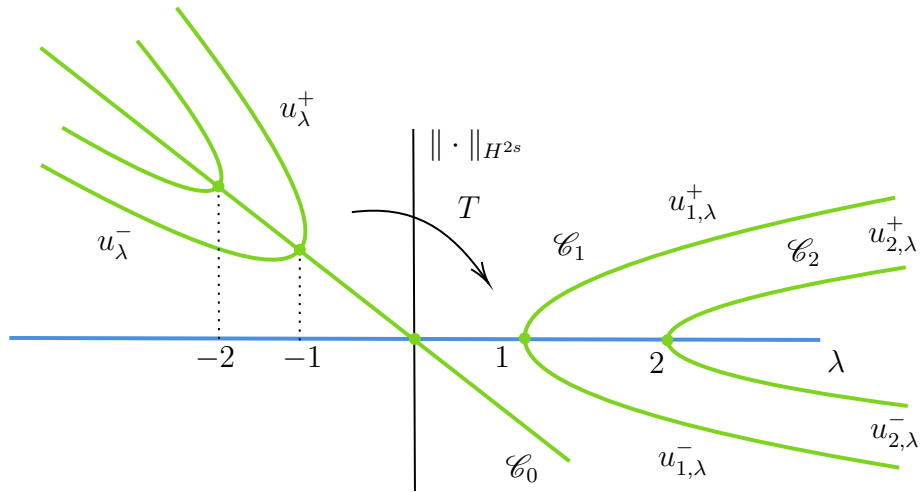
\begin{figure}[h!]

\tikzset{every picture/.style={line width=0.75pt}} 

\begin{tikzpicture}[x=0.75pt,y=0.75pt,yscale=-1,xscale=1]
	
	\draw    (280,73) -- (280.84,255.76) ;
	\draw [color={rgb, 255:red, 74; green, 144; blue, 226 }  ,draw opacity=1 ][line width=1.5]    (76.28,178.5) -- (485.28,179.5) ;
	\draw [color={rgb, 255:red, 126; green, 211; blue, 33 }  ,draw opacity=1 ][line width=1.5]    (92,32) -- (358.84,240.76) ;
	\draw  [color={rgb, 255:red, 126; green, 211; blue, 33 }  ,draw opacity=1 ][line width=1.5]  (523.06,107.68) .. controls (273.22,156.87) and (273.9,202.31) .. (525.11,243.97) ;
	\draw  [color={rgb, 255:red, 126; green, 211; blue, 33 }  ,draw opacity=1 ][line width=1.5]  (92.3,91.2) .. controls (243.2,175.06) and (265.99,148.6) .. (160.65,11.82) ;
	\draw    (249,115) .. controls (279.54,111.06) and (298.43,122.64) .. (317.14,149.75) ;
	\draw [shift={(318,151)}, rotate = 235.84] [color={rgb, 255:red, 0; green, 0; blue, 0 }  ][line width=0.75]    (10.93,-3.29) .. controls (6.95,-1.4) and (3.31,-0.3) .. (0,0) .. controls (3.31,0.3) and (6.95,1.4) .. (10.93,3.29)   ;
	\draw  [color={rgb, 255:red, 126; green, 211; blue, 33 }  ,draw opacity=1 ][fill={rgb, 255:red, 126; green, 211; blue, 33 }  ,fill opacity=1 ] (333.69,178.65) .. controls (333.69,177.27) and (334.81,176.15) .. (336.19,176.15) .. controls (337.57,176.15) and (338.69,177.27) .. (338.69,178.65) .. controls (338.69,180.03) and (337.57,181.15) .. (336.19,181.15) .. controls (334.81,181.15) and (333.69,180.03) .. (333.69,178.65) -- cycle ;
	\draw  [dash pattern={on 0.84pt off 2.51pt}]  (222.78,134) -- (222.73,180.14) ;
	\draw  [color={rgb, 255:red, 126; green, 211; blue, 33 }  ,draw opacity=1 ][fill={rgb, 255:red, 126; green, 211; blue, 33 }  ,fill opacity=1 ] (220.28,134) .. controls (220.28,132.62) and (221.4,131.5) .. (222.78,131.5) .. controls (224.16,131.5) and (225.28,132.62) .. (225.28,134) .. controls (225.28,135.38) and (224.16,136.5) .. (222.78,136.5) .. controls (221.4,136.5) and (220.28,135.38) .. (220.28,134) -- cycle ;
	\draw  [color={rgb, 255:red, 126; green, 211; blue, 33 }  ,draw opacity=1 ][fill={rgb, 255:red, 126; green, 211; blue, 33 }  ,fill opacity=1 ] (278.28,179) .. controls (278.28,177.62) and (279.4,176.5) .. (280.78,176.5) .. controls (282.16,176.5) and (283.28,177.62) .. (283.28,179) .. controls (283.28,180.38) and (282.16,181.5) .. (280.78,181.5) .. controls (279.4,181.5) and (278.28,180.38) .. (278.28,179) -- cycle ;
	\draw  [color={rgb, 255:red, 126; green, 211; blue, 33 }  ,draw opacity=1 ][line width=1.5]  (528.35,142.84) .. controls (367.79,168.59) and (368.14,191.92) .. (529.4,212.83) ;
	\draw  [color={rgb, 255:red, 126; green, 211; blue, 33 }  ,draw opacity=1 ][line width=1.5]  (102.96,72.3) .. controls (195.47,126.51) and (208.05,111.9) .. (140.71,28.46) ;
	\draw  [color={rgb, 255:red, 126; green, 211; blue, 33 }  ,draw opacity=1 ][fill={rgb, 255:red, 126; green, 211; blue, 33 }  ,fill opacity=1 ] (405.69,179.65) .. controls (405.69,178.27) and (406.81,177.15) .. (408.19,177.15) .. controls (409.57,177.15) and (410.69,178.27) .. (410.69,179.65) .. controls (410.69,181.03) and (409.57,182.15) .. (408.19,182.15) .. controls (406.81,182.15) and (405.69,181.03) .. (405.69,179.65) -- cycle ;
	\draw  [dash pattern={on 0.84pt off 2.51pt}]  (181.78,102) -- (182,179) ;
	\draw  [color={rgb, 255:red, 126; green, 211; blue, 33 }  ,draw opacity=1 ][fill={rgb, 255:red, 126; green, 211; blue, 33 }  ,fill opacity=1 ] (179.28,102) .. controls (179.28,100.62) and (180.4,99.5) .. (181.78,99.5) .. controls (183.16,99.5) and (184.28,100.62) .. (184.28,102) .. controls (184.28,103.38) and (183.16,104.5) .. (181.78,104.5) .. controls (180.4,104.5) and (179.28,103.38) .. (179.28,102) -- cycle ;
	
	\draw (286.03,71.4) node [anchor=north west][inner sep=0.75pt]    {$\| \cdot \| _{H^{2s}}$};
	\draw (487.28,182.9) node [anchor=north west][inner sep=0.75pt]    {$\lambda $};
	\draw (325.04,236.08) node [anchor=north west][inner sep=0.75pt]    {$\mathscr{C}_{0}$};
	\draw (349.04,124.08) node [anchor=north west][inner sep=0.75pt]    {$\mathscr{C}_{1}$};
	\draw (301,105.4) node [anchor=north west][inner sep=0.75pt]    {$T$};
	\draw (319.46,181.19) node [anchor=north west][inner sep=0.75pt]    {$1$};
	\draw (212.46,181.19) node [anchor=north west][inner sep=0.75pt]    {$-1$};
	\draw (467.04,127.08) node [anchor=north west][inner sep=0.75pt]    {$\mathscr{C}_{2}$};
	\draw (397.46,184.19) node [anchor=north west][inner sep=0.75pt]    {$2$};
	\draw (168.46,181.19) node [anchor=north west][inner sep=0.75pt]    {$-2$};
	\draw (407,99.4) node [anchor=north west][inner sep=0.75pt]    {$u_{1,\lambda }^{+}$};
	\draw (402,228.4) node [anchor=north west][inner sep=0.75pt]    {$u_{1,\lambda }^{-}$};
	\draw (508,115.4) node [anchor=north west][inner sep=0.75pt]    {$u_{2,\lambda }^{+}$};
	\draw (509,214.4) node [anchor=north west][inner sep=0.75pt]    {$u_{2,\lambda }^{-}$};
	\draw (211,52.4) node [anchor=north west][inner sep=0.75pt]    {$u_{\lambda }^{+}$};
	\draw (119,119.4) node [anchor=north west][inner sep=0.75pt]    {$u_{\lambda }^{-}$};

\end{tikzpicture}

	\caption{Bifurcation diagram of equation \eqref{rr}}
	\label{F3}
\end{figure}

\appendix

\section{Preliminaries on nonlinear spectral theory}
\label{A1}
In this subsection we collect some fundamental concepts about nonlinear spectral theory that will be used throughout the article. We start with some definitions. A \emph{Fredholm operator family} is any continuous map $\mathfrak{L}\in \mathcal{C}([a,b],\Phi_{0}(U,V))$.  For any given $\mathfrak{L}\in \mathcal{C}([a,b],\Phi_{0}(U,V))$, it is said that $\lambda\in[a,b]$ is a \emph{generalized eigenvalue} of $\mathfrak{L}$ if $\mathfrak{L}(\lambda)\notin GL(U,V)$, and the \emph{generalized spectrum} of $\mathfrak{L}$, $\Sigma(\mathfrak{L})$,  is defined through   	
\begin{equation*}
	\Sigma(\mathfrak{L}):=\{\lambda\in [a,b]: \mathfrak{L}(\lambda)\notin GL(U,V)\}.
\end{equation*}
One of the cornerstones of spectral theory is the concept of algebraic multiplicity. Algebraic multiplicity is classically defined for eigenvalues of compact operators in Banach spaces. Let $K: U\to U$ be a compact linear operator on the $\mathbb{K}$-Banach space $U$, $\mathbb{K}\in\{\mathbb{R},\mathbb{C}\}$. We set
\begin{equation}
	\mathfrak{L}(\lambda):=K-\lambda I_U, \quad \lambda\in J_{\lambda_0},
	\label{(8.1)}
\end{equation}
where $J_{\lambda_0}\subset \mathbb{K}$ is a neighbourhood of $\lambda_0$ and $\lambda_0$ is an eigenvalue of $K$. Then, the
classical algebraic multiplicity\index{multiplicity!classical} of
$\lambda_0$ as an eigenvalue of $K$ is defined through
\begin{equation}
	\mf{m}_{\alg}[K,\lambda_0] := \dim \bigcup_{\mu=1}^\infty
	N[(K-\lambda_0  I_U)^\mu]
	\label{(8.2)}.
\end{equation}
By Fredholm's theorem, for $\lambda_{0}\neq 0$, $\mf{L}:J_{\lambda_0}\to \Phi_{0}(U)$ where $\Phi_{0}(U)$ denotes the space of Fredholm operators of index zero $T:U\to U$. 
\par In 1988, J. Esquinas and J. L\'{o}pez-G\'{o}mez \cite{Es,ELG}, inspired by the work of Krasnoselski \cite{Kr}, Rabinowitz \cite{Ra} and Magnus \cite{Ma}, generalized the concept of algebraic multiplicity to every Fredholm operator family $\mf{L}:[a,b]\to \Phi_{0}(U,V)$, not necessarily of the form $\mathfrak{L}(\lambda)=K-\lambda I_U$, $U=V$, with $K$ compact. They denote it by $\chi[\mf{L},\lambda_{0}]$ and they proved that this concept is consistent with the classical algebraic multiplicity, that is, $\chi[K-\lambda I_{U},\lambda_{0}]=\mf{m}_{\alg}[K,\lambda_0]$ with $K$ compact and that it shares many properties of its classical counterpart. This theory was subsequently refined in the monograph \cite{LG01}.
We proceed to define the generalized algebraic multiplicity through the concept of algebraic eigenvalue going back to \cite{LG01}. 

\begin{definition}
	\label{de2.1}
	Let $\mathfrak{L}\in \mathcal{C}([a,b], \Phi_{0}(U,V))$ and $\kappa\in\mathbb{N}$. A generalized eigenvalue $\lambda_{0}\in\Sigma(\mathfrak{L})$ is said to be $\kappa$-algebraic if there exists $\varepsilon>0$ such that
	\begin{enumerate}
		\item[{\rm (a)}] $\mathfrak{L}(\lambda)\in GL(U,V)$ if $0<|\lambda-\lambda_0|<\varepsilon$;
		\item[{\rm (b)}] There exists $C>0$ such that
		\begin{equation}
			\label{2.2}
			\|\mathfrak{L}^{-1}(\lambda)\|_{\mc{L}}<\frac{C}{|\lambda-\lambda_{0}|^{\kappa}}\quad\hbox{if}\;\;
			0<|\lambda-\lambda_0|<\varepsilon;
		\end{equation}
		\item[{\rm (c)}] $\kappa$ is the minimal integer for which \eqref{2.2} holds.
	\end{enumerate}
\end{definition}
Throughout this paper, the set of $\kappa$-algebraic eigenvalues of $\mathfrak{L}$ is  denoted by $\Alg_\kappa(\mathfrak{L})$, and the set of \emph{algebraic eigenvalues} by
\[
\Alg(\mathfrak{L}):=\bigcup_{\kappa\in\mathbb{N}}\Alg_\kappa(\mathfrak{L}).
\]
We will construct an infinite dimensional analogue of the classical algebraic multiplicity for algebraic eigenvalues. It can be carried out through the theory of Esquinas and L\'{o}pez-G\'{o}mez
\cite{ELG},  where the following pivotal concept, generalizing the transversality condition of
Crandall and Rabinowitz \cite{CR},  was introduced. Throughout this paper, we set
$$\mathfrak{L}_{j}:=\frac{1}{j!}\mathfrak{L}^{(j)}(\lambda_{0}), \quad 0\leq j\leq r,$$
should these derivatives exist.

\begin{definition}
	\label{de2.3}
	Let $\mathfrak{L}\in \mathcal{C}^{r}([a,b],\Phi_{0}(U,V))$ and $1\leq \kappa \leq r$. Then, a given $\lambda_{0}\in \Sigma(\mathfrak{L})$ is said to be a $\kappa$-transversal eigenvalue of $\mathfrak{L}$ if
	\begin{equation*}
		\bigoplus_{j=1}^{\kappa}\mathfrak{L}_{j}\left(\bigcap_{i=0}^{j-1}N(\mathfrak{L}_{i})\right)
		\oplus R(\mathfrak{L}_{0})=V\;\; \hbox{with}\;\; \mathfrak{L}_{\kappa}\left(\bigcap_{i=0}^{\kappa-1}N(\mathfrak{L}_{i})\right)\neq \{0\}.
	\end{equation*}
\end{definition}

For these eigenvalues, the following generalized concept of algebraic multiplicity can be introduced:
\begin{equation}
	\label{ii.3}
	\chi[\mathfrak{L}, \lambda_{0}] :=\sum_{j=1}^{\kappa}j\cdot \dim \mathfrak{L}_{j}\left(\bigcap_{i=0}^{j-1}N[\mathfrak{L}_{i}]\right).
\end{equation}
In particular, when $N[\mf{L}_0]=\mathrm{span}\{\v\}$ for some $\v\in U$ such that $\mf{L}_1\v\notin R[\mf{L}_0]$, then
\begin{equation}
	\label{ii.4}
	\mf{L}_1(N[\mf{L}_0])\oplus R[\mf{L}_0]=V
\end{equation}
and hence, $\l_0$ is a 1-transversal eigenvalue of $\mf{L}(\l)$ with $\chi[\mf{L},\l_0]=1$. The transversality condition \eqref{ii.4} goes back to Crandall and Rabinowitz \cite{CR}. More generally, under condition \eqref{ii.4}, it holds that
\begin{equation}
	\label{CTCR}
	\chi[\mf{L},\l_0]=\dim N[\mf{L}_0].
\end{equation}
According to Theorems 4.3.2 and 5.3.3 of \cite{LG01}, for every $\mathfrak{L}\in \mathcal{C}^{r}([a,b], \Phi_{0}(U,V))$, $\kappa\in\{1,2,...,r\}$ and $\lambda_{0}\in \Alg_{\kappa}(\mathfrak{L})$, there exists a polynomial $\Phi: [a,b]\to \mathcal{L}(U)$ with $\Phi(\lambda_{0})=I_{U}$ such that $\lambda_{0}$ is a $\kappa$-transversal eigenvalue of the path $\mathfrak{L}^{\Phi}(\l):=\mathfrak{L}(\l)\circ\Phi(\l)$
and $\chi[\mathfrak{L}^{\Phi},\lambda_{0}]$ is independent of the curve of \emph{trasversalizing local isomorphisms} $\Phi$ chosen to transversalize $\mathfrak{L}$ at $\lambda_0$. Therefore, the following concept of multiplicity
is consistent
\begin{equation}
	\label{ii.6}
	\chi[\mf{L},\l_0]:= \chi[\mathfrak{L}^{\Phi},\lambda_{0}],
\end{equation}
and it can be easily extended by setting
$\chi[\mathfrak{L},\lambda_0] =0$ if $\lambda_0\notin\Sigma(\mathfrak{L})$ and
$\chi[\mathfrak{L},\lambda_0] =+\infty$ if $\lambda_0\in \Sigma(\mathfrak{L})
\setminus \Alg(\mathfrak{L})$ and $r=+\infty$. Thus, $\chi[\mathfrak{L},\lambda]$ is well defined for all  $\lambda\in [a,b]$ of any smooth path $\mathfrak{L}\in \mathcal{C}^{\infty}([a,b],\Phi_{0}(U,V))$.

\section{Preliminaries on Bifurcation Theory}\label{Se5}
In this section we review some abstract results of bifurcation theory that will be used through the analysis of equation \eqref{e}.

\subsection{The Crandall--Rabinowitz theorem}\label{SCRT} Here we recall the celebrated Crandall--Rabinowitz bifurcation theorem and some of its consequences. It was stated and proved in \cite{CR, CRex}. Let $(U,V)$ be a pair of real Banach spaces and consider a map $\mf{F}:\R\times U\to V$ of class $\mc{C}^{r}$, $r\geq 1$, satisfying the following assumptions:
\begin{enumerate}
	\item[(F1)] $\mathfrak{F}(\lambda,0)=0$ for every $\lambda\in\mathbb{R}$;
	\item[(F2)] $\mf{L}(\l):=\partial_{u}\mathfrak{F}(\lambda,0)\in\Phi_{0}(U,V)$ for each $\lambda\in\mathbb{R}$;
	\item[(F3)] $N[\mf{L}(\l_{0})]=\spann\{\varphi\}$ for some $\varphi\in U\backslash\{0\}$ and $\l_{0}\in \R$.
\end{enumerate}
The theorem reads as follows.
\begin{theorem}
	\label{Cr-Rb} 
	Let $r\geq 2$, $\mathfrak{F}\in\mc{C}^{r}(\R\times U,V)$ be a map satisfying conditions {\rm (F1)--(F3)} and $\l_{0}\in\Sigma(\mf{L})$ be a $1$-transversal eigenvalue of $\mf{L}$, that is, 
	$$\mf{L}_{1}(N[\mf{L}_{0}])\oplus R[\mf{L}_{0}]=V.$$
	Let $Y\subset U$ be a subspace	such that
	$N[\mathfrak{L}_0] \oplus Y = U$.
	Then, there exist $\e>0$ and two maps of class $\mc{C}^{r-1}$,
	$$
	\L: (-\e,\e) \longrightarrow \R, \qquad \G: (-\e,\e)\longrightarrow Y,
	$$
	such that $\L(0)=\l_0$, $\G(0)=0$,
	and for each $s\in(-\e,\e)$,
	\begin{equation}
		\mathfrak{F}(\L(s),u(s))=0, \quad u(s):= s(\v+\G(s)).
		\label{(2.2.1)}
	\end{equation}
	Moreover,  $\r >0$ exists such that if $\mathfrak{F}(\l,u)=0$ and
	$(\l,u)\in B_\r(\l_0,0)$, then either $u = 0$ or
	$(\l,u)=(\L(s),u(s))$ for some $s\in(-\e,\e)$.
\end{theorem}

	If $r=1$, the same conclusion holds provided that
	$\partial^{2}_{\l u}\mathfrak F$ exists and is continuous.

We can also compute the bifurcation direction of the curve $s\mapsto (\l(s),u(s))$ in terms of the derivatives of $\mf{F}$. Let $\varphi^{\ast}\in V^{\ast}$ such that
$$
R[\mf{L}(\l_0)]=\{v\in V \; : \; \langle v, \varphi^{\ast}\rangle =0 \},
$$
where $\langle \cdot, \cdot \rangle : V\times V^{\ast}\to\R$ denotes the duality pairing on $V$. The following result can be found, for instance, in \cite{Ki}.

\begin{theorem}
	\label{BD}
	Under the assumptions of Theorem \ref{Cr-Rb}, the following identity holds:
	\begin{align*}
		\dot\L(0)=-\frac{1}{2}\frac{\langle \partial^{2}_{uu}\mf{F}(\l_{0},0)[\varphi,\varphi],\varphi^{\ast}\rangle}{\langle \partial^{2}_{\l u}\mf{F}(\l_{0},0)[\varphi],\varphi^{\ast}\rangle}.
	\end{align*}
	Moreover, if $r\geq 3$ and $\partial^{2}_{uu}\mf{F}(\l_{0},0)[\varphi,\varphi]\in R[\mf{L}_{0}]$, then
	\begin{equation}
		\ddot{\L}(0)=-\frac{1}{3}\frac{\langle\partial^{3}_{uuu}\mf{F}(\l_{0},0)[\varphi,\varphi,\varphi],\varphi^{\ast}\rangle+3\langle \partial^{2}_{uu}\mf{F}(\l_0,0)[\varphi,\phi],\varphi^{\ast}\rangle}{\langle \partial^{2}_{\l u}\mf{F}(\l_{0},0)[\varphi],\varphi^{\ast}\rangle},
	\end{equation}
	where $\phi\in U$ is any vector satisfying $\partial^{2}_{uu}\mf{F}(\l_0,0)[\varphi,\varphi]+\partial_{u}\mf{F}(\l_0,0)[\phi]=0$.
\end{theorem}

\subsection{Global bifurcation theory}

In this subsection we estate the global alternative theorem for the bifurcation of nonlinear Fredholm operators of index zero given by López-Gómez and Mora-Corral in \cite{LG,LGMC,LGMC1} and sharpened in \cite{LGSM} in the light of the Fitzpatrick, Pejsachowicz and Rabier topological degree \cite{FPRa, FPRb, PR}, a generalization of the Leray--Schauder degree to Fredholm operators of index zero. Throughout this section, we consider a function of class $\mc{C}^{1}$, $\mathfrak{F}:\mathbb{R}\times U\to V$, such that
\begin{enumerate}
	\item[(F1)] $\mf{F}$ is orientable in the sense of Fitzpatrick, Pejsachowicz and Rabier, see \cite{FPRa}.
	\item[(F2)] $\mathfrak{F}(\lambda,0)=0$ for all $\lambda\in\mathbb{R}$.
	\item[(F3)] $\partial_{u}\mathfrak{F}(\lambda,u)\in\Phi_{0}(U,V)$ for every $\lambda\in\mathbb{R}$ and $u\in U$.
	\item[(F4)] $\mathfrak{F}$ is proper on bounded and closed subsets of $\mathbb{R}\times U$.
	\item[(F5)] $\Sigma(\mf{L})$ is a discrete subset of $\R$.
\end{enumerate}
The \textit{trivial branch} is the subset 
$$\mc{T}:=\{(\lambda,0): \l\in \R\}\subset \R\times U.$$
The \textit{set of non-trivial solutions} is defined by
\begin{equation*}
	\mf{S}=\left[ \mf{F}^{-1}(0)\backslash \mc{T}\right]\cup \{(\l,0):\;\l\in \Sigma(\mf{L})\}.
\end{equation*}
The global alternative theorem reads as follows:
\begin{theorem}[\textbf{Global alternative}]
	\label{TGB} Let $\mf{F}\in\mc{C}^{1}(\R\times U, V)$ be a map satisfying {\rm(F1)--(F5)}  and $\mf{C}$ be a connected component of the set of non-trivial solutions $\mf{S}$ such that $(\l_{0},0)\in\mf{C}$. If the linearization $\mf{L}(\l)$, $\l\in \R$, is analytic and the oddity condition
	$$\chi[\mf{L},\l_0]\in 2\N-1,$$
	holds for some $\l_0\in\Sigma(\mf{L})$, then one of the following non-excluding alternatives occur:
	\begin{enumerate}
		\item $\mf{C}$ is unbounded.
		\item There exists $\l_{1}\in \Sigma(\mf{L})$, $\l_{1}\neq\l_{0}$, such that $(\l_{1},0)\in \mf{C}$.
	\end{enumerate}
\end{theorem}


\begin{thebibliography}{10}
	
\bibitem{AV} N. Abatangelo and E. Valdinoci, \emph{Getting Acquainted with the fractional Laplacian}, Springer International Publishing, Cham, (2019).
	
\bibitem{A1} V. Ambrosio, Infinitely many periodic solutions for a fractional problem under perturbation, \emph{J. Elliptic Parabol. Equ.}, \textbf{2} (1-2), (2016), 105--117.

\bibitem{A2} V. Ambrosio, Periodic solutions for a superlinear fractional problem without the Ambrosetti- Rabinowitz condition, \emph{Discrete Contin. Dyn. Syst.}, \textbf{37} (5), (2017), 2265--2284.

\bibitem{A3} V. Ambrosio, Periodic solutions for the non-local operator $(-\D + m^{2})^{s} - m^{2s}$ with $m \geq 0$, \emph{Topol. Methods Nonlinear Anal.}, \textbf{49} (1), (2017), 75--104.

\bibitem{A4} V. Ambrosio and G. Molica Bisci, Periodic solutions for nonlocal fractional equations, \emph{Commun. Pure Appl. Anal.}, \textbf{16} (1), (2017), 331--344.
	
\bibitem{AT} C. J. Amick and J. F. Toland, Uniqueness and related analytic properties for the Benjamin--Ono equation --a nonlinear Neumann problem in the plane, \emph{Acta Math.}, \textbf{167}, (1991), 107--126.
	
\bibitem{BGMQ} B. Barrios, J. García-Melián and A. Quaas, Periodic solutions for the one-dimensional fractional Laplacian, \emph{J. Diff. Equations}, \textbf{267} (9), (2019), 5258--5289.

\bibitem{Be} T. B. Benjamin, Internal waves of permanent form in fluids of great depth, \emph{J. Fluid Mech.}, \textbf{29} (3), (1967), 559--592.
	
\bibitem{B} M. S. Berger, \emph{Nonlinearity and Functional Analysis}, Lectures on Nonlinear Problems in Mathematical Analysis, Academic Press, Inc., 1977.

\bibitem{BD} G. Bruell and R. N. Dhara, Waves of maximal height for a class of nonlocal equations with homogeneous symbols, \emph{Preprint}, (2018), \texttt{doi.org/10.48550/arXiv.1810.00248}.

\bibitem{CCM} X. Cabré, G. Csató and A. Mas, Periodic solutions to integro-differential equations: variational formulation, symmetry, and regularity, \emph{Preprint}, (2024), \texttt{doi.org/10.48550/arXiv.2404.06462}.

\bibitem{CS1} L. Caffarelli, L. Silvestre, An extension problem related to the fractional Laplacian.
\emph{Comm. Partial Differential Equations}, \textbf{32} (2007), 1245-1260.

\bibitem{CS} L. Caffarelli and L. Silvestre, Regularity theory for fully nonlinear integro-differential equations, \emph{Comm. Pure Appl. Math.}, \textbf{62} (5), (2009), 597--638.

\bibitem{CR} M. G. Crandall and P. H. Rabinowitz, Bifurcation from simple eigenvalues,
\emph{J. Funct. Anal.} \textbf{8} (1971), 321--340.

\bibitem{CRex}  M. G. Crandall and P. H. Rabinowitz, Bifurcation, perturbation from simple eigenvalues
and linearized stability, \emph{Arch. Rat. Mech. Anal.} \textbf{52} (1973), 161--180.

\bibitem{DPGW} A. DelaTorre, M. del Pino, M. a. d. M. González, and J. Wei, Delaunay-type singular solutions
for the fractional Yamabe problem, \emph{Math. Ann.}, \textbf{369} (1-2), (2017), 597--626.

\bibitem{DV} E. Di Nezza, G. Palatucci and E. Valdinoci, Hitchhiker’s guide to the fractional Sobolev spaces, \emph{Bull. Sci. Math.}, \textbf{136} (5), (2012), 521--573.

\bibitem{Es} J. Esquinas, Optimal multiplicity in local bifurcation theory, II: General case, \emph{J. Diff. Eqns.} \textbf{75} (1988), 206--215.

\bibitem{ELG} J. Esquinas and J. L\'opez-G\'omez, Optimal multiplicity in local bifurcation theory, I: Generalized generic eigenvalues, \emph{J. Diff. Eqns.} \textbf{71} (1988), 72--92.

\bibitem{FQ} P. Felmer and A. Quaas, Fundamental solutions and Liouville type theorems for nonlinear integral operators, \emph{Adv. Math.}, \textbf{226} (3), (2011), 2712--2738.

\bibitem{FPRa} P. M. Fitzpatrick, J. Pejsachowicz and P. J. Rabier, Orientability of Fredholm families and topological degree for orientable nonlinear Fredholm mappings, \emph{J. Functional Analysis} \textbf{124} (1994), 1--39.

\bibitem{FPRb} P. M. Fitzpatrick, J. Pejsachowicz and P. J. Rabier, The Degree of Proper $\mathcal{C}^{2}$ Fredholm mappings I, \emph{J. Reine Angew. Math.} \textbf{427} (1992), 1--33

\bibitem{GS} B. Gidas and J. Spruck, A priori bounds for positive solutions of nonlinear elliptic equations, \emph{Comm. Partial Differential Equations}, \textbf{6} (8), (1981), 883--901.

\bibitem{GGS} I. Gohberg, S. Goldberg and M. A. Kaashoek, \emph{Basic Classes of Linear Operators}, Springer Basel, (2003).

\bibitem{Ki} H. Kielhöfer, \emph{Bifurcation theory},Appl. Math. Sci., \textbf{156}, Springer, New York, (2012).

\bibitem{Kr} M. A. Krasnoselskij, \emph{Topological Methods in the Theory of Nonlinear Integral Equations}, Pergamon Press, New York, (1964).

\bibitem{LW} R. Liang and Y. Wang, Gibbs measure for the focusing fractional NLS on the Torus, \emph{SIAM Journal on Mathematical Analysis} \textbf{54} (6), (2022), 6096-6118.

\bibitem{LG01} J. L\'opez-G\'omez, \emph{Spectral theory and nonlinear functional analysis}, CRC Press, Chapman and Hall RNM \textbf{426}, Boca Raton, (2001).

\bibitem{LG} J. L\'{o}pez-G\'{o}mez, Global bifurcation for Fredholm operators, \emph{Rend. Istit. Mat. Univ. Trieste} \textbf{48} (2016), 539--564. 

\bibitem{LGMC} J. L\'opez-G\'omez and C. Mora-Corral, Counting zeroes of $\mathcal{C}^1$-Fredholm maps of index zero, \emph{Bull. London Math. Soc.} \textbf{37} (2005), 778--792.

\bibitem{LGMC1} J. L\'opez-G\'omez and C. Mora-Corral, Counting solutions of nonlinear abstract equations, \emph{Top. Meth. Nonl. Anal.} \textbf{24}, (2004), 307--335.

\bibitem{LGSM} J. L\'opez-G\'omez and J. C. Sampedro, Bifurcation theory for Fredholm operators, \emph{J. Differential Equations} \textbf{404}, (2024), 182--250.


\bibitem{Ma} R. J. Magnus,  A generalization of multiplicity and the problem of bifurcation,
\emph{Proc. Lond. Math. Soc.} \textbf{32} (1976), 251--278.

\bibitem{PR} J. Pejsachowicz and P. J. Rabier, Degree theory for $\mathcal{C}^{1}$ Fredholm mappings of index 0, \emph{J. Anal. Math.} \textbf{76} (1998), 289.

\bibitem{Ra}  P. H. Rabinowitz, Some global results for nonlinear eigenvalue problems, \emph{J.
	Funct. Anal.} \textbf{7} (1971), 487--513.

\bibitem{RO} X. Ros-Oton, Nonlocal elliptic equations in bounded domains: a survey, \emph{Publ. Mat.}, \textbf{60} (1),
(2016), 3--26.

\bibitem{RO2} X. Ros-Oton and J. Serra, The Dirichlet problem for the fractional Laplacian: regularity up to the boundary, \emph{J. Math. Pures Appl.}, (9) \textbf{101} (3), (2014), 275--302.


\bibitem{SV} R. Servadei and E. Valdinoci, Mountain Pass solutions for non-local elliptic operators, \emph{J. Math.
Anal. Appl.}, \textbf{389}, (2012), 887--898.

\bibitem{SV2} R. Servadei and E. Valdinoci, Variational methods for non-local operators of elliptic type,
\emph{Discrete Contin. Dyn. Syst.}, \textbf{33}, (2013), 2105--2137.

\bibitem{Si} L. Silvestre, Regularity of the obstacle problem for a fractional power of the Laplace operator, \emph{Comm. Pure Appl. Math.}, \textbf{60} (1), (2007), 67--112.

\bibitem{Y} K. Yosida, \emph{Functional analysis}, Grundlehren der Mathematischen Wissenschaften, \textbf{123},
Springer-Verlag, Berlin-New York, (1980).

\end{thebibliography}
\end{document}